\documentclass{amsart}
\usepackage[british]{babel}
\usepackage{amsmath,amssymb,amsthm,amscd}
\usepackage{graphicx,subfigure}
\usepackage{xcolor}
\usepackage{pb-diagram} 
\usepackage{tabularx}
\usepackage{multirow}
\usepackage{longtable}
\usepackage{hyperref}
\usepackage[foot]{amsaddr}

\usepackage{comment}

\newtheorem{theorem}{Theorem}[section]

\newtheorem{lemma}[theorem]{Lemma}

\theoremstyle{remark} 
\newtheorem{remark}[theorem]{Remark}

\numberwithin{equation}{section}

\newcommand{\al}{\alpha}

\newcommand{\de}{\delta}
\newcommand{\ep}{\varepsilon}

\newcommand{\om}{\omega}

\newcommand{\te}{\theta}
\newcommand{\vp}{\varphi}

\newcommand{\cA}{{\mathcal A}}
\newcommand{\cB}{{\mathcal B}}

\newcommand{\cD}{{\mathcal D}}

\newcommand{\cI}{{\mathcal I}}

\newcommand{\cO}{{\mathcal O}}

\newcommand{\cX}{{\mathcal X}}

\newcommand{\RR}{{\mathbb R}}
\newcommand{\CC}{{\mathbb C}}

\newcommand{\TT}{{\mathbb T}}
\newcommand{\ZZ}{{\mathbb Z}}

\newcommand{\abs}[1]{|{#1}|}
\newcommand{\Abs}[1]{\left|{#1}\right|}
\newcommand{\norm}[1]{\|{#1}\|}
\newcommand{\Norm}[1]{\left\|{#1}\right\|}
\newcommand{\aver}[1]{\langle{#1}\rangle}
\newcommand{\Lie}[1]{\mathfrak{L}_{#1}}
\newcommand{\Loper}[1]{\cX_{#1}}
\newcommand{\R}[1]{\mathfrak{R}_{#1}}
\newcommand{\Dioph}[2]{\cD_{#1,#2}}

\def\to{\rightarrow}

\def\epsilon{\varepsilon}

\def\ii{\mathrm{i}}
\def\ee{\mathrm{e}}
\def\re{\mathrm{Re}}
\def\im{\mathrm{Im}}

\def\dist{\mathrm{dist}}

\def\dif{ {\mbox{\rm d}} }
\def\Dif{ {\mbox{\rm D}} }
\def\Tan{ {\mbox{\rm T}} }
\def\pd{ \partial }

\def\sform{\mbox{\boldmath $\omega$}}
\def\aform{\mbox{\boldmath $\alpha$}}
\def\proj{\mbox{\boldmath $\pi$}}
\def\gform{\mbox{\boldmath $g$}}

\def\J{\mbox{\boldmath $J$}}
\def\I{\mbox{\boldmath $I$}}
\def\D{{Dom}}
\def\mani{{\mathcal{M}}}
\def\Anal{{\cA}}
\def\B{{\mathcal{B}}}
\def\U{{\mathcal{U}}}
\def\V{{\mathcal{V}}}
\def\torus{{\mathcal{K}}}
\def\K{{\tilde{K}}}

\def\LieO{\mathfrak{L}}     
\def\H{h}                   
\def\Tdown{H}               
\def\homega{\hat \omega}    
\def\Elag{\Omega_L} 
\def\Esym{E_{\mathrm{sym}}} 
\def\Ered{E_{\mathrm{red}}} 
\def\Elin{E_{\mathrm{lin}}} 
\def\xio{\xi^{\omega}}      

\def\DLieo{\Delta \LieO}
\def\DeltaK{\Delta K}
\def\Deltao{\Delta \omega}
\def\DeltaL{\Delta L}
\def\DeltaLT{\Delta {L^\top}}
\def\DeltaG{\Delta G}
\def\DeltaJ{\Delta J}
\def\DeltaJT{\Delta J^\top}
\def\DeltaTH{\Delta T_\H} 
\def\DeltaTHT{\Delta T^\top_\H} 
\def\DeltaGL{\Delta {G_L}}
\def\DeltaOmega{\Delta {\Omega}} 

\def\DeltatOmegaL{\Delta {\tilde \Omega_L}}
\def\DeltaA{\Delta A}
\def\DeltaB{\Delta B}
\def\DeltaNO{\Delta N^0}
\def\DeltaNOT{\Delta N^{0\top}}
\def\DeltaTildeN{\Delta \tilde N} 
\def\DeltaTildeNT{\Delta \tilde N^\top} 
\def\DeltaN{\Delta N}
\def\DeltaNT{\Delta N^\top}

\def\DeltaT{\Delta T}
\def\DeltaTc{\Delta T_c}
\def\DeltaTOI{\Delta \aver{T}^{-1}}
\def\DeltaTcOI{\Delta \aver{T_c}^{-1}}
\def\DeltaLieK{\Delta \LieO K}
\def\LieDeltaK{\LieO \Delta K}
\def\DeltaLieL{\Delta \LieO L}
\def\DeltaLieLT{\Delta \LieO L^\top}
\def\DeltaLieG{\Delta \LieO G}
\def\DeltaLieGL{\Delta \LieO {G_L}}
\def\DeltaLietOmega{\Delta \LieO {\tilde \Omega}}
\def\DeltaLietOmegaL{\Delta \LieO {\tilde \Omega_L}}
\def\DeltaLieB{\Delta \LieO B}
\def\DeltaLieA{\Delta \LieO A}
\def\DeltaLieJ{\Delta \LieO J}
\def\DeltaLieNO{\Delta \LieO N^0}
\def\DeltaLieN{\Delta \LieO N}

\def\Deltao{\Delta \omega}

\def\cteOmega{c_{\Omega,0}}
\def\cteDOmega{c_{\Omega,1}}
\def\ctetOmega{c_{\tilde\Omega,0}}
\def\cteDtOmega{c_{\tilde\Omega,1}}
\def\cteDDtOmega{c_{\tilde\Omega,2}}
\def\cteJ{c_{J,0}}
\def\cteDJ{c_{J,1}}
\def\cteDDJ{c_{J,2}}

\def\cteJT{c_{J^{\top}\hspace{-0.5mm},0}}
\def\cteJInv{c_{J^{-1}\hspace{-0.5mm},0}} 
\def\cteJInvT{c_{J^{-\top}\hspace{-0.5mm},0}} 
\def\cteDJT{c_{J^{\top}\hspace{-0.5mm},1}}
\def\cteG{c_{G,0}}
\def\cteDG{c_{G,1}}
\def\cteDDG{c_{G,2}}
\def\cteDH{c_{\H,1}}

\def\cteDc{c_{c,1}}
\def\cteDDc{c_{c,2}}

\def\cteXH{c_{Z_\H,0}}
\def\cteDXH{c_{Z_\H,1}}
\def\cteDDXH{c_{Z_\H,2}}
\def\cteDXHT{c_{Z_\H^\top\hspace{-0.5mm},1}}

\def\cteTH{c_{T_\H,0}} 
\def\cteTHT{c_{T^\top_\H,0}} 
\def\cteDTH{c_{T_\H\hspace{-0.5mm},1}} 
\def\cteDTHT{c_{T_\H^\top\hspace{-0.5mm},1}}

\def\cteDXp{c_{X_p,1}}
\def\cteDXpT{c_{X_p^\top\hspace{-0.5mm},1}}
\def\cteDDXp{c_{X_p,2}}
\def\cteDDXpT{c_{X_p^\top\hspace{-0.5mm},2}}
\def\cauxT{\nu} 

\def\sigmaDK{\sigma_{L}}
\def\sigmaDKT{\sigma_{L^\top\hspace{-0.5mm}}}
\def\sigmaB{\sigma_B}
\def\sigmaT{\sigma_T}
\def\sigmaTc{\sigma_{T_c}}
\def\sigmao{\sigma_\omega}

\def\CE{C_{E}}

\def\CEo{C_{E^\omega}}
\def\CEc{C_{E_c}}
\def\Clag{C_{\Omega_L}}
\def\Csym{C_{\mathrm{sym}}}
\def\Cred{C_{\mathrm{red}}}
\def\Creduu{C_{\mathrm{red}}^{1,1}}
\def\Credud{C_{\mathrm{red}}^{1,2}}
\def\Creddu{C_{\mathrm{red}}^{2,1}}
\def\Creddd{C_{\mathrm{red}}^{2,2}}
\def\Clin{C_{\mathrm{lin}}}
\def\Clino{C_{\mathrm{lin}}^\omega}

\def\Comega{{C_\omega}}

\def\COmegaL{C_{\Omega_{L}}}
\def\CT{C_T}

\def\CA{C_A}
\def\CL{C_L}

\def\CLT{C_{L^\top}}
\def\CNO{C_{N^0}}
\def\CNOT{C_{N^{0,\top}}}
\def\CN{C_N}
\def\CNT{C_{N^\top}}
\def\CLoperL{C_{\Loper{L}}}
\def\CLoperLT{C_{\Loper{L}^\top}}

\def\CxiNO{C_{\xi^N_0}}
\def\Cxio{C_{\xi^\omega}}
\def\CxiN{C_{\xi^N}}
\def\CxiL{C_{\xi^L}}
\def\Cxi{C_{\xi}}
\def\CGL{C_{G_L}}
\def\CtOmegaL{C_{\tilde \Omega_L}}

\def\CTH{C_{T_\H}}
\def\CTHT{C_{T_\H^\top}}
\def\CTE{C_{T_E}}
\def\CET{C_{E_T}}
\def\CTET{C_{T_E^\top}}
\def\CTildeN{C_{\tilde N}} 
\def\CTildeNT{C_{\tilde N^\top}} 
\def\CELieB{C_{E_{\Lie{}B}}}
\def\CELieBT{C_{E_{\Lie{}B}^\top }}
\def\CELietildeN{C_{E_{\Lie{}\tilde N}}}
\def\CELieN{C_{E_{\Lie{}N}}} 
\def\CELietildeNT{C_{E_{\Lie{}\tilde N}^\top }}
\def\CELieA{C_{E_{\Lie{}A}}}

\def\CDeltaK{C_{\DeltaK}}
\def\CDeltao{C_{\Deltao}}
\def\CDeltaL{C_{\DeltaL}}
\def\CDeltaLT{C_{\DeltaLT}}
\def\CDeltaA{C_{\DeltaA}}
\def\CDeltaB{C_{\DeltaB}}
\def\CDeltaNO{C_{\DeltaNO}}
\def\CDeltaNOT{C_{\DeltaNOT}}
\def\CDeltaTildeN{C_{\DeltaTildeN}} 
\def\CDeltaTildeNT{C_{\DeltaTildeNT}} 
\def\CDeltaN{C_{\DeltaN}}
\def\CDeltaNT{C_{\DeltaNT}}
\def\CDeltaG{C_{\DeltaG}}
\def\CDeltaGL{C_{\DeltaGL}}
\def\CDeltaOmega{C_{\DeltaOmega}} 

\def\CDeltatOmegaL{C_{\DeltatOmegaL}}

\def\CDeltaT{C_{\DeltaT}}
\def\CDeltaTc{C_{\DeltaTc}}
\def\CDeltaTOI{C_{\DeltaTOI}}
\def\CDeltaTcOI{C_{\DeltaTcOI}}
\def\CDeltaI{C_{\Delta,1}}
\def\CDeltaII{C_{\Delta,2}}
\def\CDeltaIII{C_{\Delta,3}}
\def\CDeltatot{C_{\Delta}}
\def\CDeltaJ{C_{\DeltaJ}}
\def\CDeltaJT{C_{\DeltaJT}}
\def\CDeltaTH{C_{\DeltaTH}}
\def\CDeltaTHT{C_{\DeltaTHT}}

\def\CLieOmega{C_{\LieO \Omega}}

\def\CLieK{C_{\LieO K}}
\def\CLieL{C_{\LieO L}}
\def\CLieLT{C_{\LieO L^\top}}
\def\CLieB{C_{\LieO B}}
\def\CLieA{C_{\LieO A}}
\def\CLieNO{C_{\LieO N^0}}
\def\CLieN{C_{\LieO N}}
\def\CLieJ{C_{\LieO J}}
\def\CLieJT{C_{\LieO J^\top}}
\def\CLieG{C_{\LieO G}}
\def\CLietOmega{C_{\LieO \tilde \Omega}}
\def\CLiexi{C_{\LieO \xi}}
\def\CLiexiL{C_{\LieO \xi^L}}
\def\CLiexiN{C_{\LieO \xi^N}}
\def\CLieGL{C_{\LieO G_L}}
\def\CLieOmegaL{C_{\LieO \Omega_L}}
\def\CLietOmegaL{C_{\LieO \tilde \Omega_L}}

\def\CLieDeltaK{C_{\LieDeltaK}}
\def\CDeltaLieK{C_{\DeltaLieK}}
\def\CDeltaLieL{C_{\DeltaLieL}}
\def\CDeltaLieLT{C_{\DeltaLieLT}}
\def\CDeltaLieG{C_{\DeltaLieG}}
\def\CDeltaLieGL{C_{\DeltaLieGL}}
\def\CDeltaLieB{C_{\DeltaLieB}}
\def\CDeltaLieA{C_{\DeltaLieA}}
\def\CDeltaLieJ{C_{\DeltaLieJ}}
\def\CDeltaLieNO{C_{\DeltaLieNO}}
\def\CDeltaLieN{C_{\DeltaLieN}}
\def\CDeltaLietOmega{C_{\DeltaLietOmega}}
\def\CDeltaLietOmegaL{C_{\DeltaLietOmegaL}}

\def\hCDelta{\hat{C}_\Delta}

\allowdisplaybreaks

\begin{document}

\title[Constructive QP-Time-Dependent KAM Theory for Lagrangian Tori]{
Constructive Approaches to QP-Time-Dependent KAM Theory for Lagrangian Tori in Hamiltonian Systems}

\date{\today}

\author{Renato Calleja$^{\mbox{\textdagger} }$}
\address[\textdagger]{IIMAS, Universidad Nacional Auti\'onoma de M\'exico, Apdo. Postal 20-126, C.P. 01000, M\'exico D.F., M\'exico.}
\email{calleja@mym.iimas.unam.mx}

\author{Alex Haro$^{\mbox{\textdaggerdbl}, \mbox{\textasteriskcentered}}$}
\address[\textdaggerdbl]{Departament de Matem\`atiques i Inform\`atica, Universitat de Barcelona,
Gran Via 585, 08007 Barcelona, Spain.}
\address[\textasteriskcentered]{ Centre de Recerca Matem\`atica, Edifici C, Campus Bellaterra, 08193 Bellaterra, Spain.}
\email{alex@maia.ub.es}

\author{Pedro Porras$^{\mbox{\textdagger}, \mbox{\textdaggerdbl}}$}
\email{pedro.porras@iimas.unam.mx}
\thanks{Corresponding author: Pedro Porras (\texttt{pedro.porras@iimas.unam.mx}).}
\begin{abstract}
In this paper, 
we prove a KAM theorem in a-posteriori format, 
using the parameterization method to look invariant tori in non-autonomous Hamiltonian systems with 
$n$ degrees of freedom that depend periodically or quasi-periodically (QP) on time,
with $\ell$ external frequencies. 
Such a system is described by a Hamiltonian function in the $2n$-dimensional phase space,
$\mani$, 
that depends also on $\ell$ angles, 
$\varphi\in \TT^\ell$. 
We take advantage of the fibbered structure of the extended phase space
$\mani \times \TT^\ell$. 
As a result of our approach, 
the parameterization of tori requires the last $\ell$ variables, to be precise $\varphi$, 
while the first $2n$ components are determined by an invariance equation. 
This reduction decreases the dimension of the problem where the unknown is a parameterization from $2(n+\ell)$ to $2n$.

We employ a quasi-Newton method, in order to prove the KAM theorem.
This iterative method begins with an initial parameterization of an approximately invariant torus, 
meaning it approximately satisfies the invariance equation 
The approximation is refined by applying corrections that reduce quadratically the invariance equation error. 
This process converges to a torus in a complex strip of size $\rho_\infty$, 
provided suitable Diophantine $(\gamma,\tau)$ conditions and a non-degeneracy condition on the torsion are met. 
Given the nature of the proof, 
this provides a numerical method that can be effectively implemented on a computer, 
the details are given in the companion paper \cite{porras2025constructiveII}. 
This approach leverages precision and efficiency to compute invariant tori.
\end{abstract}
\maketitle

\tableofcontents

\section{Introduction}

The KAM Theory, 
developed by Kolmogorov, 
Arnold, 
and Moser \cite{kolmogorov1954conservation, Arnold63a, moser1962invariant}, 
plays a crucial role in understanding the dynamics of Hamiltonian systems. 
This theory states that, 
certain invariant tori in an integrable system survive after small perturbations, 
maintaining a quasiperiodic structure. 
Despite its importance, 
the perturbative nature of KAM theory limits its practical application, 
as it does not always allow for the explicit calculation of invariant manifolds within the physically relevant regime.

However, 
de la Llave, 
González, 
Jorba, 
and Villanueva introduced the parameterization method in KAM Theory \cite{de2005kam}, 
which offers a significant improvement. 
This method establishes connections between theoretical predictions and practical applications, 
making it practical and accessible for real-world scenarios. 
By solving an invariance equation to parameterize the invariant torus and employing iterative schemes with quasi-Newton methods, 
the parameterization approach addresses the limitations of the classical perturbative methods. 
This allows researchers to bridge the gap between theory and practice, 
effectively applying KAM theory in various contexts.

For instance, 
time-dependent Hamiltonians, 
which naturally arise in astrodynamics and celestial mechanics, 
provide a framework for studying the evolution of dynamical systems. 
A fundamental class of these systems is known as the restricted $N$-body problems, 
where a small-mass particle moves under the gravitational influence of a set of primary bodies without affecting their motion. 

The Circular Restricted Three-Body Problem (CRTBP) is the simplest case, 
where two primary bodies orbit each other in circular paths around their common center of mass while the third body moves within their gravitational field \cite{szebehely1969theory}. 
A more general case is the Elliptic Restricted Three-Body Problem (ERTBP), 
where the primary bodies follow elliptical orbits, 
introducing explicit periodic time dependence. 

The Quasi-Bicircular Problem (QBCP) further complicates the system by incorporating a third massive body in an approximately circular orbit, 
leading to additional resonances \cite{andreu1998quasi,celletti2007kam}.
These models are crucial for understanding dynamical stability, 
the existence of invariant tori in near-integrable systems, 
and the design of space missions
\cite{andreu1999translunar, andreu2003preliminary}.

Building on this,
Fern\'andez-Mora, 
Haro and Mondelo \cite{fernandez2024flow,fernandez2024convergence} introduce a method to compute parameterizations of partially hyperbolic invariant tori and their invariant bundles in non-autonomous quasi-periodic Hamiltonian systems. 
The papers \cite{fernandez2024flow,fernandez2024convergence} generalize flow map parameterization techniques to the quasi-periodic setting.
Introduces fiberwise isotropic tori and discusses symplectic deformations and moment maps
to compute parameterizations of partially hyperbolic invariant tori and their invariant bundles in non-autonomous quasi-periodic Hamiltonian systems. 
These constructions are vital for guaranteeing the existence of solutions to cohomological equations, 
and their algorithms have been successfully applied to the elliptic restricted three-body problem to compute nonresonant 3-dimensional invariant tori and their invariant bundles around the $L_1$ point.

Other scenarios where periodic and quasi-periodic Hamiltonian systems naturally emerge include Tokamaks, toroidal devices designed for magnetic confinement fusion. 
Studies, 
such as those \cite{chandre2005control} by Chandre, 
Vittot, 
Ciraolo, 
Ghendrih, 
and Lima, 
or \cite{del2000self} by del Castillo-Negrete, 
highlight the dynamics of magnetic field lines and transport phenomena in these devices. 
Understanding the stability and structure of invariant tori in this context is essential for effective plasma confinement and reducing chaotic transport. 
All these considerations have motivated the present work.
We present a KAM theorem for quasiperiodically time-dependent systems in a-posteriori format, 
using the techniques developed in \cite{HaroCFLM16, haro2019posteriori, figueras2025sun}.

We present a proof in a-posteriori format using the parametrization method to obtain a Lagrangian torus in Hamiltonian systems that depend periodically 
or quasi-periodically on time, 
with external frequencies vector $\alpha \in \RR^\ell$.
One of the advantages of this constructive approach is that the proof suggests an algorithm that is implemented numerically on a computer in \cite{porras2025constructiveII}.
In this work, 
we also take advantage of the symplectic fibered structure of the extended phase space,
using the angles that describe the periodic or quasi-periodic external action.

When adopting the symplectic formalism, 
the dimension of the extended phase space typically becomes $2(n+\ell)$. 
However, 
given our focus on systems with temporal dynamics that are either periodic or quasi-periodic, 
we redefine the setup. 
By introducing a variable $\varphi \in \TT^\ell$, 
which parametrizes the external action, 
the dimension of the phase space is effectively reduced to $2n +\ell$. 
This adjustment not only simplifies the representation but also aligns naturally with the underlying dynamics, 
enabling a more efficient analysis of the system.

As a result of our approach, 
the parameterization of the torus requires the last $\ell$ variables to have prescribed periodic dynamics, 
while the $n$ first components are determined by an invariance equation. 
This reduction allows us to bring down the dimension problem from $2(n+\ell)$ to $2n$. 
Moreover, 
we manage to obtain a symmetric expression for the torsion, 
that does not require computing derivatives of known functions but only their products,
which represents a considerable numerical advantage.
This feature makes the corresponding numerical method very efficient.

The novel expression for the torsion, 
given in equation \eqref{eq:CnewT}, 
has a significant advantage over equation (4.60) in \cite{haro2019posteriori}, 
as it avoids using derivatives and is based solely on the product of functions. 
This reformulation eliminates the need for Cauchy estimates, 
leading to improved estimates of the quasi-Newton step for convergence. 
From a numerical point of view, 
this modification makes the algorithm more efficient, 
reducing computational costs and favoring its implementation in high-precision computations,
see \cite{porras2025constructiveII}.
The proof follows an a-posteriori approach, 
that employs a quasi-Newton method to establish the existence of invariant tori in Hamiltonian dynamical systems. 
The iterative method starts with an initial parameterization of an approximately invariant torus,
in a space with complex strip of size $\rho$,
such that the invariance equation is approximately satisfied.
This approximation is improved by adding a correction that reduces the invariance error to an error which is quadratic with respect to the original error.
The process converges to a torus in a space  with complex strip of size $\rho_\infty$, 
under suitable Diophantine $(\gamma,\tau)$ conditions and a non-degeneracy condition on the torsion, 
which allows us to prove the persistence of the invariant tori,
if $\gamma^{-4} \rho^{-4\tau}$ times the norm is quadratic with respect to the original error.
However, 
with some additional effort, 
the smallness factor can be relaxed to $\gamma^{-2} \rho^{-2\tau -1}$, like in \cite{figueras2024modified}.
The nature of the proof is constructive so it leads quite naturally to a numerical algorithm
that can be implemented on a computer. 
The framework is also suitable for obtaing results in different scales of Banach spaces using hard implicit function theorems and smothing \cite{calleja2010numerically}.
The theorem states that, 
under necessary and sufficient geometrical and dynamical hypothesis, 
there exists a constant $\mathfrak{C}_1$ such that the inequality \eqref{eq:KAM:HYP} is satisfied. 
An additional advantage of this approach is that the constants $\mathfrak{C}_1$, $(\gamma, \tau)$ and $\rho$ can be explicitly calculated for each problem.
This makes it possible to determine the accuracy needed to calculate the invariance error, 
thus guatanteeing the existence of the torus.
A great advantage of the a-posteriori approach is that the hypotheses of the proof depend on the existence of an approximately invariant solution that satisfies regularity and
non-degeneracy conditions.

The approximate solution can be obtained in several ways, 
such as through numerical computations, 
Lindstedt series, 
the integrable case, 
etc.
However, 
since physically relevant regimes are often far from integrability, 
the theorem provides a way to work in these settings, 
as long as a good approximation is available.
In particular, 
one can take advantage of numerically computed results to estimate the constant $\mathfrak{C}_1$, 
the invariance error, 
and the Diophantine conditions. 
This framework allows for Computer-Assisted Proofs (CAPs), 
as demonstrated in the work of \cite{figueras2017rigorous},(see \cite{celletti1987rigorous, valvo2021hamiltonian} for other approaches). 
This not only strengthens the validity of the theoretical results but also guarantees exceptional accuracy and robustness in verifying the existence of the torus.
Furthermore, there are real-world problems where relevant solutions can be found.
We note that our approach differs from other works that rely on powerful numerical integrations to validate flows and jets like \cite{kapela2017rigorous, makino2009rigorous}. 
Instead, we focus on validating the parameterization approximation and its respective estimations at each Newton step. 

The structure of this paper is as follows. 
In section \ref{sec:back:elem:cons} we present, 
the notation and tools needed to develop a KAM theorem for Hamiltonian systems with periodic or quasi-periodic time dependencies. 
This includes the use of the r-th order derivative of, 
analytic functions, 
norms, 
and the Diophantine condition. 
Additionally, 
we address nonautonomous Hamiltonians that depend periodically or quasi-periodically on time, 
focusing on invariant tori, 
the invariance equation, 
and the required dynamical and geometric properties.
In section \ref{sec:KAM:Thm} we prove the KAM theorem for Lagrangian tori by first establishing foundational lemmas. 
These lemmas provide the necessary conditions and properties to demonstrate the persistence of invariant tori in the system. 
This lays the groundwork for proving the theorem under specific dynamical and geometric conditions. 
In section \ref{sec:lemmas} we provide estimates for various geometrical and dynamical properties of an approximately invariant torus. 
These approximations are crucial to demonstrate the quadratic convergence of the method. 
In section \ref{sec:proof:KAM}
we provide a proof of the KAM Theorem. 
For clarity, 
we begin by outlining the approach used to correct the parameterization of the torus. 
Specifically, 
Section \ref{ssec:qNewton} covers the approximate solution of the linearized equations within the symplectic framework constructed in Section \ref{ssec:symp}, 
establishing a method for solving the invariance equation. 
Section \ref{ssec:iter:lemmas} presents quantitative estimates for the objects derived from one iteration of this procedure. 
Section \ref{ssec:proof:KAM} discusses the convergence of the quasi-Newton method.
Finally, 
in the appendices, 
we find the derivation of a novel torsion expression, 
an auxiliary lemma to control the inverse of a matrix, 
and a compendium of constants in the KAM theorem.
These constants are explicit and are used to compute the constants $\mathfrak{C}_1$ and $\mathfrak{C}_2$, 
that appear in the KAM theorem.

\section{Background and elementary constructions}\label{sec:back:elem:cons}

\subsection{Basic notation}\label{ssec:basic:notation}
%
Let $U\subset \RR^m$, 
$V\subset \RR^n$ and $\rho>0$, 
the complex strip of size 
$\rho$ is 
$U_\rho \times V_\rho= 
\{ (x, y) \in \CC^m \times \CC^n : \re\,(x, y) \in U\times V \, , \, |\im\, (x, y)|<\rho \}$. 
Given two sets $X,Y \subset \CC^{m} \times \CC^{n}$, $\dist(X,Y)$ is defined as $\inf\{|x-y| \,:\, x\in X \, , \, y\in Y\}$.

We denote by $\RR^{n_1 \times n_2}$ and $\CC^{n_1\times n_2}$ the spaces of $n_1 \times n_2$ matrices with components in $\RR$ and $\CC$, 
respectively. 
We consider the following identifications: 
$\RR^m\simeq \RR^{m\times 1}$ and $\CC^m\simeq \CC^{m\times 1}$. 
We use $I_n$ and $O_n$ to represent the $n\times n$ identity and zero matrices, 
respectively. 
The $n_1\times n_2$ zero matrix is represented by $O_{n_1\times n_2}$. 
Additionally, 
the column vector $O_{n \times 1}$ is denoted by $0_n$. 
The matrix norms in both $\RR^{n_1 \times n_2}$ and $\CC^{n_1\times n_2}$ is sup-norm. 
That is to say, 
for an $n_1 \times n_2$ matrix $M$, 
we have
\[
    |M| = \max_{i= 1,\dots,n_1} \sum_{j= 1,\dots, n_2} |M_{i,j}|.
\]   
For bounded analytic function $f: \U\times \V \to \CC$, 
defined in an open set $\U \subset \CC^n$ and $\V\subset\CC^m$,
we define its norm
\[
    \norm{f}_{\U\times\V}= \sup_{(x,y)\in \U\times\V} |f(x, y)|.
\]  
The action of the $rth$-order derivative of $f$ at a point $(x,y)\in \U \times \V$ 
on a collection of (column) vectors $v_1,\dots, v_r\in \CC^m$, 
with $v_k= (v_{1k}, \dots, v_{mk})^\top$, 
is
\[
\Dif_{x}^r f(x,y) [v_1,\dots,v_r] = \sum_{\ell_1,\dots,\ell_r} \frac{\partial^r f}{\partial x_{\ell_1}\dots\partial x_{\ell_r}}(x,y)\ v_{\ell_1 1} \cdots v_{\ell_r r}, 
\]
where the indices $\ell_1,\dots,\ell_r$ run from $1$ to $m$. 
And its norm is,
\[
    \norm{\Dif_x^r f}_{\U\times\V}= \sum_{\ell_1,\dots,\ell_r} \Norm{\frac{\partial^r f}{\partial x_{\ell_1}\dots\partial x_{\ell_r}}}_{\U\times\V} \,.
\]
Given a matrix valued function 
$M:\U\times\V\subset\CC^m\times\CC^n \to \CC^{n_1\times n_2}$ (whose components, 
$M_{i,j}$, 
are analytic functions), 
a point $(x, y)\in \U \times \V$, 
and a collection of (column) vectors $v_1,\dots, v_r\in \CC^m$, 
we obtain a $n_1\times n_2$ matrix $\Dif_{x}^r M(x,y) [v_1,\dots,v_r]$ such that 
\[
\left(\Dif_x^r M(x,y) [v_1,\dots,v_r]\right)_{i,j} = \Dif_{x}^r M_{i,j}(x,y) [v_1,\dots,v_r],
 \] 
since $M $ is analytic in $\U \times \V$, 
we can define the norm.
 \[
\norm{\Dif_{x}^r M}_{\U\times\V}= \max_{i= 1,\dots,n_1} \sum_{j= 1,\dots,n_2} \Norm{\Dif_{x}^r M_{i,j}}_{\U\times\V}\,,
\]
with,
\[
\norm{\Dif_x^r M_{i,j}}_{\U\times\V}= \sum_{\ell_1,\dots,\ell_r} \Norm{\frac{\partial^r M_{i,j}}{\partial x_{\ell_1}\dots\partial x_{\ell_r}}}_{\U\times\V}\,.
\]
Therefore, we have,
\begin{align*}
&\norm{\Dif_x^r M [v_1,\dots, v_r]}_ {\U\times\V}
\leq \Norm{\Dif_x^r M}_{\U\times\V}\ \norm{v_1}_\U\cdots \norm{v_r}_{\U} \,.
\end{align*}
There is also a similar bound for the action of the transpose,
\begin{align*}
&\norm{(\Dif_x^r M [v_1,\dots, v_r])^\top}_ {\U\times\V}
\leq  \Norm{\Dif_x^r M^\top}_ {\U\times\V}\ \norm{v_1}_\U\cdots \norm{v_r}_ {\U}\,.
\end{align*}
We introduce the key notation that will be used throughout this analysis.
$\TT^d = \RR^d/\ZZ^d$ is the $d$-dimensional standard torus, 
$\TT^{d}_{\rho}= \{\theta\in \CC^d / \ZZ^d : |\im\ \theta| < \rho\}$ 
is the complex strip of $\TT^d$ of width
$\rho>0$.
Let us consider a function $u:\TT^d \times \TT^{\ell}  \to \RR$. 
Analogously, for $\rho>0$, $u:\TT^d _\rho\times \TT^{\ell}_\rho \to \RR$.
We will also write the Fourier expansion of a periodic function as
\[
u(\theta,\varphi)=\sum_{k_1 \in \ZZ^d } \sum_{k_2 \in \ZZ^{\ell} } \hat u_{k_1,k_2} \ee^{2\pi \ii (k_1 \cdot \theta + k_2 \cdot \varphi)}, 
\]
\[
\hat u_{k_1,k_2} =
\int_{\TT^{\ell}} \int_{\TT^d} u(\theta, \varphi) \ee^{-2 \pi \ii (k_1 \cdot \theta + k_2 \cdot \varphi)} \dif \theta \, \dif \varphi \,,
\]
and introduce the notation $\aver{u}:=\hat u_{0,0}$ for the average. 
The notation above is extended
to $n_1 \times n_2$ matrix-valued periodic functions 
$M: \TT^d_\rho \times \TT^{\ell}_{\rho} \to\CC^{n_1\times n_2}$, 
for which $\hat M_{k_1,k_2} \in \CC^{n_1 \times n_2}$ 
denotes the Fourier coefficient with index 
$(k_1,k_2)\in \ZZ^d\times \ZZ^{\ell}$.
Now, 
we denote the $\Anal(\TT^d_\rho \times \TT^{\ell}_\rho)$ 
to represent the Banach space of holomorphic functions 
$u:\TT^d_\rho \times \TT^{\ell}_\rho \to \CC$ 
that can be continuously extended to 
$\bar \TT^d_\rho \times \bar \TT^{\ell}_\rho $ and whose values in 
$\TT^d_\rho \times \TT^{\ell}_\rho$, 
are real-analytic. 
The space $\Anal(\TT^d_\rho \times \TT^{\ell}_\rho )$ is endowed with the norm,
\[
\norm{u}_\rho = \norm{u}_{\TT^d_\rho \times \TT^{\ell}_\rho }= \max_{|\im(\theta,\varphi)|\leq \rho} |u(\theta,\varphi)|\,.
\] 
Using Cauchy estimates, 
we can determine the derivatives of these functions. 
For example, 
if $u \in \Anal(\TT^d_\rho \times \TT^{\ell}_\rho)$ and $\rho>0$, 
then its partial derivative $\pd u/\pd {x_\ell}$ belongs to 
$\Anal(\TT^d_{\rho-\delta} \times \TT^{\ell}_{\rho-\delta} )$ for any $0<\delta<\rho$ 
and we have the following bounds: 
 \[
\Norm{
\frac{\pd u}{\pd x_{\ell}}}_{\rho-\delta} \leq \frac{1}{\delta}\norm{u}_\rho, 
\qquad
\Norm{\Dif_{x} u}_{\rho-\delta} \leq \frac{d}{\delta}\norm{u}_\rho, 
\qquad
\Norm{(\Dif_{x} u)^\top}_{\rho-\delta} \leq \frac{1}{\delta}\norm{u}_\rho.
\]
For matrix valued functions 
$M: \TT^d_\rho \times \TT^{\ell}_\rho \to \CC^{n_1\times n_2}$ 
with components in $ \Anal(\TT^d_\rho \times \TT^{\ell}_\rho)$, 
we have $\norm{\Dif_{x} M}_{\rho-\delta} \leq \frac{d}{\delta} \norm{M}_\rho$. 
This result also applies to the transpose of $M$, 
$\norm{\Dif_{x} M^\top}_{\rho-\delta} \leq \frac{d}{\delta} \norm{M^\top}_\rho$. 
In the particular case 
$w:\TT^d_\rho \times \TT^{\ell}_\rho \to \CC^n \approx \CC^{n\times 1}$, 
we obtain:
\[
	\norm{\Dif_{x} w}_{\rho-\delta} \leq \frac{d}{\delta} \norm{w}_\rho,\, 
	\norm{\Dif_{x} w^\top}_{\rho-\delta} \leq \frac{d}{\delta} \norm{w^\top}_\rho \leq \frac{n d}{\delta} \norm{w}_\rho, \]
\[
	\norm{(\Dif_{x} w)^\top}_{\rho-\delta} \leq \frac{n}{\delta}\norm{w}_\rho.
\]
The set of Diophantine vectors is defined as
\begin{equation}\label{eq:def:Dioph}
\Dioph{\gamma}{\tau} =
\left\{
(\omega, \alpha) \in \RR^d\times\RR^{\ell} \, : \,
\abs{ k_1 \cdot \omega + k_2 \cdot\alpha} \geq \frac{\gamma}{|(k_1,k_2)|_1^{\tau}}
\,, 
\forall (k_1,k_2) \in\ZZ^d\times\ZZ^{\ell}\backslash\{0\}
\right\}\,.
\end{equation}
An essential condition in this theorem is the assumption that the frequency vector 
$(\omega,\alpha)$ for specific $\gamma >0$ and $\tau \geq d + \ell -1$, 
satisfies the Diophantine condition.
The quantity $|(k_1,k_2)|_1$ represents the sum of the absolute values of each component of the vector $(k_1,k_2)$.
Finally, 
we present notation regarding the cohomological equations which are central to KAM theory. 
This notation describes the relationship between a periodic or quasi-periodic function, 
\[
v(\theta, \varphi)=\sum_{k_1 \in \ZZ^d} \sum_{k _2\in \ZZ^{\ell}} \hat v_{k_{1,2}} \ee^{2\pi \ii (k_1 \cdot \theta + k_2 \cdot \varphi)},
\] 
and the frequency vector, 
$(\omega,\alpha) \in \RR^{d}\times \RR^{\ell}$,
\begin{equation}\label{eq:calL}
\Lie{\omega,\alpha} u = v- \aver{v}\,, \qquad
\Lie{\omega,\alpha} := -\left( \sum_{i=1}^d \omega_i \frac{\pd}{\pd \theta_i} +\sum_{j=1}^{\ell} \alpha_j \frac{\pd}{\pd \varphi_j} \right).
\end{equation}
The formal solution of equation \eqref{eq:calL} can be expressed as 
\begin{equation}\label{eq:small:formal}
\R{\omega,\alpha} v(\theta,\varphi) =\displaystyle \sum_{k_1 \in \ZZ^d \backslash \{0\} }\displaystyle\sum_{k _2\in \ZZ^{\ell} \backslash \{0\} } \hat u_{k_1,k_2} \ee^{2\pi
\ii (k_1 \cdot  \theta + k_2 \cdot \varphi)}, \quad \hat u_{k_1,k_2} = \frac{-\hat
v_{k_1,k_2}}{2\pi \ii  (k_1 \cdot \omega + k_2 \cdot \alpha )}.
\end{equation}
Notice that if $v=0$, 
then $\R{\omega, \alpha} v =0$.

The notation required for this paper, 
including analytic function spaces, 
norms of these functions and their derivatives, 
Diophantine vectors, cohomological equations, 
and related concepts
can be extended from the works  \cite{HaroCFLM16, haro2019posteriori, figueras2024modified} and is very similar to \cite{fernandez2024flow}. 
For convenience, 
we refer the reader to these references.
\subsection{QP Hamiltonian systems, symplectic bundles and invariant tori}\label{ssec:inv:tor}
In this paper, 
we require an open set $\mani$ of $\RR^{2n}$ endowed with an exact symplectic form $\sform = \dif \aform$, 
for a certain 1-form $\aform$ called \emph{action form}. 
This setting can be adapted to the case in which $\mani$ is open subset of $\mani\subset \TT^{n_1}\times\RR^{n_2}\times\RR^n$, with $n_1 + n_2 = n$. 
We will assume that $\sform$ is exact , 
so $\mani$ is endowed with an exact symplectic structure. 
The matrix representations of $\aform$ and $\sform$ are given by the matrix-valued functions
\[
\begin{array}{rcl}
a: \mani & \longrightarrow & \RR^{2n}\,, \\
z & \longmapsto & a(z)
\,,
\end{array}
\]
and
\[
\begin{array}{rcl}
\Omega: \mani & \longrightarrow & \RR^{2n \times 2n} \,,\\
z & \longmapsto & \Omega(z)
=(\Dif a(z))^\top-\Dif a(z)
 \,,
\end{array}
\]
the non-degeneracy of $\sform$ is equivalent to $\det \Omega(z) \neq 0$ for all $z\in\mani$.
If there exists a globally defined function $h: \mani \rightarrow \RR$ such that $i_{X_\H} \sform = -\dif \H$, 
then $X_\H$ is said to be a Hamiltonian vector field with respect to the symplectic form $\sform$ 
and $\H$ is a Hamiltonian function for $X_{\H}$. 
In coordinates, 
\[
    X_{\H}(z)= \Omega(z)^{-1} (\Dif \H(z))^\top\,.
\]
In the non-autonomous case, the Hamiltonian can be expressed as follows,
\[
\begin{array}{rcl}
    \H: \mani \times \RR & \longrightarrow &\RR \,, \\
(z, t) & \longmapsto & \H(z,t) \,.
\end{array}
\]
The symplectic $2$-form requires an even-dimensional phase space, 
while $\mani \times\RR$,
is odd-dimensional.   
The typical way to overcome this issue is to introduce a  on 
$\mani\times\RR$ using the form 
$\sform_{h} = \sform + \dif \H \wedge \dif t$. 
However, 
this means that the phase space dimension is $2(n+1)$, 
and if we want to study time-dependent periodic or quasi-periodic dynamics with associated frequencies $\alpha \in \RR^{\ell}$, 
the phase space dimension becomes $2(n +\ell)$,
since the symplectic form on $\mani\times\TT^\ell$ is $\sform_{h} = \sform + \dif I \wedge \dif \varphi$,
with $I$ the fictitious variables conjugated to $\varphi$. 
This would be expensive from the numerical point of view so, 
an alternative approach is to consider that the time dependence is on an 
$\ell$-dimensional torus. In other words,
\[
\begin{array}{rcl}
\H: \mani \times \TT^{\ell} & \longrightarrow &\RR \,, \\
(z, \varphi) & \longmapsto & \H(z,\varphi) \,,
\end{array}
\]
where 
$\dot \varphi = \alpha $ and 
$\alpha \in \RR^{\ell}$ are the external frequencies. 
We consider the evolution operator 
\begin{equation}\label{eq:evol:oper}
\begin{aligned}
    \phi_t : \cI(z, \varphi) \subset \RR & \longrightarrow \mani  \\
    t & \longmapsto  \phi_t(z,\varphi)   \,,
\end{aligned}
\end{equation}
where $\cI(\theta, \varphi)$ is the maximal interval of exitence,
and $\phi_{t} (z,\varphi)$ satisfies
\[
\begin{split}
    \frac{\partial \phi_t}{\partial t}(z,\varphi) &= X_\H(\phi_t(z,\varphi),\varphi + t \alpha)\,,\\
    \phi_0(z,\varphi) &=z \,,
\end{split}
\]
where $X_H$ is a vector field.
We define
\[
    \D = \{ (t,z,\varphi ) | (z,\varphi) \in \mani \times \TT^\ell, t \in \cI(z,\varphi)\}
\]
Now, we consider the following flow, 

\begin{equation}\label{eq:flow}
\begin{aligned}
\tilde{\Phi}_t : \D & \longrightarrow \mani \times \TT^{\ell} \\
(z, \varphi, t) & \longmapsto  \tilde{\Phi}_{t}(z, \varphi) := \begin{pmatrix} \phi_{t}(z,\varphi) \\ \varphi + t \alpha \end{pmatrix}  \,,
\end{aligned}
\end{equation}
\begin{remark}\label{rmk:Phi:flow}
The equation \eqref{eq:flow} satisfies the flow properties.
\begin{equation*}
 \begin{split}
\tilde{\Phi}_{0} (z,  \varphi)  & = \begin{pmatrix} \phi_{0}(z,\varphi) \\ \varphi \end{pmatrix} = \begin{pmatrix} z \\ \varphi \end{pmatrix},\\
\text{ and }\\
 \tilde{\Phi}_{t+s}(z,  \varphi )&= \begin{pmatrix} \phi_{t+s}(z,\varphi) \\ \varphi + (t + s)\alpha  \end{pmatrix} = \begin{pmatrix} (\phi_{t} \circ \phi_{s} )(z,\varphi) \\ \varphi + t\alpha +s\alpha  \end{pmatrix} =  \begin{pmatrix}  \phi(\phi(z,\varphi),s),t) \\ \varphi + t\alpha +s\alpha \end{pmatrix} \\
 &=  \tilde{\Phi}_{t}\begin{pmatrix} \phi_{s}(z,\varphi) \\ \varphi + s\alpha \end{pmatrix}= \tilde{\Phi}_{s} \left( \tilde{\Phi}_{s}\begin{pmatrix} z \\ \varphi \end{pmatrix} \right) =  \tilde{\Phi}_{s} \circ \tilde{\Phi}_{s} \begin{pmatrix} z \\ \varphi \end{pmatrix} .\\
 \end{split}
 \end{equation*}
\end{remark}

Therefore, 
we can define the vector field of a non-autonomous system that depends periodically or quasi-periodically on time accordingly.
As a result, 
we can determine the vector field of a non-autonomous system that depends periodically or quasi-periodically on time,
\[
\begin{array}{rcl}
    \tilde{Z} : \mani \times \TT^{\ell} & \longrightarrow & \RR^{2n} \times \TT^{\ell}\,, \\
(z, \varphi) & \longmapsto & \tilde{Z}(z, \varphi) := \begin{pmatrix} Z_\H(z,\varphi) \\ \alpha \end{pmatrix}  \,,
\end{array}
\]
with $Z_\H(z, \varphi)= \Omega(z)^{-1} (\Dif_{z} \H(z,\varphi))^\top$.  
We look for a parameterization $\tilde{K}:\TT^{d} \times \TT^{\ell} \to \mani \times \TT^\ell $, 
with $1\leq d \leq n$, 
that satisfies the invariance equation 
\begin{equation}\label{eq:flow-inv-eq}
\tilde{\Phi}_t\left( \tilde{K}(\theta, \varphi)\right) = \tilde{K}\left( \theta + t\omega, \varphi + t\alpha \right)\,,
\end{equation} 
where $\omega \in \RR^{d}$ is a vector of internal frequencies, 
furthermore $\omega$ and $\alpha$ are non-resonant. 
In this paper, 
we are interested in finding a parameterization for the Lagrangian Torus, 
as discussed in \cite{moser1966theory,broer1996quasi}, 
i.e.,
we consider the case where $d=n$; 
hence, 
throughout this work, 
we focus exclusively on this scenario.
From the above equation, $\tilde{K}$ is of the form,
\begin{equation}\label{eq:tilde-torus}
\tilde{K}(\theta,\varphi) = \begin{pmatrix} K(\theta,\varphi) \\ \varphi \end{pmatrix} \,.
\end{equation}
We denotate by $\torus = \K\left(\TT^n, \TT^{\ell} \right)$,
the embedded submanifold of $\mani\times\TT^{\ell}$.
If $\proj_\varphi$ is the projection on $\mani$ with basepoint $\varphi$,
then $\torus_\varphi = \K\left(\TT^n, \vp \right)$ is fiber of $\torus$.
\begin{figure}[h!]
\centering         
\includegraphics[scale=1]{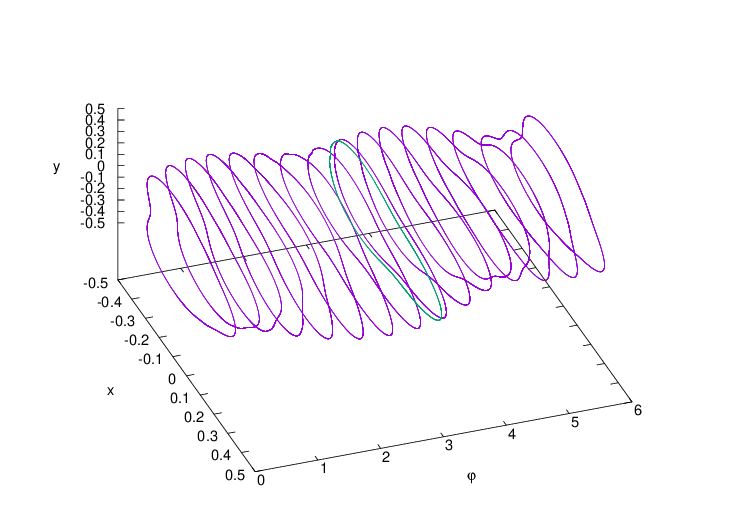}
\caption{Two-dimensional torus expressed in cylindrical coordinates, each curve is a leaf of the torus fibration. 
    The green leaf is the $K$ parameterization for a given $\varphi$, $\torus_\varphi$.}
\label{fig:twoPiTorus}
\end{figure}
Parameterization \eqref{eq:tilde-torus} has the advantage that 
$K: \TT^n \times \TT^{\ell} \to \mani$, 
now only has to solve equation,
\begin{equation}\label{eq:inv:fv}
Z_\H (K(\theta, \varphi), \varphi) = \Dif_{\theta} K(\theta,\varphi) \omega + \Dif_{\varphi} K(\theta,\varphi) \alpha \,,
\end{equation}
which is an inifitesimal version of \eqref{eq:flow-inv-eq} for the $n$-component of $\K$. 
Notice that the equation for $\varphi$ is trivial. 
Geometrically, we visualize that the manifold 
$\tilde K(\theta, \varphi)$ is foliated by 
$K(\theta, \varphi)$ for every $\varphi$, as shown in Figure \ref{fig:twoPiTorus}. 
This means that the ($n+\ell$)-dimensional manifold 
$\torus=K(\TT^n , \TT^{\ell})$ is invariant and the internal dynamics is given by the constant vector field $(\omega, \alpha)$,  
the invariance equation can be written in the form
\begin{equation}
Z_{h}(K(\theta, \varphi), \varphi) = -\Lie{\omega, \alpha} (K(\theta, \varphi))\,.
\label{eqn:invariance_equation}
\end{equation} 
If the torus is approximately invariant, 
we define the \emph{error of invariance}
$E:\TT^n \times \TT^{\ell} \rightarrow \mani$ given by
\begin{equation}\label{eq:inv:err}
E(\theta, \varphi) = Z_{h}(K(\theta, \varphi), \varphi) + \Lie{\omega, \alpha} (K(\theta, \varphi)).
\end{equation}

\begin{remark}\label{rmk:fictus:approach}
    Another way to handle time-dependent quasiperiodic Hamiltonian systems,
    \[
        \begin{array}{rcl}
            \H: \RR^{2n} \times \TT^{\ell} & \longrightarrow &\RR \,, \\
            (z, \varphi) & \longmapsto & \H(z,\varphi) \,,
        \end{array}
    \]
    is to extend the symplectic formalism by introducing fictitious variables, $\I$,
    associated with momentum, 
    which are conjugate to the quasi-periodic time variables, i.e.
\[
\begin{array}{rcl}
    \hat{\H}: \RR^{2n} \times \TT^{\ell} \times \RR^{\ell}& \longrightarrow &\RR \,, \\
    (z, \varphi, \I) & \longmapsto & \hat{\H}(z, \varphi, I) =  \H(z,\varphi) + \I \cdot \alpha \,,
\end{array}
\]
with $\alpha\in\RR^{\ell}$ external frequencies. 
    We obtain a parameterization of an invariant torus of dimension $2(n+\ell)$, as follows 
    \[
        \hat{K}(\theta, \varphi) = 
        \begin{pmatrix}
            K(\theta, \varphi) \\
            \varphi \\
            \I(\theta, \varphi)   
        \end{pmatrix} \,.
    \]

With this approach, the problem has dimension 
$2(n+\ell)$instead of $2n$. 
Moreover, 
the proof becomes harder since the involved objects have more intricate expressions. 
For further details, 
see the Appendix \ref{ssec:extended}.
\end{remark}

\subsection{Linearized dynamics and reducibility}\label{ssec:red:lin:eq}
Let us introduce an operator for any $m$-dimensional vector subbundle, 
parameterized by 
$V: \TT^n \times \TT^{\ell} \rightarrow \RR^{2n \times m}$, 
where $1\leq m \leq 2n$. 
The operator is defined as 
\begin{equation}\label{eq:Loper}
\Loper{V} (\theta, \varphi):= \Dif_{z}Z_{\H}(K(\theta, \varphi), \varphi)V(\theta,\varphi) + \Lie{\omega, \alpha} (V(\theta, \varphi))\,.
\end{equation}
We say that a subbundle is invariant under the linearized equations if 
$\Loper{V}(\theta,\varphi) = O_{2n \times m}$ 
for all 
$(\theta, \varphi) \in \TT^n \times \TT^\ell$.
We consider the map $L:\TT^n \times \TT^{\ell} \rightarrow \RR^{2 n\times n}$ given by
\[
L(\theta, \varphi) = \Dif_{\theta} K(\theta, \varphi)\,,
\]
we note that 
$\Loper{L}(\theta, \varphi) = O_{2n \times n}$, 
for all $(\theta,\varphi)\in \TT^{n}\times \TT^{\ell}$. 
Moreover, 
the quasi-periodicity implies that the torus is fiberwise Lagrangian, 
i.e., 
the pull-back for all $\varphi$, 
$K^*_\varphi\sform=0$ of the symplectic form on the torus vanishes. 
The matrix representation is 
\begin{equation} 
\label{def-OL}
\Omega_L(\theta, \varphi) 
=  (\Dif_{\theta} K(\theta, \varphi))^\top \Omega(K(\theta, \varphi)) \ \Dif_{\theta} K(\theta, \varphi)
=  (L(\theta, \varphi))^\top \Omega(K(\theta, \varphi)) \ L(\theta, \varphi)\,.
\end{equation}

Thus, 
$\torus$ is isotropic if $\Omega_L(\theta, \varphi)=O_{n}$. 
For each $\varphi \in \TT^{\ell}$,
$\torus_\varphi$ is a Lagrangian submanifold of $\mani$,
which is the case we study here.
For this purpose, 
we introduce several geometric constructions that allow us to define adapted symplectic frames. 
Hence, 
we will obtain a map 
$N: \TT^n \times \TT^{\ell} \rightarrow \RR^{2n\times n}$
such that the juxtaposed matrix
\[
P(\theta,\varphi) =
\begin{pmatrix} L(\theta,\varphi)  & N(\theta,\varphi) \end{pmatrix}\,,
\]
satisfies $\mathrm{rank}\, P(\theta,\varphi)=2n$ 
for every $(\theta, \varphi) \in \TT^n \times \TT^{\ell}$ and
\begin{equation}\label{eq:Psym}
P(\theta,\varphi)^\top \Omega(K(\theta,\varphi)) P(\theta,\varphi) = \begin{pmatrix} O_n &  -I_n \\ I_n  & O_n \end{pmatrix} := \Omega_0 \,,
\end{equation}
$\Omega_0$ denote the matrix representation of the Darboux form.
From equation \eqref{eq:Psym}, we derive the following geometric properties:
\begin{align}
    L(\theta,\varphi)^\top \Omega(K(\theta,\varphi)) L(\theta,\varphi)= {} & O_n \,, \label{eq:LTL} \\
    L(\theta,\varphi)^\top \Omega(K(\theta,\varphi)) N(\theta,\varphi)= {} & -I_n \,, \label{eq:LTN} \\
    N(\theta,\varphi)^\top \Omega(K(\theta,\varphi)) L(\theta,\varphi)= {} & I_n  \,, \label{eq:NTL} \\
    N(\theta,\varphi)^\top \Omega(K(\theta,\varphi)) N(\theta,\varphi)= {} & O_n \,. \label{eq:NTN}
\end{align}
In this case, 
we say that $P:\TT^n \times \TT^{\ell} \rightarrow \RR^{2n\times 2n}$ 
is a symplectic frame. 
The use of these linear coordinates allows the linearized equation to be reduce to a triangular form:
 \begin{equation*}
 \begin{split}
\Loper{P}(\theta, \varphi)&= 
     \Dif_{z} Z_\H(K(\theta,\varphi),\varphi) P(\theta,\varphi) + \Lie{\omega,\alpha} P(\theta,\varphi) \\
&= \begin{pmatrix}
O_{2n\times n} &  \Dif_{z} Z_\H(K(\theta,\varphi),\varphi) N(\theta,\varphi) + \Lie{\omega,\alpha} N(\theta,\varphi) 
\end{pmatrix}\\
&= P(\theta,\varphi)\Lambda(\theta,\varphi) \,,
 \end{split}
 \end{equation*}
with
\[
\Lambda(\theta,\varphi)
= \begin{pmatrix}
O_n &  T(\theta,\varphi) \\ 
O_n  & O_n
\end{pmatrix}\,,
\]
and $T: \TT^n \times \TT^{\ell} \rightarrow \RR^{n\times n}$
\begin{align}\label{eq:OldT}
T(\theta,\varphi) = {} & N(\theta,\varphi)^\top \Omega(K(\theta,\varphi))
\left(
\Dif_{z} Z_\H(K(\theta,\varphi),\varphi)
N(\theta,\varphi)
+ \Lie{\omega,\alpha} N(\theta,\varphi)
\right) \nonumber \\ 
= {} & N(\theta,\vp)^\top \Omega(K(\theta,\vp)) \Loper{N}(\theta,\vp)\,.
\end{align}
The matrix 
$T(\theta,\varphi)$ 
is usually called the \emph{torsion matrix} and will play a role in the non-degeneracy condition. 

\begin{remark}
    In this remark, 
    we consider the derivative of the expression 
    $\Omega(z)Z_{h}(z, \varphi) = \left( D_{z}h(z, \varphi) \right)^{\top}$,
    \[
        \displaystyle \sum_{m = 1}^{2n}\left( \dfrac{\partial \Omega_{j,m}(z)}{\partial z_i} Z_{m}(z, \varphi) + \Omega_{j,m}(z)\dfrac{\partial Z_{m}(z, \varphi)}{\partial z_{i}} \right) = \dfrac{\partial^2 h(z, \varphi)}{\partial z_{i}\partial z_{j}}\,.
    \]
    Since the Hamiltonian is analytic, 
    we have
    \begin{equation*}
        \begin{split}
            0 &= \dfrac{\partial^2 h(z, \varphi)}{\partial z_{j} \partial z_{i}} -  \dfrac{\partial^2 h(z, \varphi)}{\partial z_{i}\partial z_{j}}\\ 
            &=  \displaystyle \sum_{m = 1}^{2n}\left( \left( \dfrac{\partial \Omega_{i,m}(z)}{\partial z_j} - \dfrac{\partial \Omega_{j,m}(z)}{\partial z_i} \right) Z_{m}(z, \varphi) 
            + \Omega_{i,m}(z)\dfrac{\partial Z_{m}(z, \varphi)}{\partial z_{j}} \right.\\
            & \hspace{0.5cm} \left. - \Omega_{j,m}(z)\dfrac{\partial Z_{m}(z, \varphi)}{\partial z_{i}} \right)  \\
            &=  \displaystyle \sum_{m = 1}^{2n}\left( \dfrac{\partial \Omega_{i,j}(z)}{\partial z_{m}}  Z_{m}(z, \varphi) + \Omega_{i,m}(z)\dfrac{\partial Z_{m}(z, \varphi)}{\partial z_{j}} + \Omega_{m,j}(z)\dfrac{\partial Z_{m}(z, \varphi)}{\partial z_{i}} \right)\,.
        \end{split}
    \end{equation*}
    In the last equality, we use the skew-symmetry of the symplectic form and the following geometric property, $$\dfrac{\partial \Omega_{r,s}(z)}{\partial z_{t}} + \dfrac{\partial \Omega_{s,t}(z)}{\partial z_{r}} + \dfrac{\partial \Omega_{t,r}(z)}{\partial z_{s}} = 0$$
    for any triplet $(r,s,t)$. Then
    \begin{equation}
        D_{z}\Omega(z)\left[Z_{h}(z, \varphi)\right] + \left( D_{z} Z_{h}(z, \varphi)\right)^{\top}\Omega(z) + \Omega(z)D_{z}Z_{h}(z, \varphi) = O_{2n}.
        \label{eqn:Jacobi-indentity}
    \end{equation}
\end{remark} 

\subsection{Construction of a geometrically adapted frame}\label{ssec:sym:frame}
In this section, 
we construct a symplectic frame for the Lagrangian torus on the bundle 
$\Tan_\torus\mani$ by complementing the column vectors of a map 
$L:\TT^n\times \TT^{\ell} \rightarrow \RR^{2n \times n}$, 
that parameterizes a Lagrangian subbundle. 
To this end, 
we assume that we have a map 
$N^0:\TT^n\times \TT^{\ell} \rightarrow \RR^{2n\times n}$ such that.
\begin{equation}\label{eq:cond:CaseI}
\mathrm{rank} \begin{pmatrix} L(\theta, \varphi) & N^0(\theta, \varphi) \end{pmatrix} = 2n
\quad
\Leftrightarrow
\quad
\det (L(\theta,\varphi)^\top \Omega(K(\theta,\varphi)) N^0(\theta,\varphi)) \neq 0\,,
\end{equation}
for every $(\theta, \varphi) \in \TT^n \times \TT^{\ell}$. 
Then, 
we complement the Lagrangian subspace generated by $L(\theta,\varphi)$ 
by means of a map 
$N:\TT^n \times \TT^{\ell} \rightarrow \RR^{2n \times n}$ in the following manner, 
\[
N(\theta,\varphi) = L(\theta,\varphi)A(\theta,\varphi) + \tilde N(\theta,\varphi)\,.
\]
with,
\[ 
\tilde N(\theta,\varphi) = N^0(\theta,\varphi) B(\theta,\varphi).
\]
Using \eqref{eq:Psym},
we obtain,
\[
B(\theta,\varphi)=-(L(\theta,\varphi)^\top \Omega(K(\theta,\varphi)) N^0(\theta,\varphi))^{-1} \,,
\]
and 
\[
A(\theta,\varphi)=-\frac{1}{2}(\tilde N(\theta,\varphi)^\top \Omega(K(\theta,\varphi)) \tilde N(\theta,\varphi))\,,
\]
where $A$ is a solution of
\[
A(\theta,\varphi)-A(\theta,\varphi)^\top = -\tilde N(\theta,\varphi)^\top \Omega(K(\theta,\varphi)) \tilde N(\theta,\varphi).
\]
Following \cite{HaroCFLM16}, 
assume that $\mani$ is endowed with a Riemannian metric $\gform$, 
represented in coordinates as the positive definite symmetric matrix-valued function 
$G:\mani\to\RR^{2n\times 2n}$. 
Then, we define the linear isomorphism 
$\J: \Tan \mani \to \Tan\mani$ such that 
$\sform_z(\J_z u,v)=\gform_z(u,v)$, $\forall u,v \in T_z \mani$. 
Observe also that $\J$ is antisymmetric with respect to $\gform$, 
that is, 
$\gform_z(u,\J_z v)=-\gform_z(\J_zu,v)$.
The matrix-valued function gives the matrix representation of $\J$,
$J:\mani\to \RR^{2n\times 2n}$.
Then, we have
\begin{equation}\label{eq:prop:struc1}
    \Omega^\top = -\Omega\,, \qquad G^\top = G\,, \qquad J = -\Omega^{-1}G\,, 
\end{equation}
to represent the symplectic form using the frame given by $J$.
Then, we choose the map $N^0(\theta, \varphi)$ as follows
\begin{equation}\label{eq:map:N0:CII}
N^0(\theta, \varphi):= J(K(\theta, \varphi)) L(\theta, \varphi)
\end{equation}
and condition \eqref{eq:cond:CaseI} is equivalent to
\[
    \det (L(\theta, \varphi)^\top G(K(\theta, \varphi)) L(\theta, \varphi)) \neq 0\,\qquad \mbox{ for all } (\theta, \varphi) \in \TT^n \times \TT^{\ell} \,. 
\]
Moreover, 
the matrix $B(\theta,\varphi)$ is expressed as follows
\[
B(\theta,\varphi)= (L(\theta,\varphi)^\top G(K(\theta,\varphi)) L(\theta,\varphi))^{-1}\,.
\]
With this, we obtain the frame P.
Next, 
we derive the following expression for the torsion, 
with details provided in the Appendix~\ref{ssec:torsion}.

\begin{equation}\label{eq:newT}
\begin{split}
T(\te,\vp) 
= & - \dfrac{1}{2}\tilde N(\te,\vp)^\top \left(T_h(K(\te,\vp),\vp) +  T_h(K(\te,\vp),\vp)^\top\right)\tilde N(\te,\vp) \\
& + \tilde N(\te,\vp)^\top T_h(K(\te,\vp),\vp)^\top N(\te,\vp)  \\ 
& + N(\te,\vp)^\top T_h(K(\te,\vp),\vp)\tilde N(\te,\vp) \,,
\end{split}
\end{equation}
one advantage of \eqref{eq:newT}, 
is that it does not depend on the Lie derivative term of $N(\theta, \varphi)$. 
Moreover, 
the object 
\begin{equation}\label{eq:newTh}
\begin{split}
T_h(z, \varphi)= \Omega(z) \Big(& \Dif_z Z_\H(z, \vp) - \Dif_z J(z) \left[Z_\H(z, \vp) \right] J(z)^{-1} 
    - J(z) \Dif_z Z_\H(z, \vp) J(z)^{-1} \Big) \,,
\end{split}
\end{equation}
naturally emerges, 
linking the Hamiltonian vector field and the geometric structures $\Omega, G$ and $J$.
\begin{remark}
    Another expression to new torsion. equation, \eqref{eq:newT}
    \begin{equation}
        \begin{split}
            T&(\te,\vp) 
            =  N(\te,\vp)^\top \left( \dfrac{1}{2}  \left( T_h(K(\te,\vp),\vp) +  T_h(K(\te,\vp),\vp)^\top \right) \right) N(\te,\vp) \\
            & + A(\te,\vp)^\top L(\te,\vp)^\top T_h(K(\te,\vp),\vp)^\top N(\te,\vp)  \\ 
            & + N(\te,\vp)^\top T_h(K(\te,\vp),\vp)^\top L(\te,\vp) A(\te,\vp) \\ 
            & - A(\te,\vp)^\top L(\te,\vp)^\top \left( \dfrac{1}{2}  \left( T_h(K(\te,\vp),\vp) +  T_h(K(\te,\vp),\vp)^\top \right) \right) L(\te,\vp) A(\te,\vp) \,.
        \end{split}
    \end{equation}

\end{remark}
This new torsion expression has several advantanges that we discuss in further detail in the paper.
In Appendix \ref{ssec:torsion}, 
you can find further derivation details.
The objects computed so far, 
except for the new expression for the torsion, 
correspond to  \textbf{Case II} in \cite{HaroCFLM16}. 
There is also the special case where the isomorphism $\J$ is anti-involutive.
That is, 
$\J^2 = -\I$. 
Then, 
we say that the triplet $(\sform,\gform,\J)$ is
compatible and that $\J$ endows $\mani$ with an almost complex structure. 
This is called \textbf{Case III} in chapter $4$ of \cite{HaroCFLM16}.
In coordinates,
we have the following properties
\begin{equation*}
J^2= -I_{2n}, \qquad 
\Omega= J^\top \Omega J, \qquad G= J^\top G J. 
\end{equation*}
In this particular case, we have
\[
N^0(\theta,\varphi) = J(K(\theta,\varphi)) L(\theta,\varphi)\,,
\quad
A(\theta,\varphi) = O_n\,,
\]
\[
B(\theta,\varphi) = (L(\theta,\varphi)^\top G(K(\theta,\varphi)) L(\theta,\varphi))^{-1}\,,
\]
and, the new torsion is,
\begin{equation}
\label{eqn:newT-Case-III}
T(\te,\vp) 
= N(\te,\vp)^\top T_h(K(\te,\vp),\vp) N(\te,\vp) \,.
\end{equation}
and
\begin{equation}\label{eqn:newTh-Case-III}
\begin{split}
T_h(z,\varphi)= \Omega(z) \Big(& \Dif_z Z_\H(z , \vp) + \Dif_z J(z) \left[Z_\H(z, \vp) \right] J(z) + J(z) \Dif_z Z_\H(z, \vp) J(z) \Big) \,.
\end{split}
\end{equation}
We observe that both cases $T(\te, \vp)$ and $T_\H(\te, \vp)$ are symmetric.

\begin{remark}\label{rmk:sumarize}
To summarize, 
we have that, 
given a parameterization of a torus 
$\torus=K(\TT^n \times \TT^{\ell})$ (not necessarily invariant) and a tangent frame 
$L:\TT^n \times \TT^{\ell} \rightarrow \RR^{2n \times n}$,
the normal frame $N:\TT^n \times \TT^{\ell} \rightarrow \RR^{2n \times n}$ 
is constructed as follows:
\begin{equation}\label{eq:N}
N(\theta,\varphi):= L(\theta,\varphi) A(\theta,\varphi) + \tilde N(\theta,\varphi)\,,
\end{equation}
where,
\begin{align}
\tilde N(\theta,\varphi)={} & N^0(\theta,\varphi) B(\theta,\varphi)\,, \label{eq:tildeN} \\
N^0(\theta,\varphi)={} & J(K(\theta,\varphi)) L(\theta,\varphi) \,, \label{eq:N0} \\
B(\theta,\varphi)={}& (L(\theta,\varphi)^\top G(K(\theta,\varphi)) L(\theta,\varphi))^{-1}\,, \label{eq:B} \\
A(\theta, \varphi)={}& 
\begin{cases} -\displaystyle\frac{1}{2} (\tilde N(\theta,\varphi)^\top \Omega(K(\theta,\varphi))  \tilde N(\theta,\varphi)), & 
\text{if \textbf{Case II;}} \\
O_n\,. & \text{if \textbf{Case III.}} 
\end{cases}
\, \label{eq:A} 
\end{align}
The torsion matrix $T:\TT^n \times \TT^{\ell} \to\RR^{n\times n}$, 
given by 
\begin{align}
T(\theta, \varphi)={}& 
\begin{cases} - \dfrac{1}{2}\tilde N(\te,\vp)^\top \left(T_h(K(\te,\vp),\vp) +  T_h(K(\te,\vp),\vp)^\top\right)\tilde N(\te,\vp) & \text{if \textbf{Case II;}} \\ 
  \quad + \tilde N(\te,\vp)^\top T_h(K(\te,\vp),\vp)^\top N(\te,\vp)  & \\ 
  \quad + N(\te,\vp)^\top T_h(K(\te,\vp),\vp)\tilde N(\te,\vp),  & \\ 
  & \\ 
N(\te,\vp)^\top T_h(K(\te,\vp),\vp) N(\te,\vp)\,, & \text{if \textbf{Case III.}} 
\end{cases}
\, \label{eq:T} 
\end{align}
with, $T_\H:\mani \times \TT^{\ell} \to\RR^{2n\times 2n}$
\begin{align}
T_\H(z,\varphi)={}& 
\begin{cases} \Omega(z) \Big( \Dif_z Z_\H(z, \varphi) & \text{if \textbf{Case II;}} \\ 
  \quad -  \Dif_z J(z) \left[Z_\H(z, \varphi) \right] J(z)^{-1}& \\ 
  \quad - J(z) \Dif_z Z_\H(z, \varphi) J(z)^{-1} \Big), & \\ 
  & \\ 
\Omega(z) \Big( \Dif_z Z_\H(z, \vp)  & \text{if \textbf{Case III.}} \\
\quad + \Dif_z J(z) \left[Z_\H(z, \vp) \right] J(z)& \\
\quad + J(z) \Dif_z Z_\H(z , \vp) J(z) \Big)\,. &
\end{cases}
\, \label{eq:Th} 
\end{align}
\end{remark}
We can now stated our main theorem.

\section{A-posteriori KAM theory for quasi-periodic Hamiltonian systems}\label{sec:KAM:Thm}

In order to prove the KAM theorem for Lagrangian tori, 
it is necessary to establish a series of lemmas that will be the foundation for the proof. 
These lemmas will provide the conditions and properties required to demonstrate the persistence of invariant tori in a system. 
By stating these lemmas, 
we can lay the groundwork of the proof of the KAM theorem, 
which states that an invariant torus exists for certain dynamical and geometrical conditions.

\subsection{KAM theorem for Lagrangian torus}\label{ssec:theo}

\begin{theorem}[KAM theorem for quasi-periodic Hamiltonian systems]\label{theo:KAM}
Let us consider an exact symplectic structure $\sform=\dif \aform$ and 
a Riemannian metric $\gform$ on the open set $\mani\subset\RR^{2n}\times \TT^{\ell}$. 
Let 
$\H \in \Anal\left(\mani_\rho \times \TT^\ell_\rho \right)$ 
be a Hamiltonian function that depends quasi-periodically on time, 
with $1 \leq \ell \leq n$ and 
$K:\TT^n \times \TT^{\ell} \to \mani$ be a parameterization of an Lagrangian approximately invariant torus
with frequency vector $(\omega, \alpha ) \in \RR^n \times \RR^{\ell}$, 
furthermore consider the tangent frame 
$L:\TT^n \times \TT^{\ell} \to \RR^{2n\times n}$ given by
\begin{equation}\label{eq:L:theo}
L(\theta, \varphi)=
\Dif_{\theta} K(\theta, \varphi)\,.
\end{equation}
Then, we assume that the following hypotheses hold.
\begin{itemize}
\item [$H_1$]
The global objects can be analytically extended to the complex domain $\B \subset \CC^{2n}$, and there are constants that quantify the control of their analytic norms. 
For the geometric structures $\sform, \gform, \J$ in $\B$, 
the matrix functions
$\Omega,  G, J : \B \rightarrow \CC^{2n \times 2n}$
satisfy:
\begin{align*}
&\norm{\Omega}_{\B\times\TT^{\ell}} \leq \cteOmega\,,
&& \norm{\Dif_z \Omega}_{\B\times\TT^{\ell}} \leq \cteDOmega\,, 
&&  \\
&\norm{G}_{\B\times\TT^{\ell}} \leq \cteG\,,
&& \norm{\Dif_z G}_{\B\times\TT^{\ell}} \leq \cteDG\,, 
&& \norm{\Dif^2_z G}_{\B\times\TT^{\ell}} \leq \cteDDG\,,  \\
& \norm{J}_{\B\times\TT^{\ell}} \leq \cteJ\,, 
&& \norm{\Dif_z J}_{\B\times\TT^{\ell}} \leq \cteDJ\,,
&& \norm{\Dif^2_z J}_{\B\times\TT^{\ell}} \leq \cteDDJ\,, \\
& \norm{J^{\top}}_{\B\times\TT^{\ell}} \leq \cteJT\,, && \norm{\Dif_z J^\top}_{\B\times\TT^{\ell}} \leq \cteDJT \,,\\ 
& \norm{J^{-1}}_{\B\times\TT^{\ell}} \leq \cteJInv\,, \\ 
& \norm{J^{-\top}}_{\B\times\TT^{\ell}} \leq \cteJInvT\,. 
&& 
\end{align*}
For the Hamiltonian $\H: \B \times \TT^{\ell} \rightarrow \CC$ and its corresponding vector field $Z_\H: \B \times \TT^{\ell} \rightarrow \CC^{2n}$ and $T_\H:\mani \times \TT^{\ell} \to\CC^{2n}$ given by \eqref{eq:newTh},we have:
\begin{align*}
&  \norm{\Dif_z \H}_{\B\times\TT^{\ell}} \leq \cteDH\,,
&& \\
& \norm{Z_\H}_{\B\times\TT^{\ell}} \leq \cteXH\,, 
&& \norm{\Dif_z Z_\H}_{\B\times\TT^{\ell}} \leq \cteDXH\,, \\
& \norm{T_\H}_{\B\times\TT^{\ell}} \leq \cteTH\,, 
&& \norm{\Dif_z T_\H}_{\B\times\TT^{\ell}} \leq \cteDTH\,, \\
& \norm{T^\top_\H}_{\B\times\TT^{\ell}} \leq \cteTHT\,, 
&& \norm{\Dif_z T^\top_\H}_{\B\times\TT^{\ell}} \leq \cteDTHT\,, \\
& \norm{\Dif^2_z Z_\H}_{\B\times\TT^{\ell}} \leq \cteDDXH \,, 
&& \norm{\Dif_z Z_\H^\top}_{\B\times\TT^{\ell}} \leq \cteDXHT\,.\\
\end{align*}

\item [$H_2$] The parameterization $K$ is analytic in a complex strip 
$\TT^n_\rho\times\TT^{\ell}_\rho$, 
with $\rho>0$, 
which is contained in the global domain:
\[
    \dist (K(\TT^n_\rho \times \TT^\ell_\rho),\pd (\B \times \TT^\ell_\rho))>0.
\]
Moreover, 
the components of $K$ and $\Dif_{\theta} K$ belong to 
$\Anal(\TT^n_\rho \times \TT^\ell_\rho)$, 
and there are constants $\sigmaDK$ and $\sigmaDKT$ such that
\[
\norm{\Dif_{\theta} K}_{\rho} < \sigmaDK\,, \qquad
\norm{(\Dif_{\theta} K)^\top}_{\rho} < \sigmaDKT\,.
\]
\item [$H_3$] We assume that $L(\theta, \varphi)$ given by \eqref{eq:L:theo} 
has maximum rank for every $(\theta, \varphi) \in \bar\TT_\rho^n \times \bar\TT_\rho^\ell$, 
given. 
Moreover,  there exists a constant
$\sigmaB$ such that
\[
\norm{B}_\rho < \sigmaB,
\]
where $B(\theta, \varphi)$ is given by \eqref{eq:B}.
\item [$H_4$] 
There
exists a constant $\sigmaT$ such that
\[
\abs{\aver{T}^{-1}} < \sigmaT,
\]
where $T(\theta,\varphi)$ is given by \eqref{eq:T}.
\item [$H_5$] 
The frequency $(\omega, \alpha)$ belongs to $\Dioph{\gamma}{\tau}$, 
given by \eqref{eq:def:Dioph}, 
for certain $\gamma >0$ and $\tau \geq n + \ell -1$.
\end{itemize}
Under the above hypotheses, 
for each $0<\rho_\infty<\rho$ there exists a constant $\mathfrak{C}_1$ such that, 
if the error of invariance
\begin{equation}\label{eq:invE}
E(\theta,\varphi)=Z_\H(K(\theta,\varphi))+ \Lie{\omega, \alpha} K(\theta, \varphi), 
\end{equation}
satisfies
\begin{equation}\label{eq:KAM:HYP}
\frac{\mathfrak{C}_1 \norm{E}_\rho}{\gamma^4 \rho^{4 \tau}} < 1\,,
\end{equation}
then there exists an invariant torus 
$\torus_\infty = K_\infty(\TT^n\times \TT^{\ell})$ with frequencies $(\omega,\alpha)$, 
satisfying $K_\infty \in \Anal(\TT^{n}_{\rho_\infty}\times \TT^{\ell}_{\rho_\infty})$ 
and
\[
\norm{\Dif_{\theta} K_\infty}_{\rho_\infty} < \sigmaDK\,,
\qquad
\norm{(\Dif_{\theta} K_\infty)^\top}_{\rho_\infty} < \sigmaDKT\,, \]
\[
\dist(K_\infty(\TT^n_{\rho_\infty} \times \TT^{\ell}_{\rho_\infty}),\pd (\B \times \TT^\ell_\rho)) > 0\,.
\]
Furthermore, 
the objects are close to the original ones: 
there exist  a constant $\mathfrak{C}_2$ such that
\begin{equation}\label{eq:close}
\norm{K_\infty - K}_{\rho_\infty} < \frac{\mathfrak{C}_2 \norm{E}_\rho}{\gamma^2 \rho^{2\tau}}\,,
\end{equation}
The constants $\mathfrak{C}_1$ and $\mathfrak{C}_2$ are given explicitly in 
Appendix \ref{ssec:consts}.
\end{theorem}


\begin{remark}\label{rmk:canon-case}
In the \emph{canonical case}, we have $\Omega=\Omega_0$, $G = I_{2n}$,
and $J=\Omega_0$. Hence, we have 
$\cteOmega=1$,
$\cteDOmega=0$, 
$\cteG=1$,
$\cteDG=0$, 
$\cteDDG=0$,
$\cteJ=1$,
$\cteDJ=0$,
$\cteDDJ=0$,
$\cteJT=1$, and
$\cteDJT=0$.
\end{remark}

\section{Some lemmas to control approximate geometric properties}
\label{sec:lemmas}

In this section,
we include  the lemmas that are need to prove the main theorem.
%
\subsection{Estimates for cohomological equations}

The proof of the following lemma can be found in \cite{russmann1975optimal, russmann1976optimal}, 
as it is a recurring result in the context of KAM theorems. 
For refined estimates on the constant $c_R$ in terms of $\delta$, 
see \cite{figueras2018sharpness}.

\begin{lemma}[R\"ussmann estimates]\label{lem:Russmann}
Let $(\omega,\alpha) \in \Dioph{\gamma}{\tau}$ for some $\gamma>0$ and 
$\tau \geq n + \ell -1$. 
Then, 
for any $v \in \Anal(\TT^n_\rho\times\TT^{\ell}_\rho)$, 
with $\rho>0$, 
there exists a unique zero-average solution of $\Lie{\omega,\alpha} u = v -\aver{v}$, 
denoted by $u=\R{\omega, \alpha}v$. 
Moreover, 
for any $0<\delta<\rho$ we have that 
$u \in \Anal(\TT^n_{\rho-\delta} \times \TT^{\ell}_{\rho-\delta})$ and the estimate
\[
\norm{u}_{\rho-\delta} \leq \frac{c_R}{\gamma \delta^\tau}
\norm{v}_\rho\,,
\]
where $c_R$ is a constant that depends on $n+\ell$, $\tau$ and possibly on $\delta$.
\end{lemma}

\subsection{Approximate lagrangianity}

\begin{lemma}[Approximate lagrangianity]\label{lem:isotrop} 
Let us consider the setting of Theorem \ref{theo:KAM}. 
    Let us consider $\Omega_L:\mani \times \TT^{\ell} \to \RR^{2n \times 2n}$, 
the matrix representation of the pull-back on $K_\varphi^*\sform$ . 
We have
\begin{equation}\label{eq:OmegaK:aver}
\aver{\Omega_L} = O_n\,,
\end{equation}
and the following estimate holds:
\begin{equation}\label{eq:estOmegaK}
\norm{\Omega_L}_{\rho-2\delta} \leq \frac{\COmegaL }{\gamma \delta^{\tau+1}}\norm{E}_\rho\,,
\end{equation}
where the constant $\COmegaL$ is provided in Table \ref{tab:constants:all}.
\end{lemma}

\begin{proof}
    To prove that $\aver{\Omega_L} = O_n$, 
    we recall that $\Omega_{L}(z, \varphi)$ is the matrix representation of 
    $K^*_\vp \sform = K^*\sform = \mathrm{d}(K^*\aform)$, 
    so
    \begin{equation*}
        \begin{split}
            \Omega_{L}(K(\theta,\varphi), \varphi) & = \left(D_{\theta}K(\theta, \varphi) \right)^{\top} \Omega\left( K(\theta, \varphi)\right)D_{\theta}K(\theta, \varphi)  \\
            & = \left( D_{\theta}(a(K(\theta, \varphi)))\right)^{\top} D_{\theta}K(\theta, \varphi) -  \left( D_{\theta}K(\theta, \varphi) \right)^{\top} D_{\theta}(a(K(\theta, \varphi)))  \\
        \end{split}
    \end{equation*} 
    and so, 
    the components of 
    $\Omega_{L}(\theta, \varphi)$ are 
    \begin{equation*}
        \begin{split}
        \left( \Omega_L\right)_{i,j} (K(\theta,\varphi), \varphi)  = 
            \displaystyle \sum_{m = 1}^{2n} & \left(\dfrac{\partial }{\partial \theta_{i}} \left(a_{m}(K(\theta, \varphi)) \dfrac{\partial K_{m}(\theta, \varphi)}{\partial \theta_{j}} \right) \right. \\
 &- \left. \dfrac{\partial}{\partial \theta_{j}}\left( a_{m}(K(\theta, \varphi))\dfrac{\partial K_{m}(\theta, \varphi)}{\partial\theta_{i}} \right) \right).
            \end{split}
    \end{equation*} 
    We observe that the components of 
    $\Omega_L(K(\theta,\varphi), \varphi)$ 
	consist of sums of derivatives of periodic functions with respect to the variable 
	$\theta$,
    therefore $\aver{\Omega_L}=O_n$.
    We can now proceed to prove inequality \eqref{eq:estOmegaK}. 
    To do this, 
    we consider the approximately invariant case, 
    using the invariant error given by equation \eqref{eq:invE}.
    \begin{equation*}
        \Lie{\omega, \alpha}\left( \Omega (K(\theta, \varphi)) \right) 
            =  D_{z} \Omega(K(\theta, \varphi)) \left[ E(\theta,\varphi) - Z_{h}(K(\theta, \varphi), \varphi) \right]\,,
    \end{equation*}

    \begin{equation}
        \Lie{\omega, \alpha}(D_{\theta}K(\theta, \varphi)) = D_{\theta}E(\te, \vp) - D_z Z_{h}(K(\te, \vp), \vp)D_\te K(\te,\vp),
        \label{eqn:LiesDK}
    \end{equation}
    Replacing the last expressions $\Lie{\omega, \alpha} \left( \Omega_L(\theta, \varphi)\right)$, 
    we have,
    \begin{equation*}   
        \begin{split}
            \Lie{\omega, \alpha}(\Omega_L(\theta, \varphi)) &= D_{\theta}(E(\theta, \varphi))^{\top} \Omega(K(\theta, \varphi))D_{\theta}K(\theta, \varphi)\\
            &+ \left( D_{\theta}K(\theta, \varphi) \right)^{\top}D_{z}\Omega(K(\theta,\varphi))\left[ E(\theta, \varphi) \right] D_{\theta}K(\theta, \varphi)\\
            & + \left( D_{\theta}K(\theta, \varphi) \right)^{\top}\Omega(K(\theta,\varphi)) D_{\theta}E(\theta, \varphi) \,. \\
        \end{split}
    \end{equation*}
    The expression 
    $\Lie{\omega, \alpha}\Omega_L(\theta, \varphi)$, 
    is controlled using $H_{1}$, $H_{2}$, 
    the Banach algebra properties and Cauchy estimates as follows

    \begin{equation}
        \begin{split}
            \Vert \Lie{\omega, \alpha} \Omega_L \Vert_{\rho - \delta} 
            & \leq \Vert \left( D_{\theta} E\right)^\top \Vert_{\rho - \delta} \Vert \Omega \Vert_{\B\times\TT^\ell} \Vert D_{\theta} K \Vert_{\rho} \\
            & + \Vert \left( D_{\theta} K\right)^{\top}\Vert_{\rho} \Vert D_{z} \Omega \Vert_{\B\times\TT^\ell} \Vert E  \Vert_{\rho} \Vert D_{\theta} K \Vert_{\rho} \\
            & + \Vert \left( D_{\theta} K \right)^{\top}\Vert_{\rho} \Vert \Omega \Vert_{\B\times\TT^\ell} \Vert D_{\theta} E \Vert_{\rho - \delta} \\ 
            & \leq \dfrac{2n}{\delta}\Vert E \Vert_{\rho} \Vert \Omega \Vert_{\B\times\TT^\ell} \Vert D_{\theta} K \Vert_{\rho} \\
           & + \Vert  \left( D_{\theta} K\right)^{\top}\Vert_{\rho} \Vert D_{z} \Omega \Vert_{\B\times\TT^\ell} \Vert E  \Vert_{\rho} \Vert D_{\theta} K \Vert_{\rho} \\
            & + \dfrac{n}{\delta}\Vert \left( D_{\theta} K\right)^{\top}\Vert_{\rho} \Vert \Omega \Vert_{\B\times\TT^\ell} \Vert E \Vert_{\rho} \\ 
            & \leq \dfrac{2n  c_{\Omega, 0} \sigmaDK + \sigmaDKT c_{\Omega, 1} \sigmaDK \delta  + n \sigmaDKT c_{\Omega, 0}}{\delta}  \Vert E \Vert_{\rho} 
            =:\dfrac{ C_{\Lie{} \Omega_L } } {\delta} \Vert E \Vert_{\rho} \,.\\ 
        \end{split}
        \label{eq:CLieOK}
    \end{equation}
    Using Ru\"ssmann estimates, 
    we obtain
    \begin{equation}
        \Vert \Omega_L \Vert_{\rho - 2\delta} \leq \dfrac{c_{R}}{\gamma\delta^{\tau}}\Vert \Lie{\omega, \alpha} \Omega_L\Vert_{\rho - \delta} \leq \dfrac{c_{R} C_{\Lie{} \Omega_L}}{\gamma \delta^{\tau + 1}} \Vert E  \Vert_{\rho} 
        =: \dfrac{C_{\Omega_L}}{\gamma \delta^{\tau + 1}}\Vert E \Vert_{\rho}.
        \label{eq:COK}
    \end{equation}
    \end{proof}
\begin{remark}
    If the symplectic form depends explicitly on the $\varphi$ coordinates, 
    then the Lie's derivative of the reduced form is
    \[
        \Lie{\omega,\alpha}\left( \Omega\left(K(\theta,\varphi), \varphi \right)\right) = \Dif_z \Omega\left(K(\theta,\varphi), \varphi \right)\Lie{\omega,\alpha} K(\theta, \varphi) -\Dif_\varphi \Omega\left(K(\theta,\varphi), \varphi \right)\alpha \,, 
    \]
    the additional term is an obstruction in proving in general that the Lie's derivative of the reduced form is zero.
    Since,
    \[
        \Lie{\omega,\alpha}\left( \Omega_L\left(K(\theta,\varphi), \varphi \right)\right) 
        = -\left(\Dif_\theta K(\theta,\varphi) \right)^\top \left( \Dif_\varphi \Omega\left(K(\theta,\varphi), \varphi \right)\alpha \right)\Dif_\theta K(\theta,\varphi)  \,, 
    \]
    this dependence of the fibered symplectic structure on $\varphi$ suggests a more general case study that could be explored in the future.
\end{remark}
\subsection{Approximate symplectic frame}\label{ssec:symp}
In this section, 
we demonstrate that if $\torus$ is an approximately invariant torus, 
then we can construct an adapted frame that is approximately symplectic. 
The first step of this process is finding an approximate Lagrangian bundle that contains the tangent bundle $T_{\torus_\vp} \mani$ for $\vp \in \TT^\ell$.
\begin{lemma}\label{lem:Lang}
    With the hypotheses of Theorem \ref{theo:KAM}, 
    the map
    $L(\theta, \varphi)$ given by \eqref{eq:L:theo} satisfies
    \begin{equation}\label{eq:propL}
        \norm{L}_\rho \leq \CL\,,
        \qquad
        \norm{L^\top}_\rho \leq \CLT\,,
        \qquad
        \aver{L^\top (\Omega \circ K) E}=0_n\,,
    \end{equation}
    Furthermore, the object
    \begin{align}
        G_L(\theta, \varphi) := L(\theta, \varphi)^\top G(K(\theta, \varphi)) L(\theta, \varphi) \,, \label{eq:def:GL} 
    \end{align}
    is controlled as
    \begin{equation}\label{eq:est:ObjL}
        \norm{G_L}_\rho \leq \CGL\,, 
    \end{equation}
    The above constants are given explicitly in Table \ref{tab:constants:all}.
\end{lemma}
\begin{proof}
    Estimating $L(\theta, \varphi)$ and $L(\theta, \varphi)^\top$ is straightforward since the torus is Lagrangian, 
    as indicated in \eqref{eq:L:theo}.
    \begin{equation}
        \norm{L}_{\rho} = \norm{D_{\theta}K}_{\rho} < \sigmaDK =: \CL  
        \label{eq:CL}
    \end{equation}
    \begin{equation}
        \norm{ L^{\top}}_{\rho} = \norm{(D_{\theta}K)^{\top}}_{\rho} < \sigmaDKT =: \CLT  
        \label{eq:CLT}
    \end{equation}
    To derive the final expression in equation \eqref{eq:propL}, 
    we perform the following computation,
    \begin{equation*}
        \begin{split}
            L(\theta, \varphi)^{\top} & \Omega(K(\theta, \varphi))E(\theta, \varphi) \\ &=  D_{\theta}K(\theta, \varphi)^{\top}\Omega(K(\theta, \varphi))\left( Z_{h} \left( K(\theta, \varphi), \varphi  \right) + \Lie{\omega, \alpha} \left( K(\theta, \varphi)  \right) \right) \\
            &= \left( D_{\theta}\left( h\left(K(\theta, \varphi), \varphi \right)  \right) \right)^{\top} + D_{\theta}K(\theta, \varphi)^{\top}\Omega(K(\theta, \varphi))  \Lie{\omega, \alpha} \left( K(\theta, \varphi)  \right) \\ 
            &= \left( D_{\theta}\left( h\left(K(\theta, \varphi), \varphi \right)  \right) \right)^{\top} - \Omega_{L}\left( \theta, \varphi \right)\omega \\
            & \hspace{0.5cm} - D_{\theta}K(\theta, \varphi)^{\top}\Omega(K(\theta, \varphi))D_{\varphi}K(\theta, \varphi)\alpha \,. \\  
        \end{split}
        \label{eqn:LOE}
    \end{equation*}
    Where we have used \eqref{eq:invE} and \eqref{def-OL}. 
    We will now compute the average of the components of the three-term in the last expression. 
    We analyze the components of the first term for an arbitrary $\varphi$.
    \begin{equation*}
        \begin{split}
            \tiny
            \left( \langle  D_{\theta}\left( h\left(K(\theta, \varphi), \varphi \right)  \right)  \rangle_{\theta} \right)_{i,j} &=  \displaystyle \displaystyle \int_{0}^{T} \dfrac{\partial }{\partial \theta_{i}}  \left( h\left( K(\theta, \varphi) , \varphi \right) \right)  d\theta_{i} \,  = 0 \,,\\ 
        \end{split}
    \end{equation*}
    In the last equality, 
    we use the periodicity of $K$. 
    Therefore the average of the first term is zero. 
    From lemma \eqref{lem:isotrop},
    we obtain that $\aver{\Omega_L}=0$.
    We define 
    $\check{\Omega}_{K_{\varphi}}(\theta, \varphi) := D_{\theta}K(\theta, \varphi)^{\top}\Omega(K(\theta, \varphi))D_{\varphi}K(\theta, \varphi)$ 
    to prove that the last term has zero average. 
    By proceeding in an analogous way to $\Omega_L(\theta, \varphi)$, 
    we obtain that the components of $\check{\Omega}_{K}(\theta, \varphi)$ are given by
    \begin{equation*}
        \begin{split}
            \left( \check{\Omega}_{K} \right)_{i, j} & = \displaystyle \sum_{m=1}^{2n}\left( \dfrac{\partial(a_{m}(K(\theta, \varphi)))}{\partial \theta_{i}} \dfrac{\partial K_{m}(\theta, \varphi)}{\partial \varphi_{j}} - \dfrac{\partial(a_{m}(K(\theta, \varphi)))}{\partial \varphi_{j}} \dfrac{\partial K_{m}(\theta, \varphi)}{\partial \theta_{i}} \right) \\
            & = \displaystyle \sum_{m = 1}^{2n}\left(\dfrac{\partial }{\partial \theta_{i}} \left(a_{m}(K(\theta, \varphi)) \dfrac{\partial K_{m}(\theta, \varphi)}{\partial \varphi_{j}} \right) - \dfrac{\partial}{\partial \varphi_{j}}\left( a_{m}(K(\theta, \varphi))\dfrac{\partial K_{m}(\theta, \varphi)}{\partial\theta_{i}} \right) \right).
        \end{split}
    \end{equation*}
    therefore, 
    \begin{equation*}
        \begin{split}
            \left( \aver{\check{\Omega}_{K} } \right)_{i,j} &= \displaystyle \int_{0}^{T}  \displaystyle \int_{0}^{T}  \displaystyle \sum_{m = 1}^{2n}\left(\dfrac{\partial }{\partial \theta_{i}} \left(a_{m}(K(\theta, \varphi)) \dfrac{\partial K_{m}(\theta, \varphi)}{\partial \varphi_{j}} \right) \right. \\
 & \qquad \qquad \qquad - \left. \dfrac{\partial}{\partial \varphi_{j}}\left( a_{m}(K(\theta, \varphi))\dfrac{\partial K_{m}(\theta, \varphi)}{\partial\theta_{i}} \right) \right) d\theta_{i} \, d\varphi_{j} \\
            &=  \displaystyle \sum_{m = 1}^{2n}\left( \displaystyle \int_{0}^{T}  \displaystyle \int_{0}^{T}  \dfrac{\partial }{\partial \theta_{i}} \left(a_{m}(K(\theta, \varphi)) \dfrac{\partial K_{m}(\theta, \varphi)}{\partial \varphi_{j}} \right) \right. \\
& \qquad \qquad \quad - \left. \dfrac{\partial}{\partial \varphi_{j}}\left( a_{m}(K(\theta, \varphi))\dfrac{\partial K_{m}(\theta, \varphi)}{\partial\theta_{i}} \right) \right) d \theta_{i} \,d\varphi_{j}\\
            & = 0
        \end{split}
    \end{equation*}
    The last equality holds because 
	$K(\theta, \varphi)$ is periodic with respect to $\theta$ and $\varphi$. 
    Using the above, 
    we obtain the following expression  
    \[
        \aver{D_{\theta}K(\theta, \varphi)^{\top}\Omega(K(\theta, \varphi))D_{\varphi}K(\theta, \varphi) }= O_{n}.
    \]
    Finally, 
    by using the definition of $G_{L}$ 
    and the estimates for $L$ and $L^\top$, 
    we conclude
    \begin{equation}
        \norm{G_{L}}_{\rho} \leq \norm{L^{\top}}_{\rho} \norm{ G }_{\B\times\TT^\ell} \norm{L}_{\rho} \leq \CLT c_{G,0} \CL =: C_{G_{L}} \,.
        \label{eq:CGL}
    \end{equation}
    \end{proof}
    The following lemma demonstrates that the geometric constructions detailed 
    in Section~\ref{ssec:sym:frame} enable us to obtain an approximately symplectic frame 
    attached to an approximately invariant torus.
    \begin{lemma}\label{lem:sympl}
        With the hypotheses of Theorem \ref{theo:KAM}, 
        the map $N : \TT^n \times \TT^\ell \to \RR^{2n \times n}$ given by \eqref{eq:N} satisfies
        \begin{equation}\label{eq:propN}
            \norm{N}_\rho \leq \CN\,,
            \qquad
            \norm{N^\top}_\rho \leq \CNT\,,
        \end{equation}
        and the map 
        $P: \TT^n \times \TT^\ell \to \RR^{2n \times 2n}$ 
        given by
        \[
            P(\theta, \varphi)=
            \begin{pmatrix}
                L(\theta, \varphi) & N(\theta, \varphi)
            \end{pmatrix}\,,
        \]
        induces an approximately symplectic vector bundle isomorphism, i.e.,
        the error map
        \begin{equation}\label{eq:Esym}
            \Esym(\theta, \varphi) := P(\theta, \varphi)^\top \Omega(K(\theta, \varphi)) P(\theta, \varphi)-\Omega_0\,,
            \qquad
            \Omega_0 = 
            \begin{pmatrix}
                O_n & -I_n \\
                I_n & O_n
            \end{pmatrix}\,,
        \end{equation}
        is small in the following sense:
        \begin{equation}\label{eq:normEsym}
            \norm{\Esym}_{\rho-2 \delta} \leq \frac{\Csym}{\gamma \delta^{\tau+1}}
            \norm{E}_\rho\,.
        \end{equation}
        The above constants are given explicitly in Table \ref{tab:constants:all}.
    \end{lemma}
    \begin{proof}
        We shall consider equation \eqref{eq:N0} and the hypothesis $H_{1}$ to prove this result. 
        \begin{equation}
            \norm{N^{0}}_\rho \leq \norm{J\circ K}_{\rho} \norm{L}_{\rho} \leq  \norm{J}_{\B\times\TT^\ell} \norm{L}_{\rho} \leq \cteJ \CL =: \CNO\,,
            \label{eq:CNO}
        \end{equation}
        \begin{equation}
            \norm{ \left( N^{0} \right)^{\top} }_\rho  \leq   \Vert L^{\top}\Vert_{\rho} \Vert J^{\top} \Vert_{\B\times\TT^\ell}   \leq \CLT c_{J^{\top},0} =: \CNOT\,,
            \label{eq:CNOT}
        \end{equation}
        \begin{equation}\label{eq:CtildeN}
            \norm{\tilde N}_{\rho} \leq   \norm{N^0}_\rho   \norm{B}_\rho \leq  \CNO \sigmaB =: \CTildeN \,,
        \end{equation}
        and
        \begin{equation}\label{eq:CtildeNT}
            \norm{\tilde N^\top}_{\rho} \leq   \norm{B}_\rho  \norm{(N^0)^\top}_\rho  \leq \sigmaB \CNOT  =: \CTildeNT \,.
        \end{equation}
We use \eqref{eq:CtildeN} and \eqref{eq:CtildeNT}, to estimate $A$.
\begin{equation}
\Vert A \Vert_{\rho} \leq \dfrac{1}{2} \norm{\tilde N^\top}_\rho \norm{ \Omega \circ K  }_\rho \norm{ \tilde N }_\rho \leq \frac{1}{2} \CTildeNT \cteOmega\CTildeN =: C_{A}\,.
\label{eq:CA}
\end{equation}
Now, 
we compute the norm of $N(\theta, \varphi)$ using $H_{3}$ from equations \eqref{eq:B} and \eqref{eq:A}.  
We used the fact that $B(\theta,\varphi)$ is symmetric, and $A(\theta,\varphi)$ is antisymmetric. 
These yields the following bounds
\begin{equation}
\Vert N \Vert_{\rho} \leq \Vert  L \Vert_{\rho} \Vert  A \Vert_{\rho} + \norm{  \tilde N }_{\rho} \leq \CL\CA + \CTildeN =:C_{N},
\label{eq:CN}
\end{equation}

\begin{equation}
\Vert N^{\top} \Vert_{\rho} \leq  \Vert A^{\top} \Vert_{\rho} \Vert L^{\top} \Vert_{\rho}  +  \norm{ \tilde N^\top}_{\rho}  \leq  \CA\CLT + \CTildeNT =:C_{N^{\top}}.
\label{eq:CNT}
\end{equation}
Observe that,
\begin{align}
L(\theta, \varphi)^\top \Omega(K(\theta, \varphi)) & N(\theta, \varphi) \nonumber \\
= {} & \Elag( K(\theta, \varphi)), \varphi) A(\theta, \varphi) - L(\theta, \varphi)^\top G(K(\theta, \varphi)) L(\theta, \varphi)B(\theta, \varphi)
\nonumber \\
= {} & \Elag(K(\theta, \varphi)), \varphi) A(\theta, \varphi) - I_n\,, \label{eq:LON}
\end{align}
where we use that 
$J^\top \Omega( K(\theta, \varphi)) , \varphi)=G(K(\theta, \varphi)) , \varphi)$ and \eqref{eq:B}.
We also have from \eqref{eq:A},
\begin{align}
N(\theta, \varphi)^\top & \Omega(K(\theta, \varphi)) N(\theta, \varphi) \nonumber \\
= {} & A(\theta, \varphi)^\top \Elag(K(\theta, \varphi)), \varphi) A(\theta, \varphi)+A(\theta, \varphi)-A(\theta, \varphi)^\top \nonumber\\
 {} & \hspace{0.5cm} + B(\theta, \varphi)^\top L(\theta, \varphi)^\top \tilde \Omega(K(\theta, \varphi))
  L(\theta, \varphi) B(\theta, \varphi) \nonumber \\
= {} & A(\theta, \varphi)^\top \Elag(K(\theta, \varphi)), \varphi) A(\theta, \varphi) \,. \label{eq:NON}
\end{align}
We compute the error of the symplectic character of the frame, 
\begin{equation}
\begin{split}
E_{sym}(\theta, \varphi) &=   P(\theta, \varphi)^{\top} \Omega(K(\theta, \varphi))P(\theta, \varphi) - \Omega_{0} \\
&=   \begin{pmatrix}
L(\theta, \varphi)^{\top}\Omega\left( K(\theta, \varphi) \right)L(\theta, \varphi)             & L(\theta, \varphi)^{\top}\Omega\left( K(\theta, \varphi) \right)N(\theta, \varphi) + I_{n}\\
N(\theta, \varphi)^{\top}\Omega\left( K(\theta, \varphi) \right)L(\theta, \varphi) - I_{n}  & N(\theta, \varphi)^{\top}\Omega\left( K(\theta, \varphi) \right)N(\theta, \varphi)
\end{pmatrix} \\
&= \begin{pmatrix}
\Omega_{L}( K(\theta, \varphi)), \varphi) & \Omega_{L}(K(\theta, \varphi)) , \varphi)A(\theta, \varphi) \\
A(\theta, \varphi)^{\top}  \Omega_{L}(K(\theta, \varphi)), \varphi) & A(\theta, \varphi)^{\top}  \Omega_{L}(K(\theta, \varphi)) , \varphi) A(\theta, \varphi)
\end{pmatrix} \,, \\
\end{split}
\label{eqn:Sym-Error-components}
\end{equation}
Therefore using \eqref{eq:estOmegaK}, 
the control for the symplectic error is
\begin{equation}
\Vert E_{sym} \Vert_{\rho - 2 \delta}  \leq \dfrac{(1 + C_{A}) \max\{ 1, C_{A} \}C_{\Omega_{L}} }{\gamma \delta^{\tau +1}} \Vert E \Vert_\rho =: \dfrac{C_{sym}}{\gamma \delta^{\tau + 1}} \Vert E \Vert_{\rho}\,,
\label{eq:Csym}
\end{equation}
which yeilds \eqref{eq:normEsym}.
\end{proof}
\begin{remark}
The above estimates can be readily adapted to \textbf{Case III}, for which $A=O_n$. 
In this case we have,
\[
N(\te,\vp) = N^0(\te,\vp) B(\te,\vp) 
\,,
\qquad
\Esym(\theta, \varphi)
= 
\begin{pmatrix}
\Elag(\te,\vp) & O_n \\
O_n &
B(\te,\vp)^\top \Elag(\te,\vp) B(\te,\vp)
\end{pmatrix}\,.
\]
The corresponding estimates
are given explicitly in Table \ref{tab:constants:all}.
\end{remark}

\subsection{Control of the torsion matrix}
In this section, 
our objective is controlling the torsion matrix 
$T(\te,\vp)$ as defined in equation \eqref{eq:newT}. 

Although this expression may initially appear cumbersome, 
it does not rely on the Lie derivative. 
Instead, 
it depends solely on easily controllable geometric and dynamical objects. 

\begin{lemma}\label{lem:Newtwist}
With the hypotheses of
Theorem \ref{theo:KAM},
the torsion matrix $T(\theta,\varphi)$, given by \eqref{eq:newT}, 
has components in $\Anal(\TT^{n}_{\rho} \times \TT^{\ell}_{\rho} )$ and
satisfies
the estimate
\[
\norm{T}_\rho \leq \CT\,,
\]
where the constant $\CT$ is provided in Table \ref{tab:constants:all}.
\end{lemma}
\begin{proof}\label{lem:twist}
In this proof, 
we use hypothesis $H_1$ from Theorem \ref{theo:KAM}
and estimates of the quantity $\tilde N$, 
given in equation \eqref{eq:CtildeN} and \eqref{eq:CtildeNT}. 

Using all of the above and the torsion expression provided in equation \eqref{eq:newT}, 
we obtain that,
\begin{equation}\label{eq:CnewT}
\begin{split}
\norm{T}_\rho \leq & \dfrac{1}{2}\left(\norm{\tilde N^\top}_\rho \norm{T_h}_\rho \norm{\tilde N}_\rho  +\norm{\tilde N^\top}_\rho \norm{T_h^\top}_\rho \norm{\tilde N}_\rho\right) \\
& + \norm{\tilde N^\top}_\rho \norm{T_h^\top}_\rho \norm{N} + \norm{N^\top}_\rho \norm{T_h}_\rho \norm{\tilde N}_\rho \\
\leq & \dfrac{1}{2}\left(\CTildeNT \cteTH \CTildeN +\CTildeNT \cteTHT \CTildeN  \right) \\
& + \CTildeNT \cteTHT \CN  + \CNT \cteTH \CTildeN  =:  \CT \,,
\end{split}
\end{equation}
which completes the proof.
\end{proof}

\begin{remark}
    For Case III, the constants in \eqref{eq:CnewT} are: $\CTildeN = \CN$, $\CTildeNT = \CNT$, $\CA =0$ and $\CT = \CNT\cteTH\CN$.
\end{remark}

\begin{remark}
Upon comparing Lemma \ref{lem:Newtwist} of our paper with Lemma 4.8 of  \cite{haro2019posteriori}, 
we see the advantages of using the definition \eqref{eq:newT} for torsion instead of \eqref{eq:OldT}, 
which is similar to that in \cite{haro2019posteriori}. 
The torsion matrix in,  \eqref{eq:newT}, belong to $\Anal(\TT^{n}_{\rho} \times \TT^{\ell}_{\rho}))$ , 
while the matrix in \eqref{eq:OldT} belongs to $\Anal(\TT^{n}_{\rho-\delta} \times \TT^{\ell}_{\rho-\delta})$. 
This indicates that first one is more regular than  second one, 
which directly affects the numerical implementation resulting from the proof presented in our paper.
\end{remark}
Expressions \eqref{eq:OldT} and \eqref{eq:newT} coincide if the torus is invariant. 
However, 
if the torus is approximately invariant,
this does not hold. 
To show this fact, 
we can construct an expression for 
$T(\theta, \varphi)$, in the approximately invariant case.

\begin{lemma}\label{lem:constT}
        With the hypotheses of 
        Theorem \ref{theo:KAM}, 
        the torsion matrix $T(\theta,\varphi)$, given by \eqref{eq:newT}, 
        has components in $\Anal(\TT^{n}_{\rho} \times \TT^{\ell}_{\rho} )$ and
        satisfies
        the estimate
        \[
            \norm{ N^\top\Omega \Loper{N} - T}_{\rho-2\de}\leq  \frac{ \CET }{\gamma \delta^{\tau+1}}  \norm{E}_\rho \,.
        \]
        where the constant $\CET$ is provided in Table \ref{tab:constants:all}.
\end{lemma}
    \begin{proof}
        We consider the invariance equation, 
        \eqref{eq:inv:err}, for an approximately invariant torus, 
        i.e.
        \begin{equation}\label{eq:ApproxLieK}
            \Lie{\omega, \alpha}K(\theta,\varphi)=E(\theta,\varphi) - Z_\H(K(\theta,\varphi) ,\varphi)\,,
        \end{equation}
        from which we obtain
        \begin{equation}\label{eq:ApproxLieDK}
            \Lie{\omega, \alpha}L(\theta,\varphi)= \Dif_{\theta} E(\theta,\varphi) - \Dif_z Z_\H(K(\theta,\varphi) ,\varphi)L(\theta,\varphi)\,.
        \end{equation}
        Using \eqref{eq:ApproxLieDK} the Lie derivative of $J$ is,
        \begin{equation}\label{eq:ApproxLieJ}
            \Lie{\omega, \alpha}J(K(\theta,\varphi))= \Dif_z J(K(\te, \vp))\left[E(\theta,\varphi) \right]- \Dif_z J(K(\theta,\varphi)) \left[Z_\H(K(\te,\vp),\vp)\right] \,.
        \end{equation}
        Therefore, 
        \begin{equation}\label{eq:ApproxLieOmega}
            \begin{split}
                \Lie{\omega, \alpha}\Omega(K(\te,\vp))= &\Dif_{z} \Omega\left( K(\te,\vp) \right)\left[E(\te,\vp)\right] \\
                &+ \left( \Dif_z Z_\H(K(\theta,\varphi) ,\varphi) \right)^\top \Omega(K(\te,\vp))  \\
                &+ \Omega(K(\te,\vp))  \Dif_z Z_\H(K(\theta,\varphi) ,\varphi) \,.
            \end{split}
        \end{equation}
        Taking the Lie derivative of $G$, using
        \eqref{eq:prop:struc1},
        \eqref{eq:ApproxLieJ},
        \eqref{eq:ApproxLieOmega}, 
        and \eqref{eqn:Jacobi-indentity} we have,
        \begin{equation}\label{eq:ApproxLieG}
            \begin{split}
                \Lie{\om, \al}G(K(\te,\vp))
                =& -\left( \Dif_z Z_\H(K(\theta,\varphi) ,\varphi)\right)^{\top}  \Omega(K(\te,\vp)) J(K(\theta,\varphi)) \\
                & -  \Omega(K(\te,\vp))  \Dif_z Z_\H(K(\theta,\varphi) ,\varphi) J(K(\theta,\varphi)) \\
                & +\Omega(K(\te,\vp))   \Dif_z J(K(\te,\vp))  \left[Z_\H(K(\theta,\varphi) ,\varphi)\right] \\
                & -\Dif_{z}\Omega(K(\te,\vp)) \left [E(\theta,\varphi) \right] J(K(\theta,\varphi)) \\ 
                & - \Omega(K(\te,\vp))   \Dif_z J(K(\te,\vp))  \left[E(\theta,\varphi)\right]  \,.
            \end{split}
        \end{equation}
        To calculate the Lie derivative for B, 
        we use that $I_{d} = B^{-1}(\te,\vp)B(\te,\vp)$, 
        \begin{align*}
            \Lie{\omega,\alpha} B(\te,\vp) 
            &= {}  B(\theta, \varphi) \Lie{\omega,\alpha} \left( L(\theta,\varphi)^\top \right)  \Omega(K(\theta,\varphi)) N(\theta, \varphi)^0  \nonumber \\
            & \quad - B(\theta, \varphi) L(\theta, \varphi)^\top \Lie{\omega,\alpha}\left( G(K(\theta,\varphi)) \right) L(\theta,\varphi) B(\theta, \varphi) \\
            & \quad + N(\theta,\varphi)^{0, \top} \Omega(K(\theta,\varphi)) \Lie{\omega,\alpha}\left(  L(\theta,\varphi) \right) B(\theta, \varphi) \,.
        \end{align*}
        Therefore, the Lie derivative of $B(\theta,\varphi)$ for an approximately invariant torus is,
        \begin{equation}\label{eq:AproxLieB}
            \begin{split}
                \Lie{\om,\al} B(\te,\vp)
                = & \tilde N(\te,\vp)^\top J(K(\te,\vp))^{-\top} T_h(K(\te,\vp),\vp) \tilde N(\te,\vp) + E_{\Lie{} B}(\te,\vp) \,,
            \end{split}
        \end{equation}
        with,
        \begin{equation}\label{eq:ELieB}
            E_{\Lie{} B}(\te,\vp)=  \tilde N(\te,\vp)^\top J(K(\te,\vp))^{-\top}T_E(\te,\vp) \tilde N(\te,\vp)\,,
        \end{equation}
        and,
        \begin{equation}\label{eq:Te}
            \begin{split}
                T_E(\te,\vp)=&  \Dif_{z}\Omega(K(\te,\vp)) \left [E(\theta,\varphi) \right]  \\ 
                &+ \Omega(K(\te,\vp))   \Dif_z J(K(\te,\vp))  \left[E(\theta,\varphi)\right]  J(K(\te,\vp))^{-1}  \\
                &+ J(K(\te,\vp))^{\top} \Omega(K(\te,\vp))  L(\te,\vp) B(\te,\vp) \left( \Dif_{\theta} E(\theta,\varphi) \right)^\top  \Omega(K(\te,\vp)) \\ 
                &+ J(K(\te,\vp))^{\top}  \Omega(K(\te,\vp))  \Dif_{\theta} E(\theta,\varphi) B(\theta,\varphi) L(\te,\vp)^{\top} \Omega(K(\te,\vp))  \,,
            \end{split}
        \end{equation}
        where we used equations \eqref{eq:ApproxLieDK}, 
        \eqref{eq:ApproxLieG}, 
        and $I_n = \tilde N(\te,\vp)^\top \Omega(K(\te,\vp))  L(\te,\vp)$.

        Following a similar approach to that used for $B(\te,\vp)$, 
        we determine the Lie derivative for $\tilde N(\te,\vp)$,
        \begin{equation}\label{eq:ApproxLieTildeN}
            \begin{split}
                \Lie{\om,\al} \tilde N(\te,\vp)
                = & \Big( -\Dif_z J(K(\te,\vp))\left[ Z_\H\left(K(\te, \vp),\vp \right) \right]J(K(\te,\vp))^{-1}\\ 
                & - J(K(\te,\vp))  \Dif_z Z_\H(K(\theta,\varphi) ,\varphi) J(K(\te,\vp))^{-1} \\
                & + J(K(\te,\vp)) L(\te,\vp) \tilde N(\te,\vp)^\top J(K(\te,\vp))^{-\top} T_\H(\te,\vp)\Big) \tilde N(\te,\vp)\\
                & + E_{\Lie{} \tilde N}(\te,\vp)\,,
            \end{split}
        \end{equation}
        where,
        \begin{equation}\label{eq:ELieTildeN}
            \begin{split}
                E_{\Lie{} \tilde N}(\te,\vp)=  & \Dif_z J(K(\te,\vp))\left[ E(\te, \vp) \right]L(\te,\vp) B(\te,\vp) + J(K(\te,\vp))  \Dif_\te E(\theta,\varphi) B(\te,\vp) \\
                & + J(K(\te,\vp)) L(\te,\vp)  E_{\Lie{} B}(\te,\vp) \,.
            \end{split}
        \end{equation}
        Now, we calculate $\Lie{\om,\al} A(\te,\vp)$,
        \begin{equation}\label{eq:ApproxLieA}
            \begin{split}
                \Lie{\om,\al} A(\te,\vp)
                & = \bar{\Lie{\omega,\alpha}A(\te,\vp)} + E_{\Lie{} A}(\te,\vp) \,,
            \end{split}
        \end{equation}
        with,
        \begin{equation}\label{eq:barLieA}
            \begin{split}
                & \bar{\Lie{\omega,\alpha}A(\te,\vp)} \\ &
                = \dfrac{1}{2}\tilde N(\te,\vp)^\top\Big( T_\H(\te,\vp)^\top - T_\H(\te,\vp)^\top J(K(\te,\vp))^{-1} \tilde N(\te,\vp) L(\te,\vp)^\top J(K(\te,\vp))^\top \Omega(K(\te,\vp)) \\
                & \quad - T_\H(\te,\vp) - \Omega(K(\te,\vp)) J(K(\te,\vp)) L(\te,\vp) \tilde N(\te,\vp)^\top  J(K(\te,\vp))^{-\top} T_\H(\te,\vp)\Big) \tilde N(\te,\vp)  \,,
            \end{split}
        \end{equation}
        and
        \begin{equation}\label{eq:ELieA}
            \begin{split}
                E_{\Lie{} A}(\te,\vp)= & - \dfrac{1}{2}\Big(E_{\Lie{} \tilde N}(\te,\vp)^\top \Omega(K(\te,\vp)) \tilde N(\te,\vp) + \tilde N(\te,\vp)^\top  \Dif_z \Omega(K(\theta,\varphi)) \left[E(\te,\vp)\right] \tilde N(\te,\vp) \\
                & \quad + \tilde N(\te,\vp)^\top  \Omega(K(\theta,\varphi)) E_{\Lie{} \tilde N}(\te,\vp)  \Big) \,.
            \end{split}
        \end{equation}
Using equations \eqref{eq:ApproxLieDK}, \eqref{eq:ApproxLieTildeN} and \eqref{eq:ApproxLieA}, 
we obtain that the Lie derivative of $N(\theta, \varphi)$ is,
\begin{equation}\label{eq:ApproxLieN}
    \begin{split}
        \Lie{\om,\al} N(\te,\vp)
        = & \dfrac{1}{2} \Dif_z Z_\H(K(\te,\vp) ,\vp) L(\te,\vp) \tilde N(\te,\vp)^\top \Omega(K(\te,\vp))  \tilde N(\te,\vp) \\
        & + \dfrac{1}{2} L(\te,\vp)\tilde N(\te,\vp)^\top \Big(T_\H(\te,\vp)^\top  - T_\H(\te,\vp)  \\ 
        & \quad - T_\H(\te,\vp)^\top  J(K(\te,\vp))^{-1} \tilde N(\te,\vp) L(\te,\vp)^\top J(K(\te,\vp))^\top\Omega(K(\te,\vp))  \\
        & \quad - \Omega(K(\te,\vp)) J(K(\te,\vp)) L(\te,\vp) \tilde N(\te,\vp)^\top J(K(\te,\vp))^{-\top} T_\H(\te,\vp) \Big)\tilde N(\te,\vp) \\
        & - \Big(\Dif_z J(K(\te,\vp))\left[ Z_\H\left(K(\te, \vp),\vp \right) \right]J(K(\te,\vp))^{-1}\\ 
        & \quad + J(K(\te,\vp)) \Dif_z Z_\H(K(\theta,\varphi) ,\varphi)J(K(\te,\vp))^{-1} \\
        & \quad - J(K(\te,\vp)) L(\te,\vp) \tilde N(\te,\vp)^\top J(K(\te,\vp))^{-\top} T_\H(\te,\vp)\Big) \tilde N(\te,\vp)\\ 
        & + E_{\Lie{} N}(\te,\vp) \,,
    \end{split}
\end{equation}
with,
\begin{equation}\label{eq:ELieN}
E_{\Lie{} N}(\te,\vp)=  E_{\Lie{} \tilde N}(\te,\vp) + L(\theta,\varphi)E_{\Lie{} A}(\te,\vp)  + \Dif_\te E(\te,\vp) A(\te,\vp) \,.
\end{equation}
Using the above construction in equation \eqref{eq:T}
and from the definition of torsion, equation  \eqref{eq:newT}, 
we determine the difference for an approximately invariant torus as follows
\begin{equation}\label{eq:ApproxT}
 N(\theta,\vp)^\top \Omega(K(\theta,\vp)) \Loper{N}(\theta,\vp)  = T(\te,\vp) + E_T(\te,\vp) \,.
\end{equation}
As we have done before, 
we define the error,
\begin{equation}\label{eq:ET}
\begin{split}
E_T(\te,\vp)= & \dfrac{1}{2} A(\te,\vp)^\top \Omega_L(\theta, \varphi) \tilde N(\te,\vp)^\top \left(T_h(\te,\vp)^\top - T_h(\te,\vp)\right)\tilde N(\te,\vp) \\
& + A(\te,\vp)^\top \Omega_L(\theta, \varphi)  \left(A(\te,\vp) T_h(\te,\vp) \tilde N(\te,\vp) - \tilde N(\te,\vp)^\top T_h(\te,\vp)^\top  L(\te,\vp) A(\te,\vp) \right) \\
& + \left( A(\te,\vp)^\top  L(\te,\vp)^\top + \tilde N(\te,\vp)^\top \right) \Omega(K(\te,\vp)) E_{\Lie{} N}(\te,\vp)  \,.
\end{split}
\end{equation}
We estimate the error in equation  \eqref{eq:ApproxT}.
Let us begin with $E_{\Lie{} N}(\te,\vp)$ and its associated terms. 
From equation \eqref{eq:Te}, 
we have 
\begin{equation}\label{eq:CTE}
\begin{split}
\norm{T_E}_{\rho-\de} 
\leq&  \norm{\Dif_{z}\Omega}_{\B\times\TT^{\ell}} \norm{E} _\rho + \norm{\Omega}_{\B\times\TT^{\ell}} \norm{ \Dif_z J}_{\B\times\TT^{\ell}}   \norm{J^{-1}}_{\B\times\TT^{\ell}} \norm{E}_\rho \\
&+ \dfrac{2n}{\de}\norm{J^{\top}}_{\B\times\TT^{\ell}} \norm{\Omega}_{\B\times\TT^{\ell}} \norm{L}_\rho \norm{B}_{\B\times\TT^{\ell}} \norm{\Omega}_{\B\times\TT^{\ell}} \norm{E}_\rho  \\ 
&+  \dfrac{n}{\de} \norm{J^{\top}}_{\B\times\TT^{\ell}} \norm{ \Omega}_{\B\times\TT^{\ell}} \norm{B}_{\B\times\TT^{\ell}} \norm{L^\top}_\rho \norm{\Omega}_{\B\times\TT^{\ell}}  \norm{E}_\rho  \\
    \leq& \dfrac{\cteDOmega \de +  \cteOmega  \cteDJ  \cteJInv \de  + 2n \cteJT  \CL \sigmaB  \cteOmega^2  +  n  \cteJT   \sigmaB \CLT  (\cteOmega)^2}{\de}  \norm{E}_\rho  \\
=:& \dfrac{\CTE}{\de}   \norm{E}_\rho  \,.
\end{split}
\end{equation}
Analogously for $T_E^\top$, $E_{\Lie{} B}$, $E_{\Lie{} B}^\top$, $E_{\Lie{} \tilde N}$, $E_{\Lie{} \tilde N}^\top$, $E_{\Lie{} A}$, $E_{\Lie{} N}$ and $E_T$, i.e.
\begin{equation}\label{eq:CTET}
\begin{split}
    \norm{T_E^\top}_{\rho-\de} \leq& \dfrac{\cteDOmega \de +  \cteJInvT \cteDJT  \cteOmega \de  + n  \sigmaB \CLT \cteJ \cteOmega^2  +  2n  \CL \sigmaB \cteJ  (\cteOmega)^2}{\de}  \norm{E}_\rho  \\
=:& \dfrac{\CTET}{\de}   \norm{E}_\rho  \,,
\end{split}
\end{equation}

\begin{equation}\label{eq:CELieB}
\norm{E_{\Lie{} B}}_{\rho-\de} 
\leq  \dfrac{\CTildeNT \cteJInvT \CTE \CTildeN}{\de}  \norm{E}_\rho  =: \dfrac{ \CELieB }{\de} \norm{E}_\rho   \,,
\end{equation}
\begin{equation}\label{eq:CELieBT}
\norm{E_{\Lie{} B}^\top}_{\rho-\de} 
\leq  \dfrac{\CTildeNT \CTET \cteJInv  \CTildeN}{\de}  \norm{E}_\rho  =: \dfrac{ \CELieBT }{\de} \norm{E}_\rho   \,,
\end{equation}
\begin{equation}\label{eq:CELieTildeN}
\begin{split}
\norm{E_{\Lie{} \tilde N}}_{\rho-\de} 
\leq  &  \dfrac{\de \cteDJ  \CL  \sigmaB + n \cteJ \sigmaB +\cteJ  \CL  \CELieB }{\de} \norm{E}_\rho =: \dfrac{ \CELietildeN}{\de}\norm{E}_\rho  \,,
\end{split}
\end{equation}
\begin{equation}\label{eq:CELieTildeNT}
\begin{split}
\norm{E_{\Lie{} \tilde N}^\top}_{\rho-\de} &\leq  \dfrac{\de \sigmaB \CLT \cteDJT  + 2n \sigmaB \cteJT  + \CELieBT \CLT \cteJT }{\de} \norm{E}_\rho \\
&=: \dfrac{ \CELietildeNT}{\de}\norm{E}_\rho  \,,
\end{split}
\end{equation}
\begin{equation}\label{eq:CELieA}
\begin{split}
\norm{E_{\Lie{} A}}_{\rho-\de} 
\leq &  \dfrac{1}{2}\left( \dfrac{ \CELietildeNT \cteOmega \CTildeN + \de\CTildeNT\cteDOmega\CTildeN + \CTildeNT\cteOmega \CELietildeN}{\de}\right)\norm{E}_\rho \\
 =: & \dfrac{ \CELieA}{\de}\norm{E}_\rho \,,
\end{split}
\end{equation}
\begin{equation}\label{eq:CELieN}
\begin{split}
\norm{E_{\Lie{} N}}_{\rho-\de} \leq &\norm{E_{\Lie{} \tilde N}}_{\rho-\de}  + \norm{L}_{\rho-\de}\norm{E_{\Lie{} A}}_{\rho-\de}  + \norm{\Dif_\te E}_{\rho-\de}\norm{A}_{\rho-\de} \\
\leq & \dfrac{ \CELietildeN + \CL \CELieA + n \CA }{\de}\norm{E}_\rho =: \dfrac{ \CELieN }{\de}\norm{E}_\rho \,, 
\end{split}
\end{equation}
and
\begin{equation}\label{eq:CET}
\begin{split}
\norm{E_T}_{\rho-2\de}
\leq & \dfrac{1}{2} \frac{\CA \COmegaL \CTildeNT \CTHT  \CTildeN}{\gamma \delta^{\tau+1}}  \norm{E}_\rho \\
& + \dfrac{1}{2} \frac{\CA \COmegaL \CTildeNT  \CTH  \CTildeN}{\gamma \delta^{\tau+1}}  \norm{E}_\rho \\
& + \frac{\CA \COmegaL \CA  \CTH  \CTildeN}{\gamma \delta^{\tau+1}}  \norm{E}_\rho \\
& + \frac{\CA \COmegaL \CTildeNT  \CTHT  \CTildeN \CL \CA}{\gamma \delta^{\tau+1} }  \norm{E}_\rho \\
& + \frac{\gamma \delta^\tau \CA \CLT \cteOmega \CELieN }{\gamma \delta^{\tau+1}}  \norm{E}_\rho \\
& + \frac{ \gamma \delta^\tau \CTildeNT  \cteOmega \CELieN }{\gamma \delta^{\tau+1}}  \norm{E}_\rho =:   \frac{ \CET }{\gamma \delta^{\tau+1}}  \norm{E}_\rho\,.
\end{split}
\end{equation}
To control equation \eqref{eq:ApproxT}, 
we use the previous estimates made previously, i.e.,
\begin{equation}\label{eq:CApproxT}
\norm{ N^\top\Omega \Loper{N} - T}_{\rho-2\de}\leq  \frac{ \CET }{\gamma \delta^{\tau+1}}  \norm{E}_\rho \,.
\end{equation}
Refer to Table \ref{tab:constants:all} for the latest set of constants.
\end{proof}

\subsection{Approximate reducibility}

The following lemma estimates the reducibility error obtained after replacing the derivative of the vector field at the approximately invariant torus in the linearized equation for an upper triangular matrix in a new symplectic coordinate system.

\begin{lemma}\label{lem:reduc}
With the hypotheses of 
Theorem \ref{theo:KAM}, 
the map
$P:\TT^d \times \TT^\ell \to \RR^{2n \times 2n}$, characterized in Lemma \ref{lem:sympl},
approximately reduces the linearized equation associated with the vector field
 $\Dif Z_\H \circ K$ to a
block-triangular matrix, i.e. the error map
\begin{equation}\label{eq:Ered}
\Ered (\theta,\varphi) :=
-\Omega_0 P(\theta,\varphi)^\top \Omega(K(\theta,\varphi))
\left(
\Dif_z Z_\H(K(\theta,\varphi)) P(\theta,\varphi)
+\Lie{\omega,\alpha}P(\theta,\varphi)
\right) - \Lambda(\theta,\varphi)\,,
\end{equation}
with
\begin{equation}\label{eq:Lambda}
\Lambda(\theta,\varphi)
= \begin{pmatrix}
O_n &  T(\theta,\varphi) \\ 
O_n  & O_n
\end{pmatrix}
\end{equation}
and $T(\theta,\varphi)$ is given by \eqref{eq:newT}, is small in the following sense:
\[
\norm{\Ered}_{\rho-2\delta} \leq \frac{\Cred}{\gamma \delta^{\tau+1}}
\norm{E}_\rho\,,
\]
where the constant $\Cred$ is provided in Table \ref{tab:constants:all}.
\end{lemma}
\begin{proof}

By employing the notation introduced in equation \eqref{eq:Loper}, 
we represent the block components of equation \eqref{eq:Ered}, 
denoted as $\Ered^{i,j}(\theta,\varphi)$, as follows:
\begin{align}
\Ered^{1,1}(\theta, \varphi) = {} & 
N(\theta, \varphi)^\top \Omega(K(\theta, \varphi)) \Loper{L}(\theta, \varphi) \,, \label{eq:Ered11} \\
\Ered^{1,2}(\theta, \varphi) = {} & 
N(\theta, \varphi)^\top \Omega(K(\theta, \varphi)) \Loper{N}(\theta, \varphi) - T(\theta, \varphi)  \,,
\label{eq:Ered12} \\
\Ered^{2,1}(\theta, \varphi) = {} & 
-L(\theta, \varphi)^\top \Omega(K(\theta, \varphi)) \Loper{L}(\theta, \varphi) \,, 
\label{eq:Ered21} \\
\Ered^{2,2}(\theta, \varphi) = {} & 
-L(\theta, \varphi)^\top \Omega(K(\theta, \varphi)) \Loper{N}(\theta, \varphi) \,.
\end{align}
In order to control the reducibility error, 
we estimate its components. 
We notice the presence of vector operators such as $\Loper{L}$ or $\Loper{N}$. 
Consequently, 
our next step involves estimating these operators. 
First, 
we look for an expression of the tangent vector bundle in terms of the invariance error.
We find that,
\begin{align}
 \Loper{L}(\theta, \varphi)  = {} & D_{z}Z_{h}\left( K(\theta, \varphi), \varphi \right)L(\theta, \varphi) +  \Lie{\omega, \alpha}  L(\theta, \varphi) \nonumber \\
  = {} & D_{\theta} E(\theta, \varphi).
\end{align}
Similarly, 
considering the transpose of the tangent vector bundle, 
we observe that
\begin{align}
 \Loper{L}^{\top}(\theta, \varphi)  = {} & L(\theta, \varphi)^{\top} D_{z}Z_{h}\left( K(\theta, \varphi), \varphi \right)^{\top} +  \Lie{\omega, \alpha}  L(\theta, \varphi)^{\top} \nonumber \\
  = {} & \left(D_{\theta} E(\theta, \varphi) \right)^{\top}\,.
\end{align}
Then, the objects, $\Loper{L}(\theta, \varphi)$ and $\Loper{L}(\theta, \varphi)^\top$ are controlled as follows
\begin{align}
\norm{\Loper{L}}_{\rho-\delta} \leq {} & \frac{n}{\delta} \norm{E}_\rho =: \frac{\CLoperL}{\delta} \norm{E}_\rho \,, \label{eq:CLoperL} \\
\norm{\Loper{L}^\top}_{\rho-\delta} \leq {} & \frac{2n}{\delta} \norm{E}_\rho =: \frac{\CLoperLT}{\delta} \norm{E}_\rho \,. \label{eq:CLoperLT}
\end{align}

Now, we take the Lie derivative of both sides on the expression \eqref{eq:LON}.
\[
\Lie{\omega, \alpha}\Big( L(\theta, \varphi)^\top \Omega(K(\theta, \varphi)) \Big) N(\theta, \varphi) +
L(\theta, \varphi)^\top \Omega(K(\theta, \varphi)) \left( \Lie{\omega,\alpha}N(\theta, \varphi) \right) =
\Lie{\omega, \alpha} (\Elag(\theta) A(\theta))\,.
\]
Substituting the last expression into $\Ered^{2,2}(\theta, \varphi)$,
and incorporating \eqref{eq:invE} as well as the geometric property \eqref{eqn:Jacobi-indentity}, 
we have the following,
\begin{align}
\Ered^{2,2}(\theta, \varphi) 
    {} &= L(\theta, \varphi)^{\top} D_{z}\Omega(K(\theta, \varphi)) \left[E(\theta,\varphi) \right] N(\theta, \varphi) + \Loper{L}(\theta,\varphi)^\top \Omega(K(\theta, \varphi)) N(\theta, \varphi) \nonumber \\
    {} & \quad  - \Lie{\omega, \alpha} (\Omega_{L}(\theta, \varphi))A(\theta, \varphi) - \Omega_L(\theta, \varphi)\Lie{\omega, \alpha} (A(\theta, \varphi))\,.
\label{eq:Ered22}
\end{align}
We observe that, 
to control $\Ered^{2,2}$,
we only need to estimate $\Omega_L \Lie{\omega, \alpha} A$.
That is,
\begin{equation}\label{eq:OmegaLCLieA}
\begin{split}
    \norm{\Omega_L \Lie{\om,\al}  A}_{\rho-2\delta} \leq & \norm{\Omega_L \Lie{\om,\al}  \bar{A} }_{\rho-2\delta} + \norm{\Omega_L E_{\Lie{}  A} }_{\rho-2\delta} \\
    \leq & \norm{\Omega_L}_{\rho-2\delta}\norm{\Lie{\om,\al}  \bar{A} }_{\rho-\delta} + \norm{\Omega_L}_{\rho-2\delta} \norm{E_{\Lie{}  A} }_{\rho-\delta} \\
    \leq & \dfrac{C_{\Omega_L} \CLieA + \gamma\delta^\tau\sigmaDKT\cteOmega\sigmaDK \CELieA }{\gamma \delta^{\tau + 1}}\norm{E}_{\rho} \,,
\end{split}
\end{equation}
where,
we use expression \eqref{eq:barLieA},
\begin{equation}\label{eq:CLieA}
\begin{split}
    \norm{\Lie{\om,\al}  \bar A}_{\rho-\delta}
    \leq &  \dfrac{1}{2}\left(\CTildeNT \cteTHT \CTildeN +  \CTildeNT \cteTHT \cteJInv \CTildeN \CLT \cteJT \cteOmega \CTildeN  
     \right. \\
    & {}  + \left. \CTildeNT \cteTH \CTildeN  +   \CTildeNT \cteOmega \cteJ \CL \CTildeNT \cteJInvT \cteTH \CTildeN \right) \\
    &=: \CLieA \,.
\end{split}
\end{equation}
We note that the component $\Ered^{1,2}(\te, \vp)$ coincides with equation \eqref{eq:ApproxT}, i.e.
\begin{equation}\label{eq:Ered12}
\Ered^{1,2}(\theta, \varphi) = E_T(\theta, \varphi)\,.
\end{equation}
Finally, we estimate the norms of the block components of $\Ered$
using the expressions \eqref{eq:Ered11}, \eqref{eq:Ered12}, \eqref{eq:Ered21} and \eqref{eq:Ered22},
and the previous estimates:
\begin{align}
\norm{\Ered^{1,1}}_{\rho-2\delta} \leq {} &
\frac{\CNT \cteOmega \CLoperL}{\delta} \norm{E}_\rho =: 
\frac{\Creduu}{\delta}\norm{E}_\rho
\,, \label{eq:Creduu} \\
\norm{\Ered^{1,2}}_{\rho-2\delta} \leq  {} &  \frac{ \CET }{\gamma \delta^{\tau+1}}  \norm{E}_\rho =:  \frac{\Credud}{\gamma \delta^{\tau+1}}  \norm{E}_\rho  \,,  \label{eq:Credud} \\
\norm{\Ered^{2,1}}_{\rho-2\delta} \leq {} & 
\frac{\CLT \cteOmega \CLoperL}{\delta} \norm{E}_\rho 
=:
\frac{\Creddu}{\delta}\norm{E}_\rho
\,, \label{eq:Creddu} \\
\norm{\Ered^{2,2}}_{\rho-2\delta} \leq {} &
\left( \CLT \cteDOmega \CN + \frac{\CLoperLT \cteOmega \CN }{\delta}
+ \frac{\CLieOmegaL \CA}{\delta} +\frac{\Clag \CLieA  + \gamma\delta^\tau\sigmaDKT\cteOmega\sigmaDK \CELieA}{\gamma \delta^{\tau+1}} \right)\norm{E}_\rho \nonumber \\
=: {} & \frac{\Creddd}{\gamma \delta^{\tau+1}} \norm{E}_\rho
 \,.
\label{eq:Creddd}
\end{align}
Then, we end up with
\begin{equation} \label{eq:CEred}
\norm{\Ered}_{\rho-2\delta} \leq 
\frac{\max \{
\Creduu \gamma \delta^\tau + \Credud
\, , \,
\Creddu \gamma \delta^\tau + \Creddd \} }{\gamma \delta^{\tau+1}}
\norm{E}_\rho 
=: \frac{\Cred}{\gamma \delta^{\tau+1}} \norm{E}_\rho\,,
\end{equation}
thus completing the proof.
\end{proof}
\section{Proof of the Lagrangian KAM theorem}\label{sec:proof:KAM}

In this section we present a proof of Theorem \ref{theo:KAM}.
For convenience, we start by outlining the scheme used
to correct the parameterization of the torus. That is,
in Section \ref{ssec:qNewton} we discuss the approximate solution of the
linearized equations in the symplectic frame constructed in Section 
\ref{ssec:symp}. This establishes a  method to
obtain a solution of the invariance equation. In Section
\ref{ssec:iter:lemmas} we produce quantitative estimates for
the objects obtained when performing one iteration of the
previous procedure. Finally, in Section
\ref{ssec:proof:KAM} we discuss the convergence of the quasi-Newton method.

\subsection{The quasi-Newton method}\label{ssec:qNewton}

As it is usual in the a-posteriori approach in KAM theory, the argument consists
in refining an approximate solution $K(\te, \vp)$, 
by means of a quasi-Newton method. 
Let us consider the equations associated with the invariance error \eqref{eq:invE}.

We obtain the new parameterization
$\bar K(\te,\vp)= K(\te,\vp)+\DeltaK(\te,\vp)$
by considering the linearized equation
\begin{equation}\label{eq:lin1}
\Dif Z_\H (K(\te,\vp),\vp) \DeltaK(\te,\vp) + \Lie{\om,\al}\DeltaK(\te,\vp) 
= - E(\te,\vp) \,.
\end{equation}
If we obtain a good enough approximation of the solution
$\DeltaK(\te,\vp)$
of \eqref{eq:lin1},
then $\bar K(\te,\vp)$ provides a parameterization
of an approximately invariant torus with frequency $(\omega, \alpha)$,
having a quadratic error in terms of $E(\te,\vp)$. 

To solve the linearized equation \eqref{eq:lin1}, 
we make use of the approximately symplectic
frame $P(\te,\vp)$, defined on the full tangent
space, which has been characterized in Section \ref{sec:lemmas}
(see Lemma \ref{lem:sympl}).
In particular, we introduce the linear change of coordinates,
\begin{equation}\label{eq:choice:DK}
\DeltaK (\te,\vp) = P(\te,\vp) \xi(\te,\vp) \,,
\end{equation}
where $\xi(\te,\vp)$ is the new unknown.
Taking into account this expression, 
the linearized equation becomes
\begin{equation} \label{eq:lin:1E} 
\begin{split}
\left( \Dif Z_\H (K(\te,\vp),\vp) P(\te,\vp) \right.+ & \left.\Lie{\om, \al}P(\te,\vp) \right) \xi(\te,\vp) \\ 
+ &P(\te,\vp) \Lie{\om,\al} \xi(\te,\vp) =   - E(\te,\vp) \,, 
\end{split}
\end{equation}
We now multiply both sides of \eqref{eq:lin:1E} by $-\Omega_0 P(\te,\vp)^\top
\Omega(K(\te,\vp))$, and we 
use the
geometric properties in Lemma \ref{lem:sympl} and
Lemma \ref{lem:reduc}, thus obtaining the equivalent equations:
\begin{equation} \label{eq:lin:1Enew}
\begin{split}
\left(\Lambda(\te,\vp)+\Ered(\te,\vp) \right) \xi(\te,\vp)
+ {} & \left( I_{2n} - \Omega_0 \Esym(\te,\vp)  \right) \Lie{\om, \al} \xi(\te,\vp) \\ 
= {} &
\Omega_0 P(\te,\vp)^\top \Omega(K(\te,\vp)) E(\te,\vp):=
\begin{pmatrix}
\eta^L(\te,\vp) \\
\eta^N(\te,\vp)
\end{pmatrix}
\,,
\end{split}
\end{equation}
$\Lambda(\te,\vp)$ is the triangular matrix-valued map given
in \eqref{eq:Lambda}.
We notice that $\aver{\eta^N}=0_n$ from the third equation in \eqref{eq:propL}.

Then, it turns out that the solution of \eqref{eq:lin:1Enew}
are approximated by the solutions of a triangular system
that involves solving two cohomological equations of
the form \eqref{eq:calL} consecutively.
Quantitative estimates for the solutions of 
such equations are obtained by applying 
R\"ussmann estimates.
This is summarized
in the following standard statement.

\begin{lemma}[Upper triangular equations]\label{lem:upperT}
Let 
$(\om, \al) \in \Dioph{\gamma}{\tau}$
and let us consider a map
$\eta= (\eta^L,\eta^N) : \TT^n \times \TT^\ell \to
\RR^{2n} \simeq \RR^n\times\RR^n$, with components in $\Anal(\TT^n_\rho\times \TT^\ell_\rho)$, and a 
map $T : \TT^n \times \TT^\ell \rightarrow \RR^{n\times n}$,
with components in $\Anal(\TT^n_\rho \times \TT^\ell_\rho)$.
Assume that $T$ satisfies the non-degeneracy condition $\det
\aver{T} \neq 0$ and $\eta$ satisfies the compatibility condition $\aver{\eta^N}=0_n$.
Then, for any $\xi^L_{0,0}\in \RR^n$, the system of equations
\begin{equation} \label{eq:lin:last1}
\begin{pmatrix}
O_n & T(\te,\vp) \\
O_n & O_n
\end{pmatrix}
\begin{pmatrix}
\xi^L(\te,\vp) \\
\xi^N(\te,\vp)
\end{pmatrix}
+
\begin{pmatrix}
\Lie{\om, \al} \xi^L(\te,\vp) \\
\Lie{\om, \al} \xi^N(\te,\vp)
\end{pmatrix}
=
\begin{pmatrix}
\eta^L(\te,\vp) \\
\eta^N(\te,\vp)
\end{pmatrix}
\end{equation}
has a solution of the form
\begin{align*}
\xi^N(\te,\vp)={} & \xi^N_{0,0} + \R{\om,\al}(\eta^N(\te,\vp)) \,, \\
\xi^L(\te,\vp)={} & \xi^L_{0,0} +\R{\om,\al}(\eta^L(\te,\vp) - T(\te,\vp)
\xi^N(\te,\vp)) \,, 
\end{align*}
where
\[
\xi^N_{0,0}= \aver{T}^{-1} \aver{\eta^L- T \R{\omega}(\eta^N)} 
\]
and $\R{\omega,\alpha}$ is given by \eqref{eq:small:formal}. 
Moreover, 
we have the estimates
\begin{align*}
& |\xi_{0,0}^N| \leq \Abs{\aver{T}^{-1}}
\Big(
\norm{\eta^L}_\rho
+ \frac{c_R}{\gamma \delta^\tau} \norm{T}_\rho \norm{\eta^N}_\rho
\Big)\, , \\
& \norm{\xi^N}_{\rho-\delta} \leq |\xi_{0,0}^N| + \frac{c_R}{\gamma \delta^\tau}
\norm{\eta^N}_\rho\, , \\
& \norm{\xi^L}_{\rho-2\delta} \leq |\xi_{0,0}^L| + \frac{c_R}{\gamma \delta^\tau}
\Big(
\norm{\eta^L}_{\rho-\delta} + 
\norm{T}_\rho \norm{\xi^N}_{\rho-\delta}
\Big)\, .
\end{align*}
\end{lemma}

\begin{proof}
From (\ref{eq:lin:last1}), we have,
\begin{equation*}
\begin{split}
\Lie{\om, \al} \xi^{N} (\te, \vp) &= \eta^N(\te, \vp) \\
\Lie{\om, \al} \xi^{L} (\te, \vp) &= \eta^L(\te, \vp) - T(\te,\vp) \xi^{N}(\te, \vp) \,.
\end{split}
\end{equation*}
By hypothesis we know that  $\aver{\eta^{N} } = 0_{n}$, so 
\begin{equation}\label{eq:solve-xi-N}
\xi^{N}\left(\theta, \varphi \right) =\xi_{0,0}^{N} +  \R{\om, \al}\left( \eta^{N}(\theta, \varphi) \right) \,.
\end{equation}
To solve the second equation, the average  $\aver{ \eta^{L} - T \xi^{N}}$ must be zero. i.e.
\begin{equation*}
\begin{split}
0_n = \aver{\eta^{L} - T \xi^{N} } 
&= \aver{\eta^{L} - T \R{\omega, \alpha}\left( \eta^{N} \right) } - \aver{T}\xi_{0,0}^{N} \,.
\end{split}
\end{equation*}
Then, 
\begin{equation}
 \xi_{0,0}^{N}= \aver{T}^{-1}\aver{ \eta^{L} - T \R{\omega, \alpha}\left( \eta^{N} \right) }\,.
 \label{eq:solve-xi-N_0}
\end{equation}
Back to the second equation, its solution is,
\begin{equation}
\xi^{L}\left(\theta, \varphi \right) =\xi_{0,0}^{L} + \R{\omega, \alpha}\left( \eta^{L}(\theta, \varphi) - T(\theta, \varphi) \xi^{N}(\theta, \varphi)\right)\,,
\label{eq:solve-xi-L}
\end{equation}
in addition, we have estimates
\begin{equation*}
\begin{split}
\norm{\xi^{N} }_{\rho - \delta} 
&\leq  | \xi_{0,0}^{N} |   + \dfrac{c_R}{\gamma \delta^{\tau}} \norm{\eta^{N} }_{\rho}  \,,
\end{split}
\end{equation*}
now, for $\xi_{0,0}^{N}$ 
\begin{equation*}
\begin{split}
\left\vert \xi_{0,0}^{N} \right\vert &\leq \left\vert \langle T \rangle^{-1} \right\vert  \left\vert \langle \eta^{L} - T \R{\omega, \alpha}\left( \eta^{N} \right) \rangle \right\vert \\
&\leq  \left\vert \langle T \rangle^{-1}  \right\vert  \left( \Vert \eta^{L} \Vert_{\rho} + \dfrac{c_{R}}{\gamma \delta^{\tau}} \Vert T \Vert_\rho  \Vert \eta^{N}  \Vert_{\rho} \right)\,.
\end{split}
\end{equation*}
Finally, 
\begin{equation*}
\begin{split}
\Vert \xi^{L} \Vert_{\rho - 2\delta} &= \Vert \xi_{0,0}^{L} + \R{\omega, \alpha}\left( \eta^{L} - T \xi^{N} \right) \Vert_{\rho - 2\delta} \\
&\leq  \vert \xi_{0,0}^{L}  \vert   + \dfrac{c_{R}}{\gamma \delta^{\tau}} \Vert \eta^{L} - T \xi^{N}  \Vert_{\rho - \delta} \\
&\leq  \vert \xi_{0,0}^{L}  \vert   + \dfrac{c_{R}}{\gamma \delta^{\tau}} \left( \Vert \eta^{L} \Vert_{\rho - \delta} + \Vert T \Vert_\rho  \Vert \xi^{N}  \Vert_{\rho - \delta} \right)\,,
\end{split}
\end{equation*}
and this concludes the proof.
\end{proof}

To approximate the solutions of \eqref{eq:lin:1Enew} we invoke Lemma \ref{lem:upperT} taking
\begin{equation}\label{eq:eta:corr}
\begin{split}
\eta^L(\te,\vp) &= - N(\te,\vp)^\top \Omega(K(\te,\vp)) E(\te,\vp)\,, \\
\eta^N(\te,\vp) &= L(\te,\vp)^\top \Omega(K(\te,\vp)) E(\te,\vp)\,,
\end{split}
\end{equation}
and $T(\te,\vp)$ given by \eqref{eq:newT}. 
Note that $\aver{\xi^N}=\xi^N_{0,0}$ and
we have the freedom of choosing any
value for $\aver{\xi^L}= \xi^L_{0,0} \in \RR^n$.
For convenience in setting the phase, 
we choose $\xi_{0,0}^L=0_n$, 
even though
other choices are possible.

From Lemma \ref{lem:upperT} we read that 
$\norm{\xi}_{\rho-2\delta}=\cO(\norm{E}_\rho)$ and,
using the geometric properties characterized in Section \ref{sec:lemmas},
we have $\norm{\Ered}_{\rho-2\delta}=\cO(\norm{E}_\rho)$ and $\norm{\Esym}_{\rho-2\delta}=\cO(\norm{E}_\rho)$.
From these estimates, we conclude that the solution of equation
\eqref{eq:lin:1Enew}
is approximated by the solution of
the cohomological equation
\begin{equation}\label{eq:mycohoxi}
\Lambda(\te,\vp) \xi(\te,\vp) + \Lie{\om,\al} \xi(\te,\vp)=\eta(\te,\vp)\,.
\end{equation}
This, together with other estimates, will be suitably quantified in the next section.

\subsection{One step of the iterative procedure}\label{ssec:iter:lemmas}

In this section we apply one correction of the quasi-Newton method
described in Section \ref{ssec:qNewton} and we obtain sharp quantitative
estimates
for the new approximately invariant torus and related objects. We
set sufficient conditions to preserve the control of the previous
estimates.

\begin{lemma} [The Iterative Lemma] \label{lem:KAM:inter:integral}
With the hypotheses of Theorem \ref{theo:KAM}, 
Then, there exist constants
$\CDeltaK$, 
$\CDeltaB$, 
$\CDeltaTOI$ and
$\CE$
such that 
if the inequalities
\begin{equation}\label{eq:cond1:K:iter}
\frac{\hCDelta \norm{E}_\rho}{\gamma^2 \delta^{2\tau+1}} < 1\,,
\qquad\qquad
\frac{\CE \norm{E}_\rho}{\gamma^4 \delta^{4\tau}} < 1\,,
\end{equation}
hold for some $0<\delta< \rho$, where
\begin{equation}\label{eq:mathfrak1}
\begin{split}
\hCDelta := \max \left\{
\Csym \gamma \delta^{\tau} \, , \,      
\frac{ n \CDeltaK }{\sigmaDK - \norm{\Dif K}_\rho} \, , \,      
\frac{ 2n  \CDeltaK }{\sigmaDKT - \norm{(\Dif K)^\top}_\rho} \, , \, \right. \\
\left.
    \frac{\CDeltaB}{\sigmaB - \norm{B}_\rho} \, , \,      
\frac{\CDeltaTOI }{\sigmaT - \abs{\aver{T}^{-1}}} \, , \,      
\frac{\CDeltaK \delta}{\dist (K(\TT^{n+\ell}_\rho),\partial (\B\times \TT^\ell_\rho ))} \,, \right\} \,.
\end{split}
\end{equation}
then we have an approximate torus of the same frequency $\omega$
given by $\bar K=K+\DeltaK$, with components in $\Anal(\TT^{n}_{\rho-2\delta} \times  \TT^{\ell}_{\rho-2\delta} )$, that defines new objects $\bar B$
and $\bar T$ (obtained by replacing $K$ by $\bar K$) satisfying
\begin{align}
& \norm{\Dif \bar K}_{\rho-3\delta} < \sigmaDK \,, \label{eq:DK:iter1} \\
& \norm{(\Dif \bar K)^\top}_{\rho-3\delta} < \sigmaDKT \,, \label{eq:DKT:iter1} \\
& \norm{\bar B}_{\rho-3\delta} < \sigmaB \,, \label{eq:B:iter1} \\
& \abs{\aver{\bar T}^{-1}} < \sigmaT \,, \label{eq:T:iter1} \\
& \dist(\bar K(\TT^{n}_{\rho-2\delta} \times  K(\TT^{\ell}_{\rho-2\delta} ),\partial  (\B\times \TT^\ell_\rho )) >0 \,, \label{eq:distB:iter1} 
\end{align}
and
\begin{align}
& \norm{\bar K-K}_{\rho-2 \delta} < \frac{\CDeltaK}{\gamma^2 \delta^{2\tau}}
\norm{E}_\rho\,, \label{eq:est:DeltaK} \\
& \norm{\bar B-B}_{\rho-3\delta} < \frac{\CDeltaB}{\gamma^2 \delta^{2 \tau+1}}
\norm{E}_\rho\,, \label{eq:est:DeltaB} \\
& \abs{\aver{\bar T}^{-1}-\aver{T}^{-1} } < \frac{\CDeltaTOI}{\gamma^2 \delta^{2 \tau+1}}
\norm{E}_\rho\,, \label{eq:est:DeltaT}
\end{align}
The new error of invariance is given by
\[
\bar E(\te,\vp) =  Z_\H (\bar K(\te,\vp),\vp) + \Lie{\omega, \alpha} \bar K(\theta,\varphi) \,,
\]
and satisfies
\begin{equation}\label{eq:E:iter1}
\norm{\bar E}_{\rho-2\delta} < \frac{\CE}{\gamma^4
\delta^{4\tau}}\norm{E}_\rho^2\,.
\end{equation}
The above constants are collected in Table~\ref{tab:constants:all:2}.
\end{lemma}

\begin{proof} 
To obtain the estimates of the theorem requires many computations, 
so we break down the demonstration into seven steps.
\bigskip


\paragraph
{\emph{Step 1:
Control of the new parameterization}.}
We start by considering a new parameterization 
$\bar K (\theta,\varphi)=K(\theta,\varphi) + \Delta K(\theta,\varphi)$, 
which is obtained from the system \eqref{eq:lin:last1}, 
with $\eta(\theta,\varphi)$ given by \eqref{eq:eta:corr}. 
We choose the solution that satisfies $\xi_{0,0}^L=0$. 
Using the estimates obtained in Section \ref{sec:lemmas}, 
we have the following inequalities:
\[
\norm{\eta^L}_\rho \leq \CNT \cteOmega \norm{E}_\rho\,\text{, and}
\qquad
\norm{\eta^N}_\rho \leq \CLT \cteOmega \norm{E}_\rho\,.
\]
Combining 
Lemma \ref{lem:twist}
and 
Lemma \ref{lem:upperT},
we obtain estimates for the solution
of the cohomological equations
\begin{align}
& \abs{\xi^N_{0,0}} \leq \abs{\aver{T}^{-1}}
\Big(
\norm{\eta^L}_{\rho} + \frac{c_R}{\gamma \delta^\tau} \norm{T}_\rho \norm{\eta^N}_\rho
\Big) \nonumber \\
& \qquad \leq \sigmaT \Big(
\CNT \cteOmega + \frac{c_R}{\gamma \delta^\tau} \CT \CLT \cteOmega
\Big)
\norm{E}_\rho =: \frac{\CxiNO}{\gamma \delta^\tau} \norm{E}_\rho\,, \label{eq:CxiNO} \\
& \norm{\xi^N}_{\rho-\delta} \leq \abs{\xi^N_0}+ \frac{c_R}{\gamma \delta^\tau} \norm{\eta^N}_\rho \nonumber \\
& \qquad \leq 
\frac{\CxiNO}{\gamma \delta^\tau} \norm{E}_\rho 
+ \frac{c_R}{\gamma \delta^\tau}\CLT \cteOmega \norm{E}_\rho
 =: 
\frac{\CxiN}{\gamma \delta^\tau} \norm{E}_\rho \,, \label{eq:CxiN} \\
& \norm{\xi^L}_{\rho-2\delta} \leq \abs{\xi_0^L} + \frac{c_R}{\gamma \delta^\tau}
\Big(
\norm{\eta^L}_{\rho} + \norm{T}_\rho \norm{\xi^N}_{\rho-\delta}
\Big) \nonumber \\
& \qquad \leq \frac{c_R}{\gamma \delta^\tau} \Big(
\CNT \cteOmega + \CT \frac{\CxiN}{\gamma \delta^\tau}
\Big)
\norm{E}_\rho =: \frac{\CxiL}{\gamma^2 \delta^{2\tau}} \norm{E}_\rho\,. \label{eq:CxiL} 
\end{align}
The magnitude of the vector $\xi(\theta, \varphi)$, 
which satisfies \eqref{eq:mycohoxi}, 
is estimated as
\begin{align}
\norm{\xi}_{\rho-2\delta} \leq {} &
\max\{ \norm{\xi^L}_{\rho-2\delta}\, , \, \norm{\xi^N}_{\rho-\delta}\} \nonumber \\
\leq {} &
\frac{\max\{\CxiL \, , \, \CxiN  \gamma \delta^\tau\}}{\gamma^2 \delta^{2\tau}} \norm{E}_\rho
=: \frac{\Cxi}{\gamma^2 \delta^{2\tau}} \norm{E}_\rho\,. \label{eq:Cxi}
\end{align}
The new parameterization $\bar K(\theta,\varphi)$ and associated objects are estimated through conventional calculations. 
The control \eqref{eq:est:DeltaK} stems directly from
\[
\bar K(\te, \vp)-K(\te, \vp) = \DeltaK(\te, \vp) = P(\te, \vp) \xi(\te, \vp) = L(\te, \vp) \xi^L(\te, \vp) + N (\te, \vp) \xi^N(\te, \vp)\,.
\]

In other words, 
by employing approximations provided in \eqref{eq:CxiN} and \eqref{eq:CxiL},
we obtain
\begin{equation}\label{eq:CDeltaK}
\norm{\bar K - K}_{\rho-2\delta} 
\leq \frac{\CL \CxiL + \CN \CxiN \gamma \delta^\tau}{\gamma^2 \delta^{2\tau}} \norm{E}_\rho =:
\frac{\CDeltaK}{\gamma^2 \delta^{2\tau}} \norm{E}_\rho\,.
\end{equation}

To complete this step, 
we need to make sure that the value of 
$\bar K(\te,\vp)$ 
falls within the defined domain $\B\times \TT^\ell_\rho$ of all the entities involved. 
This verification is essential as it enables us to calculate the new error $\bar E(\te,\vp)$ before proceeding with the remaining geometrical elements. 
To achieve this, 
we take note of the fact that:
\begin{align}
\dist(\bar K(\TT^{n+\ell}_{\rho-2\delta}),\partial (\B\times \TT^\ell_{\rho-2\delta})) \geq {} &  \dist(K(\TT^{n+\ell}_\rho ),\partial( \B\times \TT^\ell_\rho)) - \norm{\DeltaK}_{\rho-2\delta} \nonumber \\
\geq {} &  \dist(K(\TT^{n +\ell}_\rho ),\partial (\B\times \TT^\ell_\rho)) - \frac{\CDeltaK}{\gamma^2 \delta^{2\tau}} \norm{E}_\rho>0\,. \label{eq:ingredient:iter:2}
\end{align}
The last inequality follows from Hypothesis \eqref{eq:cond1:K:iter}, 
using the sixth term in \eqref{eq:mathfrak1}. 
Thus, 
we have established the control in \eqref{eq:distB:iter1}.
\bigskip


\paragraph
{\emph{Step 2: Control of the new error of invariance}.}

To guarantee that the corrected parameterization $\bar K$ is invariant, 
we control the quadratic terms neglected in equation  \eqref{eq:mycohoxi}, 
by first considering the error in solving the linearized equation \eqref{eq:lin:1Enew},
\begin{equation}\label{eq:Elin}
\Elin (\te, \vp) = \Ered(\te, \vp) \xi(\te, \vp) - \Omega_0 \Esym(\te, \vp) \Lie{\om,\al}\xi(\te, \vp)\,.
\end{equation}
We control the term $\Lie{\om,\al} \xi(\te,\vp)$ by leveraging the fact that 
$\xi(\te,\vp)$ satisfies the solution of the cohomological equation \eqref{eq:mycohoxi}:
\begin{align}
\norm{\Lie{\omega} \xi^N}_{\rho}
= {} & 
\norm{\eta^N}_{\rho}
\leq \CLT \cteOmega \norm{E}_\rho =: \CLiexiN \norm{E}_\rho \,,
\label{eq:CLiexiN}
\\
\norm{\Lie{\omega} \xi^L}_{\rho-\delta}
={} &
\norm{\eta^L - T \xi^N}_{\rho-\delta}
\leq \left(\CNT \cteOmega + \CT \frac{\CxiN}{\gamma \delta^\tau}\right)\norm{E}_\rho \nonumber \\
=: {} & \frac{\CLiexiL}{\gamma \delta^{\tau}} \norm{E}_\rho \,,
 \label{eq:CLiexiL} \\
\norm{\Lie{\omega} \xi}_{\rho-\delta} 
    \leq {} & \max\left\{ \frac{\CLiexiL}{\gamma \delta^{\tau}}, \CLiexiN \right\} \norm{E}_\rho =: \frac{\CLiexi}{\gamma \delta^{\tau}} \norm{E}_\rho\,.
 \label{eq:CLiexi}
\end{align}

Therefore, the control for $\Elin(\te, \vp)$ is given by 
\begin{align}
\norm{\Elin}_{\rho-2\delta} \leq {} & 
\norm{\Ered}_{\rho-2\delta} \norm{\xi}_{\rho-2\delta} +
\cteOmega \norm{\Esym}_{\rho-2\delta} \norm{\Lie{\omega}\xi}_{\rho-2 \delta} \nonumber \\
\leq {} &
\frac{\Cred}{\gamma \delta^{\tau+1}}
\frac{\Cxi}{\gamma^2 \delta^{2 \tau}}
\norm{E}_\rho^2
+
\cteOmega
\frac{\Csym}{\gamma \delta^{\tau+1}}
\frac{\CLiexi}{\gamma \delta^{\tau}}
\norm{E}_\rho^2
=: \frac{\Clin}{\gamma^3 \delta^{3\tau+1}} \norm{E}_\rho^2\,.
\label{eq:Clin}
\end{align}
After making the correction, 
the invariance error associated with the new parameterization can be determined.
\begin{equation}\label{eq:new:E:comp}
\begin{split}
\bar E(\te,\vp) = {} & Z_\H(K(\te,\vp)+\DeltaK(\te,\vp),\vp) + \Lie{\om,\al}K(\te,\vp) + \Lie{\om, \al}\DeltaK(\te,\vp) \\
= {} & \Big( \Dif Z_\H (K(\te,\vp),\vp) P(\te,\vp) + \Lie{\om,\al}P(\te,\vp) \Big) \xi(\te,\vp)
+P(\te,\vp) \Lie{\om,\al} \xi(\te,\vp) \\
{} & \quad+  E(\te,\vp) + \Delta^2Z_\H( K(\te,\vp),\vp) \\
= {} & P(\te,\vp) (I_{2n}-\Omega_0 \Esym(\te,\vp))^{-1} \Elin(\te,\vp)+\Delta_\theta^2 Z_\H(K(\te,\vp) ,\vp)\,,
\end{split}
\end{equation}
where
\begin{equation}\label{eq:Delta2X}
\begin{split}
\Delta^2_\theta Z_\H( K(\te,\vp),\vp) = {} & Z_\H(K(\te,\vp)+\DeltaK(\te,\vp),\vp)-Z_\H(K(\te,\vp),\vp)-\Dif_z Z_\H(K(\te,\vp),\vp) \DeltaK(\te,\vp) \\
= {} & \int_0^1 (1-s) \Dif^2_z Z_\H(K(\te,\vp)+t \DeltaK(\te,\vp),\vp) [\DeltaK(\te,\vp),\DeltaK(\te,\vp)] \dif s\,,
\end{split}
\end{equation}
and we used equation \eqref{eq:Esym} and \eqref{eq:Elin}. We use the mean value theorem to control $\Delta_\theta^2 Z(\te,\vp)$, i.e.,
\begin{equation}
\begin{split}
\norm{ \Delta_{\theta}^{2}Z_{h}(\theta, \varphi)}_{\rho -2\delta} 
&\leq \dfrac{1}{2}  c_{Z_{h}, 2} \dfrac{(C_{\Delta K})^{2}}{\gamma^{4} \delta^{4\tau}} \norm{E}_{\rho}^{2} \,.
\label{eqn: Delta2}
\end{split}
\end{equation}

It is noteworthy that the error function equation \eqref{eq:new:E:comp}, 
is well-defined as a result of the calculations in equation \eqref{eq:ingredient:iter:2}. 
We estimate the norm of the error function as follows
\[
\norm{\bar E}_{\rho-2\delta} \leq \norm{P}_{\rho-2\delta} \norm{(I-\Omega_0 \Esym)^{-1}}_{\rho-2\delta} \norm{\Elin}_{\rho-2\delta}+\norm{\Delta^2_\theta Z_\H}_{\rho-2\delta}\,.
\]
Then, using a Neumann series argument, we obtain
\[
    \norm{(I-\Omega_0 \Esym)^{-1}}_{\rho-2\delta} \leq \frac{1}{1-\norm{\Omega_0 \Esym}_{\rho-2\delta}} < 1 \,,
\]
where we used the inequality
\[
    \frac{\Csym}{\gamma \delta^{\tau+1}}\norm{E}_\rho < 1\,,
\]
that corresponds to the first term in \eqref{eq:mathfrak1}
(Hypothesis \eqref{eq:cond1:K:iter}).
Putting together the above estimates, we obtain
\begin{equation}\label{eq:CE}
\norm{\bar E}_{\rho-2\delta} \leq \left(\frac{ 2(\CL+\CN)\Clin}{\gamma^3 \delta^{3\tau+1}} + \frac{1}{2} \cteDDXH \frac{(\CDeltaK)^2}{\gamma^4 \delta^{4\tau}} \right)
\norm{E}_\rho^2 =:
\frac{\CE}{\gamma^4 \delta^{4\tau}} \norm{E}_\rho^2\,.
\end{equation}
We have obtained estimate \eqref{eq:E:iter1}.

\bigskip


\paragraph
{\emph{Step 3: Control of the new frame $L(\te,\vp)$}.}
Equation \eqref{eq:CDeltaK} can be combined with Cauchy estimates to derive the constraint presented in equation \eqref{eq:DK:iter1},
\begin{equation} \label{eq:ingredient:iter:4}
\norm{\Dif_\theta \bar K}_{\rho-3\delta} \leq 
\norm{\Dif_\theta K}_{\rho} +
\norm{\Dif_\theta \DeltaK}_{\rho-3\delta} \leq 
\norm{\Dif_\theta K}_{\rho} +
\frac{n \CDeltaK}{\gamma^2 \delta^{2\tau+1}} \norm{E}_\rho < \sigmaDK \,,
\end{equation}
where the last inequality follows from
Hypothesis \eqref{eq:cond1:K:iter} (this corresponds to the second term in \eqref{eq:mathfrak1}).
Estimate \eqref{eq:DKT:iter1} is analogous
\begin{equation} \label{eq:ingredient:iter:5}
\norm{(\Dif_\theta \bar K)^\top}_{\rho-3\delta} \leq 
\norm{(\Dif_\theta K)^\top}_{\rho} +
\frac{2n \CDeltaK}{\gamma^2 \delta^{2\tau+1}} \norm{E}_\rho < \sigmaDKT \,.
\end{equation}
After obtaining the estimates \eqref{eq:ingredient:iter:4} and \eqref{eq:ingredient:iter:5},
it is clear that
\[
\norm{\bar L}_{\rho-3\delta} \leq \CL\,,
\qquad
\norm{\bar L^\top}_{\rho-3\delta} \leq \CLT\,.
\]
We can control the norm of the corresponding corrections using Cauchy estimates, 
the mean value theorem, 
and the estimate \eqref{eq:CDeltaK}:
\begin{align}
\norm{\bar L-L}_{\rho-3 \delta} \leq {} & 
\frac{n \CDeltaK}{\gamma^2 \delta^{2\tau+1}} \norm{E}_\rho  =: \frac{\CDeltaL}{\gamma^2 \delta^{2\tau+1}} \norm{E}_\rho\,, \label{eq:CDeltaL} \\
\norm{\bar L^\top-L^\top}_{\rho-3 \delta} \leq {} & 
\frac{ 2n \CDeltaK }{\gamma^2 \delta^{2\tau+1}} \norm{E}_\rho 
=: \frac{\CDeltaLT}{\gamma^2 \delta^{2\tau+1}} \norm{E}_\rho\,. \label{eq:CDeltaLT}
\end{align}

\bigskip


\paragraph
{\emph{Step 4: Control of the new transversality condition}.}
To control $\bar B$, 
we take

\begin{align*}
G_L(\te, \vp) = L(\te, \vp)^\top G(K(\te, \vp)) L(\te, \vp) \,,\\
G_{\bar L}(\te, \vp) = \bar L(\te, \vp)^\top G(\bar K(\te, \vp)) \bar L(\te, \vp) \,,
\end{align*}
where we have used
the notation introduced in \eqref{eq:def:GL}.
First, we compute
\begin{align}
\norm{G \circ \bar K-G \circ K}_{\rho-2\delta} \leq {} & \norm{\Dif G}_{\B} \norm{\bar K - K}_{\rho-2\delta}
\leq \frac{\cteDG \CDeltaK}{\gamma^2 \delta^{2\tau}} \norm{E}_\rho \nonumber \\
=: {} & \frac{\CDeltaG}{\gamma^2 \delta^{2\tau}} \norm{E}_\rho\,, \label{eq:CDeltaG}
\end{align}
and
\begin{align}
\norm{G_{\bar L}-G_L}_{\rho-3\delta} 
\leq {} & \norm{\bar L^\top (G \circ \bar K) \bar L - \bar L^\top (G \circ \bar K) L}_{\rho-3\delta}  \nonumber\\
        & +\norm{\bar L^\top (G \circ \bar K) L - \bar L^\top (G \circ K) L}_{\rho-3\delta} \nonumber \\
        &+\norm{\bar L^\top (G \circ K) L - L^\top (G \circ K) L}_{\rho-3\delta} \nonumber \\
\leq {} & \norm{\bar L^\top}_\rho \norm{G}_\B \norm{\bar L- L}_{\rho-3\delta}  \nonumber\\
        & + \norm{\bar L^\top}_\rho \norm{G \circ \bar K-G \circ K}_{\rho-3\delta} \norm{L}_\rho \nonumber \\
        &+\norm{\bar L^\top - L^\top}_{\rho-3\delta}\norm{G}_\B \norm{L}_\rho \nonumber \\ 
\leq {} & \frac{\CLT \cteG \CDeltaL + \CLT \CDeltaG \CL \delta + \CDeltaLT \cteG \CL }{\gamma^2 \delta^{2\tau+1}} \norm{E}_\rho \nonumber \\
  =: {} & \frac{\CDeltaGL}{\gamma^2 \delta^{2\tau+1}} \norm{E}_\rho \,. \label{eq:CDeltaGL} 
\end{align}
Afterward, we introduce the constant
\[
\CDeltaB := (\sigmaB)^2 \CDeltaGL\,,
\]
and review and verify the condition (A.1)  mentioned in Lemma (A.1), as stated in \cite{figueras2024modified} 
\begin{align}
\frac{(\sigmaB)^2 \norm{G_{\bar L}-G_L}_{\rho-3\delta}}{\sigmaB-\norm{B}_\rho} 
\leq \frac{\CDeltaB}{\sigmaB-\norm{B}_\rho} \frac{\norm{E}_\rho}{\gamma^2 \delta^{2\tau+1}}
< 1\,.
\label{eq:ingredient:iter:6}
\end{align}
The final inequality stems from the Hypothesis \eqref{eq:cond1:K:iter}, 
corresponding to the fourth term in \eqref{eq:mathfrak1}. 
Thus, 
with the Lemma (A.1), as given in \cite{figueras2024modified}, 
we conclude that
\begin{equation}\label{eq:CDeltaB}
\norm{\bar B}_{\rho-3\delta} < \sigmaB\,, \qquad
\norm{\bar B-B}_{\rho-3\delta} \leq \frac{(\sigmaB)^2 \CDeltaGL}{\gamma^2 \delta^{2\tau+1}}
\norm{E}_\rho 
= \frac{\CDeltaB}{\gamma^2 \delta^{2\tau+1}}
\norm{E}_\rho\,,
\end{equation}
which allows us to arrive at the estimates \eqref{eq:B:iter1} and \eqref{eq:est:DeltaB} for the new object.

\bigskip


\paragraph
{\emph{Step 5: Control of the new frame $N(\te,\vp)$}.}
For controlling the newly adapted normal frame $\bar N(\te,\vp)$, 
it is advantageous to adapt the notation:
\begin{align*}
N(\te,\vp) = {} &  L(\te,\vp) A(\te,\vp) + \tilde N(\te,\vp)\,, \\
\tilde N(\te,\vp) = {} & N^0(\te,\vp) B(\te,\vp)\,, \\
N^0(\te,\vp) = {} & J(K(\te,\vp)) L(\te,\vp) \,,\\
A(\te,\vp) = {} & -\tfrac{1}{2} (\tilde N(\te,\vp)^\top  \Omega(K(\te,\vp))\tilde N(\te,\vp)) \,, \\
B(\te,\vp) = {} & (L(\te,\vp)^\top G(K(\te,\vp)) L(\te,\vp))^{-1}\,.
\end{align*}
As usual, we obtain the new objects $\bar N(\theta,\varphi)$, $\bar A(\theta,\varphi)$, $\bar N^0(\theta,\varphi)$, $\bar{\tilde{N}}(\theta,\varphi)$ and $\bar B(\theta,\varphi)$ by replacing $K(\te,\vp)$ with $\bar K(\te,\vp)$. Note that we controlled $\bar B(\theta,\varphi)$ in Step~4.
Now, we 
reproduce the computations in \eqref{eq:CDeltaG} and \eqref{eq:CDeltaGL} for the matrix functions.
\begin{align}
\norm{\Omega \circ \bar K - \Omega \circ K}_{\rho -2\delta} \leq {} & \frac{\cteDOmega \CDeltaK}{\gamma^2 \delta^{2\tau}} \norm{E}_\rho =:\frac{\CDeltaOmega}{\gamma^2 \delta^{2\tau}} \norm{E}_\rho \,, \label{eq:CDeltaOmega} \\
\norm{J \circ \bar K-J \circ K}_{\rho -2\delta} \leq {} & \frac{\cteDJ \CDeltaK}{\gamma^2 \delta^{2\tau}} \norm{E}_\rho =:\frac{\CDeltaJ}{\gamma^2 \delta^{2\tau}} \norm{E}_\rho \,, \label{eq:CDeltaJ} \\
\norm{(J \circ \bar K-J \circ K)^\top}_{\rho -2\delta} \leq {} & \frac{\cteDJT \CDeltaK}{\gamma^2 \delta^{2\tau}} \norm{E}_\rho =:\frac{\CDeltaJT}{\gamma^2 \delta^{2\tau}} \norm{E}_\rho \,, \label{eq:CDeltaJT} \\
\norm{\bar N^0-N^0}_{\rho-3\delta} \leq {} &
\norm{J}_\B \norm{\bar L-L}_{\rho-3\delta} + \norm{J \circ \bar K- J \circ K}_{\rho-2 \delta} \norm{L}_\rho \nonumber \\
\leq {} & \frac{\cteJ \CDeltaL + \CDeltaJ \CL \delta}{\gamma^2 \delta^{2\tau+1}} \norm{E}_\rho \nonumber \\
=: {} &  \frac{\CDeltaNO}{\gamma^2 \delta^{2\tau+1}} \norm{E}_\rho \,, \label{eq:CDeltaNO} \\
\norm{(\bar N^0)^\top-(N^0)^\top}_{\rho-3\delta} 
\leq {} & \frac{\CDeltaLT \cteJT + \CLT \CDeltaJT \delta }{\gamma^2 \delta^{2\tau+1}} \norm{E}_\rho \nonumber \\
=: {} & \frac{\CDeltaNOT}{\gamma^2 \delta^{2\tau+1}} \norm{E}_\rho \,, \label{eq:CDeltaNOT} \\
\norm{ \bar{\tilde{N}} - \tilde{N} }_{\rho-3\delta} 
\leq {} & \frac{\CDeltaNO \sigmaB + \CNO \CDeltaB }{\gamma^2 \delta^{2\tau+1}} \norm{E}_\rho =:  \frac{\CDeltaTildeN}{\gamma^2 \delta^{2\tau+1}} \norm{E}_\rho \,, \label{eq:CDeltaTildeN} \\
\norm{ \bar{\tilde{N}}^\top - \tilde{N}^\top }_{\rho-3\delta} 
\leq {} & \frac{\sigmaB \CDeltaNOT  + \CDeltaB \CNOT  }{\gamma^2 \delta^{2\tau+1}} \norm{E}_\rho =:  \frac{\CDeltaTildeNT}{\gamma^2 \delta^{2\tau+1}} \norm{E}_\rho \,. \label{eq:CDeltaTildeNT}
\end{align}
Now, we control the matrix $\bar A(\theta, \varphi)$ as follows
\begin{align}
\norm{\bar A-A}_{\rho-3\delta} 
    \leq {} & \dfrac{ \frac{1}{2}\Big( \CDeltaTildeNT  \cteOmega \CTildeN + \CTildeNT  \CDeltaOmega \CTildeN \delta +  \CTildeNT \cteOmega \CDeltaTildeN \Big)}{\gamma^2 \delta^{2\tau+1}} \norm{E}_\rho
=: \frac{\CDeltaA}{\gamma^2 \delta^{2\tau+1}}\norm{E}_\rho \,. \label{eq:CDeltaA}
\end{align}
In this instance, we have employed the fact that $B(\theta,\varphi)^\top = B(\theta,\varphi)$, along with the constants specified in equations \eqref{eq:CDeltaOmega}, \eqref{eq:CDeltaTildeN}, and \eqref{eq:CDeltaTildeNT}. Similarly, the corresponding control is maintained for $\norm{\bar{A}^\top - A^\top}_{\rho - 3\delta}$, as we observe that $A(\theta,\varphi)^\top = -A(\theta,\varphi)$.
We notice that $\CDeltaA=0$ in \textbf{Case III}.

Then, we control the correction of the adapted normal frame:
\begin{align}
\|\bar N-N &\|_{\rho-3\delta} \nonumber \\
& \leq \frac{\CL \CDeltaA + \CDeltaL \CA + \CDeltaTildeN}{\gamma^2 \delta^{2\tau+1}} \norm{E}_\rho 
\nonumber \\
& =: \frac{\CDeltaN}{\gamma^2 \delta^{2\tau+1}} \norm{E}_\rho \,, \label{eq:CDeltaN}
\\
\|\bar N^\top-N^\top&\|_{\rho-3\delta} \nonumber \\
& \leq \frac{\CA \CDeltaLT + \CDeltaA \CLT + \CDeltaTildeNT}{\gamma^2 \delta^{2\tau+1}} \norm{E}_\rho \nonumber \\
& =: \frac{\CDeltaNT}{\gamma^2 \delta^{2\tau+1}} \norm{E}_\rho \,. \label{eq:CDeltaNT}
\end{align}
Finally, we have
\begin{align*}
    \norm{ \bar{\tilde{N}} }_{\rho-3\delta} \leq {}& \norm{J \circ \bar K}_{\B \times \TT^\ell_\rho} \norm{ \bar L}_{\rho-3\delta}  \norm{ \bar B}_{\rho-3\delta} \leq \cteJ \CL \sigmaB \leq \CTildeN \,,\\
    \norm{ \bar{\tilde{N}}^\top }_{\rho-3\delta} \leq {}& \norm{ \bar B}_{\rho-3\delta} \norm{ \bar L^\top}_{\rho-3\delta}  \norm{(J \circ \bar K)^\top}_{\B \times \TT^\ell_\rho} \leq \CTildeNT \,,\\
    \norm{ \bar{A}}_{\rho-3\delta} \leq {}& \frac{1}{2} \norm{ \bar{\tilde{N}}^\top }_{\rho-3\delta} \norm{\Omega \circ \bar K}_{\B \times \TT^\ell_\rho} \norm{ \bar{\tilde{N} } }_{\rho-3\delta}  \leq \CA \,,\\
    \norm{ \bar{N}}_{\rho-3\delta}  \leq {}&\CN \,,\\
    \norm{ \bar{N}^\top}_{\rho-3\delta}  \leq {}& \CNT \,.
\end{align*}
\bigskip


\paragraph
{\emph{Step 6: Control of the $T_\H(K(\theta,\varphi),\varphi)$ operator}.}
We consider the correction $T_\H(\bar K(\theta,\varphi),\varphi)$ and $T_\H(\bar K(\theta,\varphi),\varphi)^\top$,i.e., 
\begin{align}
\norm{\bar T_\H }_\rho  \leq {} & \cteTH \,, \label{eq:CbarTH} \\
\norm{\bar T^\top_\H }_\rho  \leq {} & \cteTHT \,. \label{eq:CbarTHT}
\end{align}
The last two inequalities are satisfied since $T_\H$ is a global object defined in $\B\times\TT^\ell_\rho$, 
hypothesis $H_1$ and Step~2.
On the other hand, the control of the correction to the new $T_\H(\theta, \varphi)$,
\begin{equation}\label{eq:CDeltaTH}
\begin{split}
\norm{T_\H \circ \bar K - T_\H \circ K}_{\rho -2\delta}
\leq {} &  \norm{ \Dif_z  T_\H}_{\rho -2\delta}  \norm{ \bar K - K}_{\rho -2\delta} \\
\leq {} &  \frac{\cteDTH \CDeltaK}{\gamma^2 \delta^{2\tau}} \norm{E}_\rho   =: \frac{\CDeltaTH}{\gamma^2 \delta^{2\tau}} \norm{E}\,.
\end{split}
\end{equation}
Finally,
\begin{equation}\label{eq:CDeltaTHT}
\begin{split}
\norm{ (T_\H \circ \bar K )^\top - (T_\H \circ  K )^\top}_{\rho -2\delta}
\leq {} &   \frac{\cteDTHT \CDeltaK}{\gamma^2 \delta^{2\tau}} \norm{E}_\rho   =: \frac{\CDeltaTHT}{\gamma^2 \delta^{2\tau}} \norm{E} \,.
\end{split}
\end{equation}
{\color{blue}
}
\bigskip


\paragraph
{\emph{Step 7: Control of the new non-degeneracy condition}.}
We now establish that the non-degeneracy (twist) condition with the new torsion matrix is controled,
\begin{equation} 
\begin{split}
\norm{\bar T -T}_{\rho-3\delta} 
& \leq \dfrac{1}{2}  \frac{\CDeltaTildeNT \cteTH \CTildeN + \CTildeNT \CDeltaTH \CTildeN \delta + \CTildeNT \cteTH \CDeltaTildeN}{\gamma^2 \delta^{2\tau+1}} \norm{E}_\rho  \\
& \quad +  \frac{1}{2}  \frac{\CDeltaTildeNT \cteTHT \CTildeN + \CTildeNT \CDeltaTHT \CTildeN \delta + \CTildeNT \cteTHT \CDeltaTildeN}{\gamma^2 \delta^{2\tau+1}} \norm{E}_\rho \\
& \quad + \frac{\CDeltaTildeNT \cteTHT \CN + \CTildeNT \CDeltaTHT \CN \delta + \CTildeNT \cteTHT \CDeltaN}{\gamma^2 \delta^{2\tau+1}} \norm{E}_\rho  \\
& \quad +  \frac{\CDeltaNT \cteTH \CTildeN + \CNT \CDeltaTH \CTildeN \delta + \CNT \cteTH \CDeltaTildeN}{\gamma^2 \delta^{2\tau+1}} \norm{E}_\rho \\
& =: \frac{\CDeltaT}{\gamma^2 \delta^{2\tau+1}} \norm{E}_\rho\,. \label{eq:CDeltaT}
\end{split}
\end{equation}
Now,
we introduce the constant
\[
\CDeltaTOI :=  (\sigmaT)^2 \CDeltaT\,,
\]
and verify condition (A.1) in Lemma (A.1) from \cite{figueras2024modified} 
\begin{align}
\frac{(\sigmaT)^2 |\aver{\bar T} - \aver{T}|}{\sigmaT-|\aver{T}^{-1}|} \leq {} &
\frac{(\sigmaT)^2 \norm{\bar T - T}_{\rho-3\delta}}{\sigmaT-|\aver{T}^{-1}|} \leq 
\frac{\CDeltaTOI}{\sigmaT - |\aver{T}^{-1}|} \frac{\norm{E}_\rho}{\gamma^2 \delta^{2\tau+1}}
 < 1\,,
\label{eq:ingredient:iter:8}
\end{align}
where the last inequality follows from
Hypothesis \eqref{eq:cond1:K:iter} (this corresponds to the sixth term in \eqref{eq:mathfrak1}).
Hence, by invoking Lemma (A.1) from \cite{figueras2024modified}, 
we conclude that
\begin{equation}\label{eq:CDeltaTOI}
|\aver{\bar T}^{-1}| <\sigmaT\,, \qquad
|\aver{\bar T}^{-1} - \aver{T}^{-1}| \leq \frac{ (\sigmaT)^2 \CDeltaT}{\gamma^2 \delta^{2\tau+1}}
\norm{E}_\rho =: \frac{\CDeltaTOI}{\gamma^2 \delta^{2\tau+1}}
\norm{E}_\rho\,,
\end{equation}
and so, we obtain the estimates \eqref{eq:B:iter1} and \eqref{eq:est:DeltaT} on the new object.
\end{proof}

\subsection{Convergence of the iterative process}\label{ssec:proof:KAM}
We are ready to prove our KAM theorem for Lagrangian tori. 
Once we established the quadratic procedure in Section \ref{ssec:iter:lemmas}, 
proving the convergence of the scheme follows by standard arguments. 
However, we will carefully explain the required computations as our goal is to present explicit conditions for the KAM theorem.

\begin{proof}
[Proof of Theorem \ref{theo:KAM}]
Let us consider the approximate torus $K_0:=K$ with initial error $E_0:=E$. We
also introduce $B_0:=B$ and $T_0:=T$ associated with the initial approximation.
By applying Lemma \ref{lem:KAM:inter:integral} recursively, we obtain new objects
$K_s:=\bar K_{s-1}$,
$E_s:=\bar E_{s-1}$,
$B_s:=\bar B_{s-1}$ and
$T_s:=\bar T_{s-1}$. 
At each step, 
the domain of analyticity of these objects gets reduced. 
To characterize this, we introduce parameters
$a_1>1$, $a_2>1$, $a_3:=3 \frac{a_1}{a_1-1} \frac{a_2}{a_2-1}$ and
define
\[
\rho_0:=\rho, \quad
\delta_0:=\frac{\rho_0}{a_3}, \quad
\rho_s:=\rho_{s-1} - 3 \delta_{s-1}, \quad
\delta_s:= \frac{\delta_0}{a_1^s}, \quad
\rho_\infty:= \lim_{s\to \infty} \rho_s = \frac{\rho_0}{a_2}\,.
\]

Assuming that we have applied the Iterative Lemma \ref{lem:KAM:inter:integral} $s$ times successfully, 
we can denote the objects at the $s$-step of the quasi-Newton method as $K_s$, $E_s$, $B_s$, and $T_s$. 
It is important to note that condition  \eqref{eq:cond1:K:iter} is required at every step, 
but the construction has been performed in a way that we can control 
$\norm{\Dif K_s}_{\rho_s}$,
$\norm{(\Dif K_s)^\top}_{\rho_s}$,
$\norm{B_s}_{\rho_s}$,
$\dist(K_s(\TT^{n}_{\rho_s} \times \TT^{\ell}_{\rho_s}  ),\partial (\B\times \TT^\ell_\rho))$,
and
$\abs{\aver{T_s}^{-1}}$
uniformly with respect to $s$. 
This means that the constants obtained in Table \ref{tab:constants:all} and
Table \ref{tab:constants:all:2}) are the same for all steps, 
as we consider the worst value of $\delta_s$, that is, $\delta_0 = \rho_0/a_3$.

The first computation is monitoring the sequence $E_s$ of invariance errors
\begin{equation}\label{eq:conv:geom}
\begin{split}
\norm{E_s}_{\rho_s} < {} & 
\frac{\CE}{\gamma^4 \delta_{s-1}^{4\tau}} \norm{E_{s-1}}_{\rho_{s-1}}^2 =
\frac{\CE a_1^{4\tau(s-1)}}{\gamma^4 \delta_0^{4\tau}} \norm{E_{s-1}}_{\rho_{s-1}}^2 \\
< {} & \left( \frac{a_1^{4 \tau} \CE \norm{E_0}_{\rho_0}}{\gamma^4 \delta_0^{4\tau}}
\right)^{2^s-1} a_1^{-4\tau s} \norm{E_0}_{\rho_0} 
< a_1^{-4\tau s} \norm{E_0}_{\rho_0}
\,, 
\end{split}
\end{equation}
where we used the sums
$1+2+\ldots+2^{s-1}=2^s-1$, and
$1(s-1)+2(s-2)+2^2(s-3)\ldots+2^{s-2} 1 = 2^s-s-1$.
Notice that we also used the inequality
\begin{equation}\label{eq:theo:conv:1}
\frac{a_1^{4 \tau} \CE \norm{E_0}_{\rho_0}}{\gamma^4 \delta_0^{4\tau}} <1\,,
\end{equation}
which is included in \eqref{eq:KAM:HYP}.
From equation \eqref{eq:conv:geom}, 
we verify Hypothesis  \eqref{eq:cond1:K:iter} of the iterative Lemma, 
enabling us to proceed with step $s + 1$. 
The necessary and sufficient condition will be incorporated into Hypothesis \eqref{eq:KAM:HYP} of the KAM theorem.
The first inequality in \eqref{eq:cond1:K:iter} reads
\[
\frac{\CE \norm{E_s}_{\rho_s}}{\gamma^4 \delta_s^{4\tau}} \leq
\frac{\CE a_1^{-4\tau s} \norm{E_0}_{\rho_0}}{\gamma^4 \delta_s^{4\tau}} 
= 
\frac{\CE \norm{E_0}_{\rho_0}}{\gamma^4 \delta_0^{4\tau}} 
\leq a_1^{-4\tau}
< 1 \,,
\]
where we used \eqref{eq:conv:geom} and \eqref{eq:theo:conv:1}.
The first term in \eqref{eq:mathfrak1}  is guaranteed by
performing the  computation
\begin{equation}\label{eq:theo:conv:3}
\frac{2 \Csym \norm{E_s}_{\rho_s}}{\gamma \delta_s^{\tau+1}}
\leq 
\frac{2 \Csym a_1^{s(\tau+1)} a_1^{-4\tau s} \norm{E_0}_{\rho_0}}{\gamma \delta_0^{\tau+1}}
\leq 
\frac{2 \Csym \norm{E_0}_{\rho_0}}{\gamma \delta_0^{\tau+1}}
< 1 \,,
\end{equation}
where we used \eqref{eq:conv:geom}, the fact that $1-3\tau<0$,
and we have included the last inequality in \eqref{eq:KAM:HYP}.
The remaining conditions are similar. 
It is important to note that at the $s$-th step, 
the objects involved are $\Dif K_s$, $(\Dif K_s)^\top$, $B_s$ and $\aver{T_s}^{-1}$. 
Therefore, 
it is necessary to establish a relationship between these conditions and the initial objects $\Dif K_0$, 
$(\Dif K_0)^\top$, 
$B_0$ and 
$\aver{T_0}^{-1}$. 
For example, 
the second term in \eqref{eq:mathfrak1} reads
\[
\left(\frac{ n\CDeltaK}{\sigmaDK-\norm{\Dif K_s}_{\rho_s}}\right) \frac{\norm{E_s}_{\rho_s}}{\gamma^2 \delta_s^{2\tau+1}}< 1\,,
\]
and it is verified by conducting the following calculations
\begin{align}
\norm{\Dif K_s}_{\rho_s} + & \frac{n \CDeltaK\norm{E_s}_{\rho_s}}{\gamma^2 \delta_s^{2\tau+1}} 
<  
\norm{\Dif K_0}_{\rho_0} + \sum_{j=0}^s \frac{n \CDeltaK \norm{E_j}_{\rho_j}}{\gamma^2 \delta_j^{2\tau+1}} \nonumber \\
< {} &  
\norm{\Dif K_0}_{\rho_0} + \sum_{j=0}^\infty \frac{n \CDeltaK
a_1^{(1-2\tau)j}}{\gamma^2 \delta_0^{2\tau+1}}
\norm{E_0}_{\rho_0} \nonumber
\\
= {} &  
\norm{\Dif K_0}_{\rho_0} + \frac{n
\CDeltaK \norm{E_0}_{\rho_0}}{\gamma^2 \delta_0^{2\tau+1}}
\left(
\frac{1}{1-a_1^{1-2\tau}}
\right) < \sigmaDK\,. \label{eq:theo:conv:4}
\end{align}
The last inequality is incorporated in equation \eqref{eq:KAM:HYP}. 
After reproducing the same computations,
we obtain the 
third,
fourth, and
fiftht terms
in \eqref{eq:mathfrak1}, 
related to $(\Dif K_s)^\top$, $B_{s}$ and $\aver{T_s}^{-1}$. 
Finally, the sixth term in  \eqref{eq:mathfrak1} is verified as follows
\begin{align}
\dist (K_s(\TT^{n}_{\rho_s} \times \TT^{\ell}_{\rho_s}  ),\partial (\B\times\TT^\ell_\rho)) - & \frac{\CDeltaK}{\gamma^2 \delta_s^{2\tau}} \norm{E_s}_{\rho_s}> 
\dist (K_0(\TT^{n}_{\rho_0} \times \TT^{\ell}_{\rho_0} ),\partial (\B\times\TT^\ell_\rho)) \nonumber \\
& - \sum_{j=0}^\infty \frac{\CDeltaK a_1^{-2\tau j} \norm{E_0}_{\rho_0}}{\gamma^2 \delta_0^{2\tau}} 
\nonumber \\
& = \dist (K_0(\TT^{n}_{\rho_0} \times \TT^{\ell}_{\rho_0} ),\partial (\B\times\TT^\ell_\rho)) \nonumber \\ 
&- \frac{\CDeltaK \norm{E_0}_{\rho_0}}{\gamma^2 \delta_0^{2\tau}}
\frac{1}{1-a_1^{-2\tau}} > 0\,,
\label{eq:check:dist}
\end{align}
the final inequality is included into \eqref{eq:KAM:HYP}.

Having guaranteed all hypotheses of Lemma \ref{lem:KAM:inter:integral},
we collect the inequalities
\eqref{eq:theo:conv:1},
\eqref{eq:theo:conv:3},
\eqref{eq:theo:conv:4} and
\eqref{eq:check:dist} that are included into
hypothesis \eqref{eq:KAM:HYP}. This follows
by introducing the constant $\mathfrak{C}_1$ as
\begin{equation}\label{eq:cte:mathfrakC1}
\mathfrak{C}_1:=\max
\left\{
(a_1 a_3)^{4\tau} \CE \, , \, (a_3)^{2\tau+1} \gamma^2 \rho^{2\tau-1} \CDeltatot
\right\}
\end{equation}
where
\begin{align}
\CDeltaI := {} &
\max
\Bigg\{
\frac{n \CDeltaK}{\sigmaDK-\norm{\Dif K}_{\rho}} \, , \,
\frac{2n \CDeltaK}{\sigmaDKT-\norm{(\Dif K)^\top}_{\rho}} \, , \label{eq:CDeltaI}
\\
& \qquad\qquad\qquad\qquad  \frac{\CDeltaB}{\sigmaB-\norm{B}_{\rho}} \, , \,
\frac{\CDeltaTOI}{\sigmaT-\abs{\aver{T}^{-1}}}
\Bigg\} \,, \nonumber \\
\CDeltaII := {} & \frac{\CDeltaK \delta}{\dist(K(\TT^d_{\rho}),\partial \B)}\,,
\label{eq:CDeltaII} \\
\CDeltatot := {} & \max
\Bigg\{
\Csym \gamma \delta^\tau \, , \,
\frac{\CDeltaI}{1-a_1^{1-2\tau}} \, , \, 
\frac{\CDeltaII}{1-a_1^{-2\tau}}
\Bigg\}
\label{eq:CDeltatot} \,.
\end{align}
It is worth highlighting that we have reverted to the original notation, 
where $K=K_0$, $B=B_0$, $T=T_0$, $\rho=\rho_0$
and $\delta=\delta_0$ for the new initial objects.

Through induction, we can repeat the iterative process an infinite number of times.
To clarify, we have:
\[
\norm{E_s}_{\rho_s} < a_1^{-4 \tau s} \norm{E}_{\rho} \longrightarrow  0\
\quad \mbox{when} \quad s\rightarrow \infty\,.
\]
The iterative scheme leads to the convergence of a true quasi-periodic torus, denoted by $K_\infty$.
As a result of the output of Lemma \ref{lem:KAM:inter:integral}, this object
satisfies
$K_\infty \in \Anal(\TT^{n}_{\rho_\infty} \times \TT^{\ell}_{\rho_\infty} )$ and
\[
\norm{\Dif K_\infty}_{\rho_\infty} < \sigmaDK\,,
\qquad
\norm{(\Dif K_\infty)^\top}_{\rho_\infty} < \sigmaDKT\,,
\qquad
\dist(K_\infty(\TT^{n} \times \TT^{\ell} ),\pd (\B\times\TT^\ell_\rho)) > 0\,.
\]
To control the limiting objects, we repeat the computations in  \eqref{eq:check:dist}  as follows
\begin{align}
&\norm{K_\infty - K}_{\rho_\infty} \leq {} \sum_{j=0}^\infty
\norm{K_{j+1} - K_j}_{\rho_{j+1}} < \frac{\CDeltaK \norm{E}_{\rho}}{\gamma^2 \delta^{2\tau}}
\frac{1}{1-a_1^{-2\tau}} =:  
\frac{\mathfrak{C}_2 \norm{E}_{\rho}}{\gamma^2 \rho^{2\tau}} \label{eq:Cmathfrak2}\,,
\end{align}
thus obtaining the estimates in \eqref{eq:close}.
This completes the proof of the KAM theorem to Lagrangian Tori.
\end{proof}

\section{Acknowledgements}
We would like to express our gratitude to the following organizations for their support: 
CONACYT for the PhD fellowship, 
DGAPA-UNAM through project PAPIIT IN103423, IN104725 
and the support received from the project PID2021-125535NB-I00 (MCIU/AEI/FEDER, UE). 
Additionally, we acknowledge the funding received from the Severo Ochoa and María de Maeztu Program for Centers and Units of Excellence in R\&D (CEX2020-001084-M),
P.P. has been partially supported by the Spanish Government grants PID2019-104851GB-I00 (MICINN/FEDER,UE) and PID2021-125535NB-I00 (MCIU/AEI/FEDER, UE).
A.H. has been supported by the Spanish grant PID2021-
125535NB-I00 (MCIU/AEI/FEDER, UE) and the Severo
Ochoa and María de Maeztu Program for Centers and Units of Excellence in R\&D (CEX2020-001084-M).
We are also grateful to
A. Viero,
J.-Ll. Figueras, J.M. Mondelo and \'A. Fern\'andez-Mora for fruitful discussions.
We also express our gratitude to the Departments of Mathematics and Computer Science at the University of Barcelona, 
and to the Department of Mathematics at Uppsala University, 
for their hospitality during our stays at their institutions while carrying out this work.


\newcommand{\etalchar}[1]{$^{#1}$}

\break
\appendix

\section{Extended symplectic formulation for QP Hamiltonian}
\label{ssec:extended}
A different approach to handling time-dependent quasiperiodic Hamiltonian systems involves extending the symplectic structure by introducing auxiliary variables. 
This method effectively transforms the system into an autonomous form, 
allowing the use of standard symplectic techniques. 
By augmenting the phase space with fictitious variables, 
one can preserve the Hamiltonian framework while incorporating time dependence naturally. 
In this appendix, 
we outline the construction of this extended formulation and discuss its advantages in the context of KAM theory.

We consider a qp-time-dependent Hamiltonian system of the form
    \[
        \begin{array}{rcl}
            \H: \RR^{2n} \times \TT^{\ell} & \longrightarrow &\RR \,, \\
            (z, \varphi) & \longmapsto & \H(z,\varphi) \,,
        \end{array}
    \]
    is to extend the symplectic formalism by introducing fictitious variables, $\I$,
    associated with momentum, 
    which are conjugate to the quasi-periodic time variables, i.e.
\[
\begin{array}{rcl}
    \hat{\H}: \RR^{2n} \times \TT^{\ell} \times \RR^{\ell}& \longrightarrow &\RR \,, \\
    (z, \varphi, \I) & \longmapsto & \hat{\H}(z, \varphi, I) =  \H(z,\varphi) + \I \cdot \alpha \,,
\end{array}
\]
with $\alpha\in\RR^{\ell}$ external frequencies. 
The corresponding symplectic vector field is then given by:
\[
    Z_{\hat \H}(z,\varphi, \I)= \hat{\Omega}(z)^{-1} (\Dif\hat{\H}(z,\varphi,\I))^\top\,,
\]
with,
\[
    \hat{\Omega}(z)^{-1} =    
\begin{pmatrix}  
    \Omega(z)^{-1} 
  &
    O_{2n\times 2\ell}
  \\
    O_{2\ell \times 2n}
  &
    \Omega_{\ell}^{-1}
\end{pmatrix}\,,
\]
where $\Omega(z)$ denotes a $2n$ symplectic matrix, 
while $\Omega_{\ell}= \begin{pmatrix} 
    O_{\ell} & -I_{\ell} \\
    I_{\ell} & O_{\ell}
    \end{pmatrix}$ 
    is the $2\ell$ standar symplectic matrix. 
    Moreover, 
    the matrix representations of the Riemannian metric and the linear isomorphism are expressed as follows.
\[
    \hat{G}(z)=    
    \begin{pmatrix}  
        G(z) & O_{2n\times 2\ell} \\
        O_{2\ell \times 2n} & I_{2\ell}
\end{pmatrix}\,, \quad
    \hat{J}(z)=\begin{pmatrix}  
        J(z) & O_{2n\times 2\ell} \\
        O_{2\ell \times 2n} & \Omega_\ell
\end{pmatrix}\,.
\]
    Analogous to \cite{HaroCFLM16,haro2019posteriori}, 
    we could to look a parameterization $\hat K: \TT^n\times\TT^\ell  \to \mani \times \TT^\ell \times \RR^\ell$
\[
    \hat{K}(\theta, \varphi) = 
    \begin{pmatrix}
        K(\theta, \varphi) \\
        \Phi(\theta,\varphi) \\ 
        \I(\theta, \varphi)   
    \end{pmatrix} \,.
\]
That satisfies the follow inifinitesimal invariance equation,
with internal dynamics is given by the constant vector field $(\omega, \alpha)$,  
\[
    Z_{\hat \H}(K(\theta, \varphi),\Phi(\theta, \varphi), \I(\theta, \varphi)) 
    =  \Dif(K(\theta, \varphi),\Phi(\theta, \varphi), \I(\theta, \varphi) ) (\omega, \alpha)^\top \,.
\]
Where we obtain the follow equations,
\begin{align}
    Z_\H(\theta,\varphi) =: \Omega(z)^{-1}(\Dif_z h(\theta, \varphi))^\top= & -\Lie{\omega,\alpha}K(\theta, \varphi) \,, \label{eq:inv-eq-K} \\
    \alpha = & -\Lie{\omega,\alpha}\Phi(\theta, \varphi) \,, \label{eq:inv-eq-P} \\
    -\left(\Dif_{\varphi}h(\theta, \varphi) \right)^\top = & -\Lie{\omega,\alpha}\I(\theta, \varphi) \,. \label{eq:inv-eq-I}
\end{align}
    By solving these equations, 
    we can obtain a parameterization of an invariant torus of dimension $2(n+\ell)$. 
    Since we assume that $\varphi \in \TT^\ell$,
    we obtain that,
    \[
        \Phi(\theta,\varphi) = \varphi\,,
    \]
    i.e.,
    \[
        \hat{K}(\theta, \varphi) = 
        \begin{pmatrix}
            K(\theta, \varphi) \\
            \varphi \\
            \I(\theta, \varphi)   
        \end{pmatrix} \,.
    \]
    Therefore, the tangent bundle, 
    $\hat L: \TT^n\times\TT^\ell  \to \RR^{2(n+\ell)\times(n+\ell)}$
    takes the following form,
    \[
        \hat{L}(\theta, \varphi) = 
        \begin{pmatrix}
            \Dif_\theta K(\theta, \varphi) & \Dif_\varphi K(\theta, \varphi)\\
            O_{\ell \times n} & I_\ell \\
            \Dif_\theta \I(\theta, \varphi) &  \Dif_\varphi \I(\theta, \varphi)
        \end{pmatrix} \,.
    \]
    In comparison with the approach in this paper, 
    by assuming $\varphi \in \TT^\ell$,
    therefore 
    We not only reduce the dimension of the torus, 
    from $2(n+\ell)$ to $2n$,
    also the number of equations to solve, 
    thereby lowering the computational cost of the resulting algorithms derived from the proof.

\section{Computation of a novel Torsion expression}\label{ssec:torsion}
In this appendix, 
we derive an expression for the torsion without involving Lie derivatives, 
explicitly obtaining equation \eqref{eq:newT} from \eqref{eq:T}. 

\begin{lemma}\label{lem:torsion}
Let $T: \TT^n \times \TT^\ell \to \RR^n$ be the torsion matrix, 
as defined in equation \eqref{eq:T}. 
Then, 
it satisfies the expression given in equation \eqref{eq:newT}.
\end{lemma}

\begin{proof}
Let us consider
\[
\tilde N(\te, \vp) = N(\te,\vp)^0B(\te, \vp) = J(K((\te,\vp))L(\te, \vp)B(\te, \vp)\,.
\]
Through the substitution of this expression into the equations,
\eqref{eq:N} and \eqref{eq:A}, 
we obtain:
\begin{equation}\label{eq:N-App}
N(\te, \vp) = L(\te,\vp)A(\te,\vp) + \tilde{N}(\te,\vp)\,,
\end{equation}
\begin{equation}\label{eq:A-App}
A(\te, \vp) = -\dfrac{1}{2}\tilde N(\te,\vp)^\top \Omega(K(\te,\vp),\vp)\tilde{N}(\te,\vp)\,.
\end{equation}
Now,
we calculate the following derivatives; 
$\Lie{\om, \al}L(\te, \vp)$, 
$\Lie{\om, \al}B(\te, \vp)$, 
$\Lie{\om, \al}\tilde N(\te, \vp)$ 
and $\Lie{\om, \al} A(\te, \vp)$. 
Since they are necessary to calculate $\Lie{\om, \al}N(\te, \vp)$.
For the first expression, 
it is sufficient to differentiate the invariance equation, 
\eqref{eqn:invariance_equation},
with respect to $\theta$, that is,
\begin{equation}\label{eq:LieDK}
    \Lie{\omega, \alpha}L(\theta,\varphi)=
 - \Dif_z Z_\H(K(\theta,\varphi) ,\varphi)L(\theta,\varphi)\,,
\end{equation}
For the second one, 
let us take the Lie derivative of $I_{d} = B^{-1}B$, 
we have obtain that 
$\Lie{\omega,\alpha} B 
= -B \Lie{\omega,\alpha} B^{-1} B 
= -B \Lie{\omega,\alpha} \left( L^\top (G \circ K)) L \right)B$, 
i.e,
\begin{equation}\label{eq:LieB}
\begin{split}
\Lie{\om,\al} B(\te,\vp)
= & B(\te,\vp)^\top L(\theta,\varphi)^\top \Big(\Omega( K(\theta,\varphi) ,\vp) \Dif_z Z_\H (K(\theta,\varphi) ,\varphi) \\
  &\qquad - \Omega(K(\te,\vp)) \Dif_z J( K(\theta,\varphi))\left[ Z_\H\left(K(\te, \vp),\vp \right) \right]  J(K(\te,\vp))^{-1} \\
  &\qquad - \Omega(K(\te,\vp)) J(K(\te,\vp)) \Dif_z Z_\H(K(\theta,\varphi) ,\varphi)  J(K(\te,\vp))^{-1} \Big) \tilde N(\theta,\varphi) \\
= & \tilde N(\te,\vp)^\top J(K(\te,\vp))^{-\top} T_\H( K(\theta,\varphi) ,\vp)) \tilde N(\te,\vp)\,.
\end{split}
\end{equation}
Equations \eqref{eq:LieDK}, 
\eqref{eqn:Jacobi-indentity} and
\begin{equation*}
\begin{split}
T_h(z,\vp)= \Omega(z) \Big(& \Dif_z Z_\H(z, \vp) 
    -  \Dif_z J(z) \left[z, \vp) \right] J(z)^{-1}
    - J(z) \Dif_z Z_\H(z, \vp) J(z)^{-1} \Big) \,,
\end{split}
\end{equation*}
were used for this calculation.
For the next one, 
we consider, the definition of $\tilde N$  and the equations \eqref{eq:LieDK} and \eqref{eq:LieB}.
\begin{equation}\label{eq:LieTildeN}
\begin{split}
\Lie{\om,\al} \tilde N(\te,\vp)
    = &  -\Dif_z J(K(\te,\vp))\left[ Z_\H\left(K(\te, \vp),\vp \right) \right]L(\te,\vp) B(\te,\vp)\\ 
   & - J(K(\te,\vp))  \Dif_z Z_\H(K(\theta,\varphi) ,\varphi) L(\te,\vp)  B(\te,\vp) \\
    & + J(K(\te,\vp)) L(\te,\vp) \tilde N(\te,\vp)^\top J(K(\te,\vp))^{-\top} T_\H(K(\theta,\varphi),\varphi)) \tilde N(\te,\vp)\,.
\end{split}
\end{equation}
Therefore the Lie derivative of $A$, is:
\begin{equation}\label{eq:LieA}
\begin{split}
\Lie{\om,\al} & A(\te,\vp) \\
= & - \dfrac{1}{2}\Lie{\om,\al} \tilde N(\te,\vp)^\top  \Omega(K(\te,\vp))  \tilde N(\te,\vp)  \\ 
   & - \dfrac{1}{2}  \tilde N(\te,\vp)^\top  \Lie{\om,\al}  \Omega(K(\te,\vp))  \tilde N(\te,\vp) \\
   & - \dfrac{1}{2} \tilde N(\te,\vp)^\top  \Omega(K(\te,\vp)) \Lie{\om,\al} \tilde N(\te,\vp) \\
    = &  \dfrac{1}{2} \tilde N(\te,\vp)^\top J(K(\te,\vp))^{-\top} \Dif_z J(K(\te,\vp))\left[ Z_\H\left(K(\te, \vp),\vp \right) \right]^\top  \Omega(K(\te,\vp))  \tilde N(\te,\vp) \\
   & +\dfrac{1}{2} \tilde N(\te,\vp)^\top J(K(\te,\vp))^{-\top} (\Dif_z  Z_\H\left(K(\te, \vp), \vp \right) )^\top J(K(\te,\vp))^\top \Omega(K(\te,\vp)) \tilde N(\te,\vp) \\
    & - \dfrac{1}{2}  \tilde N(\te,\vp)^\top T_\H(K(\theta, \varphi),\vp))^\top J(K(\te,\vp))^{-1} \tilde N(\te,\vp) L(\te,\vp)^\top  J(K(\te,\vp))^\top  \Omega(K(\te,\vp)) \tilde N(\te,\vp) \\
    & - \dfrac{1}{2}  \tilde N(\te,\vp)^\top \left( \Dif_z Z_\H\left(K(\te, \vp),\vp \right) \right)^\top  \Omega(K(\te,\vp))N(\te,\vp) \\
    & -\dfrac{1}{2}  \tilde N(\te,\vp)^\top  T_\H(K(\theta, \varphi),\vp))   \tilde N(\te,\vp) \\
    & - \dfrac{1}{2}  \tilde N(\te,\vp)^\top  \Omega(K(\te,\vp)) J(K(\te,\vp)) L(\te,\vp) \tilde N(\te,\vp)^\top J(K(\te,\vp))^{-\top} T_\H(K(\theta,\varphi),\vp)) \tilde N(\te,\vp) \,,
\end{split}
\end{equation}
Using \eqref{eq:N-App}, \eqref{eq:LieDK},\eqref{eq:LieB}, \eqref{eq:LieTildeN}, \eqref{eq:LieA} in \eqref{eq:T}, 
we obtain
\begin{equation*}
\begin{split}
    T(K(\theta, \varphi),\vp) 
    = & \dfrac{1}{2}\tilde N(\te,\vp)^\top \left(T_h(K(\theta,\varphi),\vp) +  T_h(K(\theta,\varphi),\vp)^\top\right)\tilde N(\te,\vp) \\
& +
    \tilde N(\te,\vp)^\top T_h(K(\theta,\varphi),\vp)^\top L(\te,\vp) A(\te,\vp) \\ 
    & - A(\te,\vp)L(\te,\vp)^\top T_h(K(\theta, \varphi),\vp)\tilde N(\te,\vp) \,.
\end{split}
\end{equation*}

\end{proof}

\section{Compendium of constants in the KAM theorem}\label{ssec:consts}

In this appendix we collect the recipes to compute all constants involved in the different estimates presented in the paper. 
Keeping track of these constants is crucial to apply the KAM theorem in particular problems and for concrete values of parameters. 
Let us remark that, 
as it becomes clear in the proof, 
the numbersv$a_1$, $a_2$ and $a_3$ are independent parameters, 
that can be selected in order to optimize the applicability of the theorems
depending on the particular problem at hand.
Table \ref{tab:constants:all} corresponds to the geometric construction of the theorem. 
The constants needed for a Newton step are outlined in Table \ref{tab:constants:all:2},  
while Table \ref{tab:constants:all:3}  contains the constants related to the method's convergence.

\bgroup
{\scriptsize
\def\arraystretch{1.5}
\begin{longtable}{|l l l l|}
\caption{Constants introduced in Section \ref{sec:lemmas}. 
    Constants with $^*$ in the label correspond to {\bf Case III}. }
\label{tab:constants:all} \\
\hline
Object & Constant & Label & Result\\
\hline
\hline
$\Lie{\omega}\Omega_L$ &
$\CLieOmegaL = 2 n \cteOmega \sigmaDK+ \sigmaDKT \cteDOmega \sigmaDK \delta + n\sigmaDKT \cteOmega$ & 
\eqref{eq:CLieOK} & Lemma \ref{lem:isotrop}\\
\hline
$\Omega_L$ & $\COmegaL = c_R \CLieOmegaL$ & \eqref{eq:COK} & Lemma \ref{lem:isotrop}\\
\hline
$L$ & $\CL =  \sigmaDK$ & \eqref{eq:CL} & Lemma \ref{lem:Lang} \\
\hline
$L^\top$ & $\CLT = \sigmaDKT$ & \eqref{eq:CLT} & Lemma \ref{lem:Lang} \\
\hline
$G_L$ & $\CGL =  \CLT \cteG \CL$ & \eqref{eq:CGL} & Lemma \ref{lem:Lang} \\
\hline
$N^0$ & $\CNO = \cteJ \CL$ & \eqref{eq:CNO} & Lemma \ref{lem:sympl} \\
\hline
$(N^0)^\top$ & $\CNOT = \CLT \cteJT$ & \eqref{eq:CNOT} & Lemma \ref{lem:sympl} \\
\hline
    $\tilde N$ & $\CTildeN = \CNO \sigmaB$  & \eqref{eq:CtildeN} & Lemma \ref{lem:sympl} \\
\hline
$\tilde N^\top$ & $\CTildeNT = \sigmaB \CNOT$ & \eqref{eq:CtildeNT} & Lemma \ref{lem:sympl} \\
\hline
$A$ & $\CA = \tfrac{1}{2} \CTildeNT \cteOmega\CTildeN$ \,, $ ^*\CA = 0 $ 
 & \eqref{eq:CA}$^*$ & Lemma \ref{lem:sympl} \\
\hline
    $N$ & $\CN = \CL \CA + \CTildeN $ \,, $ ^*\CN = \CTildeN $  & \eqref{eq:CN} & Lemma \ref{lem:sympl} \\
\hline
$N^\top$ & $\CNT = \CA \CLT + \CTildeNT$  \,, $ ^*\CNT = \CTildeNT$ & \eqref{eq:CNT} & Lemma \ref{lem:sympl} \\
\hline
$\Esym$ & 
    $\Csym =  (1 + \CA) \max\{1,\CA \} \, \Clag$ \,, $ ^* \Csym =  \Clag \max\{1,(\sigmaB)^2\}$
  & \eqref{eq:Csym} & Lemma \ref{lem:sympl} \\
\hline
$T$ & {$\!\begin{aligned} 
\CT = 
{} &
\tfrac{1}{2}(\CTildeNT\cteTH\CTildeN + \CTildeNT\cteTHT\CTildeN)  \\
& + \CTildeNT \cteTHT \CN + \CNT \cteTH  \CTildeN \\
^*\CT= {} & \CNT\cteTH\CN
\end{aligned}$}
$\vphantom{
\left\{\begin{array}{l}
0 \\
0 
\end{array}\right.}$
& \eqref{eq:CnewT} & Lemma \ref{lem:Newtwist} \\
\hline
$T_E$ & {$\!\begin{aligned} 
\CTE = 
{} &
\cteDOmega \de + \cteOmega \cteDJ \cteJInv  \de  \\ 
& +  2n \cteJT \CL \sigmaB  \cteOmega^2  +  n \cteJT \sigmaB \CLT  \cteOmega^2 
\end{aligned}$}
$\vphantom{
\left\{\begin{array}{l}
0 \\
0 
\end{array}\right.}$
& \eqref{eq:CTE} & Lemma \ref{lem:constT} \\
\hline
$T_E^\top$ & {$\!\begin{aligned} 
\CTET = 
{} &
\cteDOmega \de +  \cteJInvT \cteDJT  \cteOmega \de  \\ 
& +  n  \sigmaB \CLT \cteJ \cteOmega^2  +  2n  \CL \sigmaB \cteJ\cteOmega^2 
\end{aligned}$}
$\vphantom{
\left\{\begin{array}{l}
0 \\
0 
\end{array}\right.}$
& \eqref{eq:CTET} & Lemma \ref{lem:constT} \\
\hline
$E_{\Lie{} B}$ & $\CELieB =\CTildeNT \cteJInvT \CTE \CTildeN$ & \eqref{eq:CELieB} & Lemma \ref{lem:constT} \\
\hline
$E_{\Lie{} B}^\top$ & $\CELieBT = \CTildeNT \CTET \cteJInv  \CTildeN$ & \eqref{eq:CELieBT} &Lemma \ref{lem:constT} \\
\hline
$E_{\Lie{} \tilde N}$ & $ \CELietildeN = \de \cteDJ  \CL  \sigmaB + n \cteJ \sigmaB +\cteJ  \CL  \CELieB$ & \eqref{eq:CELieTildeN} & Lemma \ref{lem:constT} \\
\hline
$E_{\Lie{}\tilde N}^\top$ & $ \CELietildeNT = \de \sigmaB \CLT \cteDJT  + 2n \sigmaB \cteJT  + \CELieBT \CLT \cteJT$ & \eqref{eq:CELieTildeNT} &Lemma \ref{lem:constT} \\
\hline
$E_{\Lie{}A}$ & $ \CELieA = \tfrac{1}{2} (\CELietildeNT \cteOmega \CTildeN + \de\CTildeNT\cteDOmega\CTildeN + \CTildeNT\cteOmega \CELietildeN )$ & \eqref{eq:CELieA} &Lemma \ref{lem:constT} \\
\hline
$E_{\Lie{}N}$ & $ \CELieN = \CELietildeN + \CL \CELieA + n \CA$ & \eqref{eq:CELieN} &Lemma \ref{lem:constT} \\
\hline
$E_T$ & 
{$\!\begin{aligned} 
\CET=
{} &
 \tfrac{1}{2}(\CA \COmegaL \CTildeNT \cteTHT  \CTildeN + \CA \COmegaL \CTildeNT  \cteTH  \CTildeN)  \\
& + \CA \COmegaL \CA  \cteTH  \CTildeN + \CA \COmegaL \CTildeNT  \cteTHT  \CTildeN \CL \CA \\
& + \gamma \delta^\tau \CA \CLT \cteOmega \CELieN + \gamma \delta^\tau \CTildeNT  \cteOmega \CELieN
\end{aligned}$}
$\vphantom{
\left\{\begin{array}{l}
0 \\
0 \\
0
\end{array}\right.}$
& \eqref{eq:CET} & Lemma \ref{lem:constT} \\
\hline
$\Loper{L}$ & $\CLoperL = n$ & \eqref{eq:CLoperL} & Lemma \ref{lem:reduc}\\
\hline
$\Loper{L}^\top$ & $\CLoperLT = 2n $ & \eqref{eq:CLoperLT} & Lemma  \ref{lem:reduc} \\
\hline
$\Lie{}\bar A$ & 
{$\!\begin{aligned} 
\CLieA=
{} &
 \tfrac{1}{2}(\CTildeNT \cteTHT \CTildeN + \CTildeNT \cteTHT \cteJInv \CTildeN \CLT \cteJT \cteOmega \CTildeN \\
& + \CTildeNT \cteTH \CTildeN 
    + \left.\CTildeNT \cteOmega \cteJ \CL \CTildeNT \cteJInvT \cteTH \CTildeN\right) \\
    ^*\CLieA= {} & 0    
\end{aligned}$
}
$\vphantom{
\left\{\begin{array}{l}
0 \\
0 \\
0 \\
0 
\end{array}\right.}$
& \eqref{eq:CLieA} & Lemma \ref{lem:reduc} \\
\hline
$\Ered^{1,1}$ & $\Creduu = \CNT \cteOmega \CLoperL$ & \eqref{eq:Creduu} & Lemma \ref{lem:reduc} \\
\hline
$\Ered^{1,2}$ & $\Credud = \CET$ & \eqref{eq:Credud} & Lemma \ref{lem:reduc} \\
\hline
$\Ered^{2,1}$ & $\Creddu = \CLT \cteOmega \CLoperL$ & \eqref{eq:Creddu} & Lemma \ref{lem:reduc} \\
\hline
$\Ered^{2,2}$ & 
{$\!\begin{aligned} 
\Creddd =
{} &
 (\CLT \cteDOmega \CN \delta + \CLoperLT \cteOmega \CN + \CLieOmegaL \CA +\sigmaDKT\cteOmega\sigmaDK\CELieA) \gamma \delta^\tau \\
& + \Clag \CLieA
\end{aligned}$}
$\vphantom{
\left\{\begin{array}{l}
0 \\
0 
\end{array}\right.}$
& \eqref{eq:Creddd} & Lemma \ref{lem:reduc} \\
\hline
$\Ered$ &
$\Cred =
\max \{
\Creduu \gamma \delta^{\tau} + \Credud
\, , \,
\Creddu \gamma \delta^{\tau} + \Creddd
\}$
& \eqref{eq:CEred} & Lemma \ref{lem:reduc} \\
\hline
\end{longtable}}
\egroup

\bgroup
{\tiny
\def\arraystretch{1.5}
\begin{longtable}{|l l l|}
\caption{\scriptsize Constants introduced in Lemma \ref{lem:KAM:inter:integral}.
Given an object $Z_\H$, we use the notation $\Delta \LieO Z_\H = \Lie{\omega,\alpha} \bar Z_\H-\Lie{\omega,\alpha} Z_\H$.
Constants with $^*$ in the label are $0$ in {\bf Case III}. 
} \label{tab:constants:all:2} \\
\hline
Object & Constant & Label \\
\hline
\hline
$\xi_0^N$ & $\CxiNO = \sigmaT(\CNT \cteOmega \gamma \delta^\tau + c_R \CT \CLT \cteOmega)$
& \eqref{eq:CxiNO} \\
\hline
$\xi^N$ &
$\CxiN = \CxiNO + c_R \CLT \cteOmega$ & \eqref{eq:CxiN} \\ 
\hline
$\xi^L$ &
$\CxiL = c_R (\CNT \cteOmega \gamma \delta^\tau + \CT \CxiN)$ & \eqref{eq:CxiL} \\
\hline
$\xi$ & 
$\Cxi = \max \{\CxiL \, , \, \CxiN \gamma \delta^\tau\}$ & \eqref{eq:Cxi} \\
\hline
$\DeltaK$ &
$\CDeltaK = \CL \CxiL + \CN \CxiN \gamma \delta^\tau$ & \eqref{eq:CDeltaK} \\
\hline
$\Lie{\omega} \xi^N$ &
$\CLiexiN=  \CLT \cteOmega$
& \eqref{eq:CLiexiN} \\
\hline
$\Lie{\omega} \xi^L$ &
$\CLiexiL= \CNT \cteOmega \gamma \delta^\tau + \CT \CxiN$
& \eqref{eq:CLiexiL} \\
\hline
$\Lie{\omega} \xi$ &
$\CLiexi = \max \{ \CLiexiL \, , \, \CLiexiN \gamma \delta^\tau\}$
& \eqref{eq:CLiexi} \\
\hline
$\Elin$ &
$\Clin = \Cred \Cxi + \cteOmega \Csym \CLiexi \gamma \delta^\tau$ &
\eqref{eq:Clin} \\
\hline
$\bar E$ &
$\CE = 2(\CL +\CN) \Clin \gamma \delta^{\tau-1} + \tfrac{1}{2} \cteDDXH (\CDeltaK)^2$ & \eqref{eq:CE} \\
\hline
$\DeltaL$ &
$\CDeltaL = n \CDeltaK$ &
\eqref{eq:CDeltaL} \\
\hline
$\DeltaLT$ &
$\CDeltaLT = 2n\CDeltaK$ &
\eqref{eq:CDeltaLT} \\
\hline
$\DeltaG$ & $\CDeltaG = \cteDG \CDeltaK$ & \eqref{eq:CDeltaG} \\
\hline
$\DeltaGL$ &
$\CDeltaGL = \CLT \cteG \CDeltaL + \CLT \CDeltaG \CL \delta + \CDeltaLT \cteG \CL$ & \eqref{eq:CDeltaGL} \\
\hline
$\DeltaB$ &
$\CDeltaB = (\sigmaB)^2 \CDeltaGL$ & \eqref{eq:CDeltaB} \\
\hline
$\DeltaOmega$ &
$\CDeltaOmega = \cteDOmega \CDeltaK$ &
\eqref{eq:CDeltaOmega} \\
\hline
$\DeltaJ$ &
$\CDeltaJ = \cteDJ \CDeltaK$ &
\eqref{eq:CDeltaJ} \\
\hline
$\Delta J^\top$ & $\CDeltaJT = \cteDJT \CDeltaK$ & 
\eqref{eq:CDeltaJT} \\
\hline
$\DeltaNO$ &
$\CDeltaNO = \cteJ \CDeltaL + \CDeltaJ \CL \delta$
& \eqref{eq:CDeltaNO} \\
\hline
$\DeltaNOT$ &
$\CDeltaNOT = \CDeltaLT \cteJT + \CLT \CDeltaJT \delta$ & \eqref{eq:CDeltaNOT} \\
\hline
$\DeltaTildeN$ &
$\CDeltaTildeN = \CDeltaNO \sigmaB  + \CNO \CDeltaB$ & \eqref{eq:CDeltaTildeN} \\
\hline
$\DeltaTildeNT$ &
$\CDeltaTildeNT = \sigmaB \CDeltaNOT   + \CDeltaB \CNOT $ & \eqref{eq:CDeltaTildeNT} \\
\hline
$\DeltaA$ &
$\CDeltaA = \dfrac{1}{2} \Big( 
 \CDeltaTildeNT \cteOmega \CTildeN + \CTildeNT\CDeltaOmega\CTildeN\delta + \CTildeNT\cteOmega\CDeltaTildeN \Big) $
& \eqref{eq:CDeltaA} \\
& $^* \CDeltaA = 0$ & \\
\hline
$\DeltaN$ &
$\CDeltaN = \CA \CDeltaLT + \CDeltaA \CLT + \CDeltaTildeN$ & \eqref{eq:CDeltaN} \\
\hline
$\DeltaNT$ &
$\CDeltaNT = \CA \CDeltaLT + \CDeltaA \CLT + \CDeltaTildeNT$ & \eqref{eq:CDeltaNT} \\
\hline
$\DeltaTH$ &
$\CDeltaTH = \cteDTH \CDeltaK$ & \eqref{eq:CDeltaTH} \\
\hline
$\DeltaTHT$ &
$\CDeltaTHT = \cteDTHT \CDeltaK$ & \eqref{eq:CDeltaTHT} \\
\hline
$\DeltaT$ & {$\!\begin{aligned} 
\CDeltaT = {} &
\tfrac{1}{2}\CDeltaTildeNT \cteTH \CTildeN + \CTildeNT \CDeltaTH \CTildeN \delta + \CTildeNT \cteTH \CDeltaTildeN \\
& + \tfrac{1}{2} \CDeltaTildeNT \cteTHT \CTildeN + \CTildeNT \CDeltaTHT \CTildeN \delta + \CTildeNT \cteTHT \CDeltaTildeN \\
& + \CDeltaTildeNT \cteTHT \CN + \CTildeNT \CDeltaTHT \CN \delta + \CTildeNT \cteTHT \CDeltaN \\
& + \CDeltaNT \cteTH \CTildeN + \CNT \CDeltaTH \CTildeN \delta + \CNT \cteTH \CDeltaTildeN
\end{aligned}$}
$\vphantom{
\left\{\begin{array}{l}
0 \\
0 \\
0 \\
0
\end{array}\right.}$
& \eqref{eq:CDeltaT} \\
\hline
$\DeltaTOI$ &
$\CDeltaTOI = (\sigmaT)^2 \CDeltaT$ & \eqref{eq:CDeltaTOI} \\
\hline
\end{longtable}}
\egroup


\bgroup
{\scriptsize
\def\arraystretch{1.5}
\begin{longtable}{|l l|}
\caption{Constants introduced in Section \ref{ssec:proof:KAM} associated
with the convergence of que quasi-Newton method that yields
Theorem \ref{theo:KAM}, i.e., the ordinary formulation.} \label{tab:constants:all:3} \\
\hline
Constant & Label \\
\hline
\hline
$\mathfrak{C}_1 = \max
\left\{
(a_1 a_3)^{4\tau} \CE \, , \, (a_3)^{2\tau+1} \gamma^2 \rho^{2\tau-1} \CDeltatot
\right\}$
&
\eqref{eq:cte:mathfrakC1} \\
\hline
$\CDeltaI = 
\max
\Bigg\{
\frac{n \CDeltaK}{\sigmaDK-\norm{\Dif K}_{\rho}} \, , \,
\frac{2n \CDeltaK}{\sigmaDKT-\norm{\Dif K^\top}_{\rho}} \, , \, 
\frac{\CDeltaB}{\sigmaB-\norm{B}_{\rho}} \, , \,
\frac{\CDeltaTOI}{\sigmaT-\abs{\aver{T}^{-1}}}
\Bigg\}$
& \eqref{eq:CDeltaI} \\
\hline
$\CDeltaII = \frac{\CDeltaK \delta}{\dist(K(\TT^d_{\rho}),\partial \B)}$
& \eqref{eq:CDeltaII} \\
\hline
$\CDeltatot = \max
    \Bigg\{ 
\Csym \gamma \delta^\tau \, , \,
\frac{\CDeltaI}{1-a_1^{1-2\tau}} \, , \, 
\frac{\CDeltaII}{1-a_1^{-2\tau}}
\Bigg\}$
& \eqref{eq:CDeltatot} \\
\hline
$\mathfrak{C}_2 = a_3^{2\tau} \CDeltaK/(1-a_1^{-2\tau})$ &
\eqref{eq:Cmathfrak2} \\
\hline
\end{longtable}}
\egroup

\end{document}